\documentclass[a4paper,oneside,10pt]{amsart}

\usepackage[a4paper]{geometry}
\usepackage{subcaption}
\usepackage[font=scriptsize]{caption}
\usepackage{bbm}
\usepackage{graphicx}
\usepackage{amssymb}
\usepackage{float}
\usepackage{array}
\usepackage{diagbox}
\usepackage{verbatim}

\usepackage[table]{xcolor}

\usepackage[]{hyperref}
\definecolor{hyperrefLinkColor}{rgb}{0.1,0.1,1.0}
\definecolor{hyperrefCiteColor}{rgb}{0.75,0,0}
\definecolor{hyperrefUrlColor}{rgb}{0,0,0.8}
\definecolor{hyperrefFileColor}{rgb}{0.8,0.1,0.8}
\hypersetup{
  pdftitle={Minimal surfaces based on fundamental pentagons},
  pdfauthor={A. Bobenko, S. Heller, N. Schmitt},
  colorlinks=true,
  urlcolor=hyperrefUrlColor,
  filecolor=hyperrefFileColor,
  linkcolor=hyperrefLinkColor,
  citecolor=hyperrefCiteColor
}
\usepackage{cleveref}

\numberwithin{equation}{section}

\crefname{figure}{figure}{figures}

\newcommand{\MatrixGroup}[1]{{\mathrm{#1}}}

\newcommand{\matSL}[2]{\MatrixGroup{SL}_{#1}{#2}}

\newcommand{\matSU}[2]{\MatrixGroup{SU}_{#1}{#2}}

\makeatletter

\def\part{\@startsection{part}{0}%
  \z@{\linespacing\@plus\linespacing}{.5\linespacing}%
  {\normalfont\Large\bfseries\raggedright}}

\def\section{\@startsection{section}{1}%
  \z@{.7\linespacing\@plus\linespacing}{.5\linespacing}%
  {\normalfont\large\bfseries\centering}}

\makeatother

\newcommand{\transpose}[1]{{{#1}^{\top}}}

\newcommand{\abs}[1]{{\lvert#1\rvert}}

\newcommand{\half}{\tfrac{1}{2}}

\newcommand{\fourth}{\tfrac{1}{4}}

\newcommand{\ol}[1]{\overline{#1}}

\DeclareMathSymbol{\varnothing}{\mathord}{AMSb}{"3F}

\DeclareMathOperator{\cross}{\times}

\DeclareMathOperator{\del}{\partial\!}
\DeclareMathOperator{\diag}{diag}

\DeclareMathOperator{\id}{\mathbbm{1}}

\DeclareMathOperator*{\res}{res}

\DeclareMathOperator{\suchthat}{|}

\DeclareMathOperator{\tr}{tr}

\newcommand{\bbC}{\mathbb{C}}
\newcommand{\bbD}{\mathbb{D}}

\newcommand{\bbN}{\mathbb{N}}

\newcommand{\bbR}{\mathbb{R}}
\newcommand{\bbS}{\mathbb{S}}

\newcommand{\bbZ}{\mathbb{Z}}

\newcommand{\calU}{\mathcal{U}}

\newcommand{\CPone}{\bbC\mathrm{P}^1}

\newcommand{\deriv}{\mathrm{d}}

\newcommand*{\coloneq}{
  \mathrel{\vcenter{\baselineskip0.5ex \lineskiplimit0pt
      \hbox{\scriptsize.}\hbox{\scriptsize.}}}%
  =}

\newcommand{\spacecomma}{\,\,\,,}
\newcommand{\spaceperiod}{\,\,\,.}

\theoremstyle{plain}

\newtheorem{theorem}{Theorem}[section]
\newtheorem*{theorem*}{Theorem}

\newtheorem{lemma}[theorem]{Lemma}
\newtheorem*{lemma*}{Lemma}

\newtheorem{corollary}[theorem]{Corollary}
\newtheorem*{corollary*}{Corollary}

\newtheorem{proposition}[theorem]{Proposition}
\newtheorem*{proposition*}{Proposition}

\newtheorem*{example*}{Example}

\newtheorem*{conjecture*}{Conjecture}

\theoremstyle{definition}

\newtheorem{definition}[theorem]{Definition}
\newtheorem*{definition*}{Definition}

\newtheorem*{notation*}{Notation}

\newtheorem{remark}[theorem]{Remark}
\newtheorem*{remark*}{Remark}

\theoremstyle{plain}

\newcommand{\theoremname}[1]{}

\newcommand{\gauge}[2]{{{#1}{.}{#2}}}

\newcommand{\NEG}{\text{--}}

\newenvironment{smatrix}{\bigl[\begin{smallmatrix}}{\end{smallmatrix}\bigr]}

\newcommand{\Asurface}[2]{\mathrm{A}_{#1,#2}}
\newcommand{\Bsurface}[2]{\mathrm{B}_{#1,#2}}

\mathchardef\mhyphen="2D

\newcommand{\Ztwo}{\bbZ_2}
\newcommand{\tet}{\textnormal{tet}}
\newcommand{\oct}{\textnormal{oct}}
\newcommand{\ico}{\textnormal{ico}}
\newcommand{\fivecell}{\textnormal{$5$-cell}}
\newcommand{\demitesseract}{\textnormal{demitesseract}}
\newcommand{\sixteencell}{\textnormal{$16$-cell}}
\newcommand{\twentyfourcell}{\textnormal{$24$-cell}}
\newcommand{\sixhundredcell}{\textnormal{$600$-cell}}

\newenvironment{statictable}[1]%
{ 
  \par\noindent
  \begin{minipage}{\textwidth}
    \centering
    #1
}%
{ 
  \end{minipage}
}%

\DeclareMathOperator{\Id}{Id}

\newcommand{\C}{\mathbb{C}}
\newcommand{\N}{\mathbb{N}}

\newcommand{\CP}{\mathbb{CP}}

\DeclareMathOperator{\genus}{\mathrm{genus}}

\newcommand{\smalldiagram}[1]{%
  \raisebox{-0.2cm}{\fontsize{5}{6}%
  \def\svgwidth{0.8cm}%
  \input{#1_svg-tex.pdf_tex}%
}}

\usepackage{import}

\title[Minimal reflection surfaces in $\mathbb S^3$.]{Minimal reflection surfaces in $\mathbb S^3.$\\ \small
Combinatorics of curvature lines and minimal surfaces based on fundamental pentagons.}

\author{Alexander I. Bobenko}
\address{Institut f\"ur Mathematik, TU Berlin,
  Str. des 17. Juni 136, 10623 Berlin, Germany}
\email{bobenko@math.tu-berlin.de}

\author{Sebastian Heller}
\address{Beijing Institute of Mathematical Sciences and Applications,
Beijing, China}
\email{sheller@bimsa.cn}

\author{Nick Schmitt}
\address{Institut f\"ur Mathematik, TU Berlin,
  Str. des 17. Juni 136, 10623 Berlin, Germany}
\email{schmitt@math.tu-berlin.de}

\begin{document}

\begin{abstract}
We study compact minimal surfaces in the 3-sphere which are constructed by successive reflections from a minimal $n$-gon --- so-called minimal reflection surfaces. The minimal $n$-gon solves a free boundary problem in a fundamental piece of the respective reflection group. We investigate the combinatorics of the curvature lines of reflection surfaces, and construct new examples of minimal reflection surfaces based on pentagons.
We end the paper by discussing the area of these minimal surfaces.
\end{abstract}

\maketitle

\section*{Introduction}
\typeout{== figure/introf.tex ============================================}\begin{figure}[b]
  \includegraphics[width=0.45\textwidth]{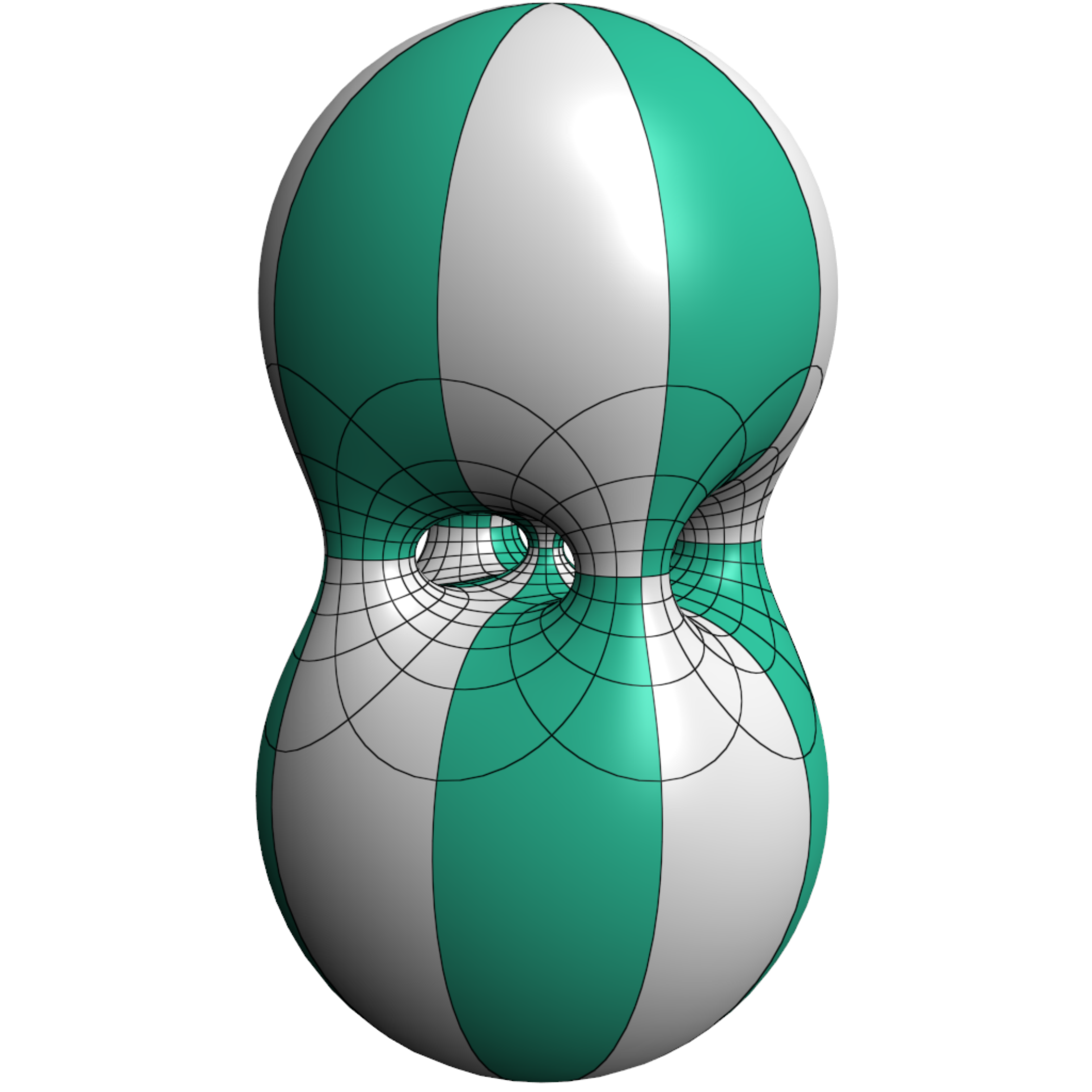}
  \includegraphics[width=0.45\textwidth]{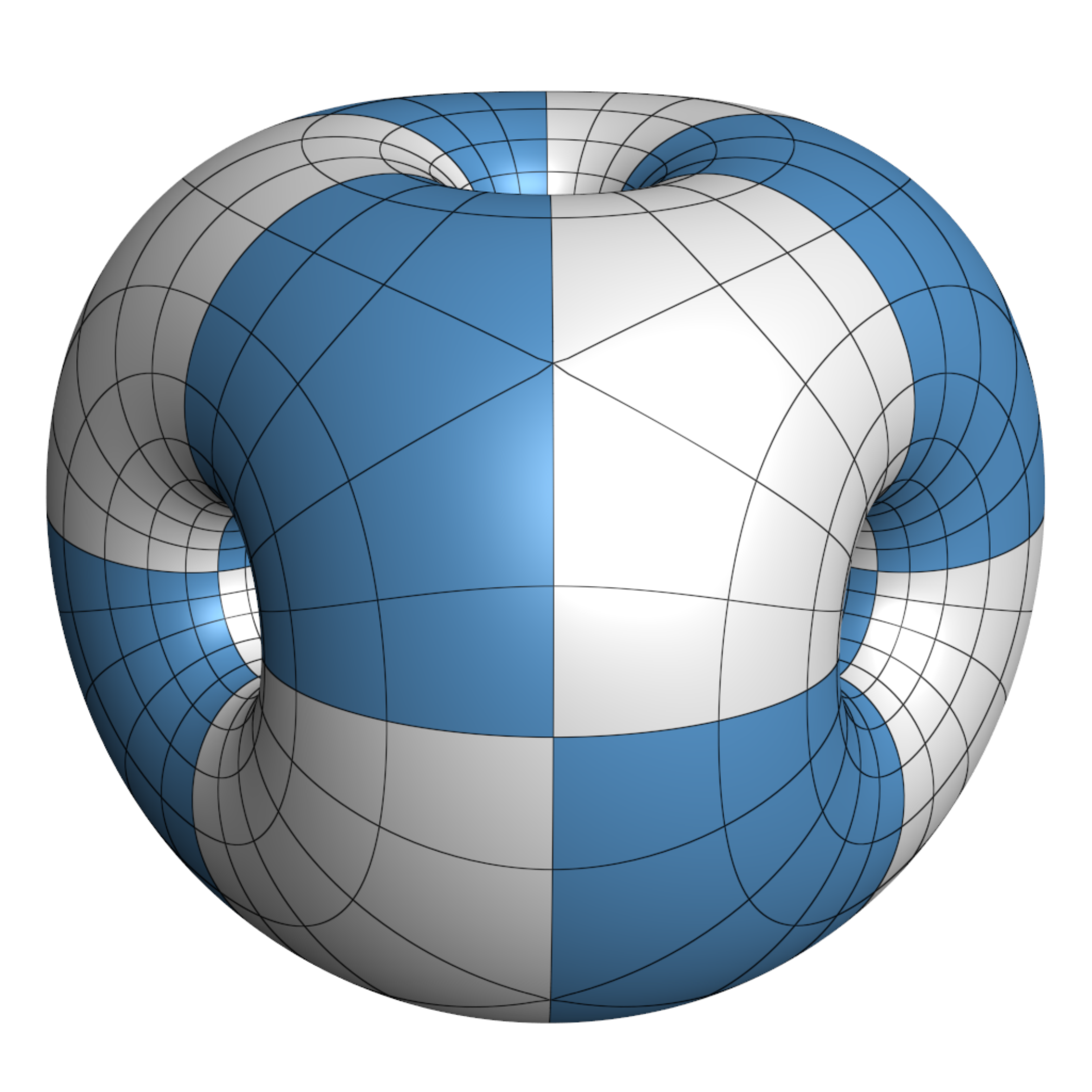}
  \caption{Two new minimal surfaces of genus $4$.}
\end{figure}

In 1978, Lawson \cite{Lawson_1970} constructed 
 a 2-integer family of embedded minimal surfaces $\xi_{k,l}$ in the 3-sphere, by using the Plateau solutions for appropriate
geodesic 4-gons. 
After only a few examples of compact minimal surfaces in the 3-sphere were known for a long time, new methods of construction have been developed in recent years. In particular, Kapouleas et al 
applied gluing methods for doubling Clifford tori and spheres \cite{Kapouleas_2017,Kapouleas_McGrath_2019}, the min-max theory of Marques and Neves \cite{MN}
has been used to produce many
equivariant minimal surfaces \cite{Ket}, 
and very recently 
work using Laplace and Steklov eigenvalues
optimization  in the presence of a discrete symmetry group has produced further examples \cite{KKMcGS_24}. 
Different methods often provide the same examples, but  it is not trivial to distinguish between them even in specific
cases \cite{KW}.
In addition to these analytical methods, the more algebro-geometric integrable systems methods have also been used in recent years to study minimal surfaces and their properties, e.g. \cite{He3,HHS, Heller_Heller_Traizet_2023}.
Besides detailed experimental investigations such as in \cite{Bobenko_Heller_Schmitt_2021}, these methods also provide the possibility of obtaining precise statements about geometric quantities like the area \cite{Heller_Heller_Traizet_2023,HHT3}.

 The Lawson surfaces can also been constructed
 by first solving a free boundary problem for minimal surfaces contained in a certain tetrahedron, and then building 
 up a closed surface through repetitive reflections  across the geodesic faces of the tetrahedron.
 Of course, these tetrahedra must be fundamental pieces of a finite reflection group -- namely  $D_m\times D_n$ in the case of the Lawson surfaces -- 
 in order to obtain a closed surface.
 This method has been first applied by Karcher-Pinkall-Sterling \cite{Karcher_Pinkall_Sterling_1988}, when they constructed new minimal surfaces in $\bbS^3$ built out of
 free boundary minimal surfaces in tetrahedrons which themselves are fundamental domains for (some of) the exceptional
 finite subgroups of $\mathrm{O}(4),$ see Table \ref{t:ref-group} below.
In both cases, the minimal surface $S$ has a fundamental piece $P$ which is a 4-gon, or, equivalently, there is a finite reflection group $G$ with an order two subgroup $\Gamma$ consisting of orientation preserving symmetries
such that $S\to S/\Gamma=\CPone$ is branched over exactly 4 points.
 
 In \cite{Bobenko_Heller_Schmitt_2021} we have studied minimal and constant mean curvature (CMC) surfaces
 in $\bbS^3$ and $\bbR^3$ which are based on such fundamental quadrilaterals. In particular, we 
 constructed some minimal surfaces in $\bbS^3$ of KPS-type which were missing in their original work \cite{Karcher_Pinkall_Sterling_1988}, e.g., a surface of genus 29 and a new surface of genus 11 with octahedral symmetry.
It became clear that one should also construct surfaces with $n$-gons as fundamental pieces for $n\ge 5$ and that one has to develop tools to distinguish the different surfaces. (It should be mentioned that, apart from the round sphere, no compact minimal surfaces based on fundamental 3-gons are possible.) 
 For example, some of the minimal surfaces based on fundamental pentagons are actually  Lawson or KPS surfaces.
 Therefore, in the first part of the paper we study reflection surfaces (for details see Definition \ref{def:refsurf} below), together with the (combinatorial) properties of their curvature lines. It is shown in Theorem \ref{thm:closed-curvature-lines} that each compact reflection surface has closed curvature lines, and we develop tools to distinguish
 different surface classes, see  e.g. Theorem \ref{thm:dihedral-classification}. As a consequence, we obtain
 that most minimal reflection surfaces obtained from fundamental pentagons are neither the Lawson nor the KPS surfaces.

It is worth mentioning that the treatment of reflection surfaces based on the combinatorics of their curvature lines is similar to approaches in Discrete Differential Geometry (DDG). Curvature line parameterization is naturally used for structure preserving discretizations in DDG. In \cite{BHS_minimal} it was shown how the geometry of discrete minimal surfaces in Euclidean space is determined by the combinatorics of their curvature lines. The approach is based on the construction of a polyhedron with prescribed combinatorics whose edges are tangent to a sphere (Koebe polyhedron). The latter is interpreted as a discrete Gauss map of the corresponding surface. A similar method was recently developed in \cite{BHS_cmc} for the construction of discrete surfaces of constant mean curvature from orthogonal ring patterns in $\bbS^2$. The orthogonal ring patterns are also uniquely determined by their combinatorics and can be interpreted as discrete curvature lines.

%
 
 In the second part of the paper, we  
numerically construct a 2-integer family of embedded minimal surfaces with the same reflection symmetry groups as those of $\xi_{k,l}$ with a fundamental pentagon solving a free boundary problem, using the DPW construction \cite{Dorfmeister_Pedit_Wu_1998}.
After recalling the basics of integrable surface theory in the first part of Section \ref{sec:dpw}, we shortly explain
the  setup for our experiments (Section \ref{ssec:flow}). These are based on a flow on the space of so-called
reflection potentials, which is by now well-understood mathematically in the case of 4-gons, see \cite{HHS, HHT2}.
In Section \ref{sec:lawson-potential}, we explain in detail how to show the existence of reflection potentials for $n$-gons for any  $n\geq4$, which can then serve as initial conditions of our numerical flow. 
 
Based on our experiments we conjecture the following surfaces:
for each reflection group and each combinatorial way to inscribe a pentagon into a fundamental region there exists
a minimal reflection surfaces in $\bbS^3$ with the given combinatorics.

Some of these surfaces have already been proved to exists previously: 
Kapouleas doubling spheres instantiate minimal reflection surfaces for the dihedral families $B_{2,k}$ and  $B_{k,2}$  for $k$
sufficiently large \cite{Kapouleas_2017,Kapouleas_McGrath_2019}.

All images show the reflection surfaces in $\mathbb S^3$ after stereographic projection to $\bbR^3.$ A stereographic projection might break a symmetry of a surface, see for example Figure \ref{fig:b41}. 
All figures show the curvature lines of the surfaces, i.e., horizontal and vertical trajectories of the holomorphic quadratic Hopf differentials.

\section*{Acknowledgements}
The first author is partially supported by the DFG Collaborative
Research Center TRR 109 \emph{Discretization in Geometry and
Dynamics}.  The second  author thanks the TU Berlin and the DFG Collaborative
Research Center TRR 109 \emph{Discretization in Geometry and
Dynamics}  for excellent research conditions and financial support during research stays when parts of the work were carried out.  The second author is partially supported by the Beijing NSF {\em International Scientists Project}. The third author is supported by the DFG Collaborative
Research Center TRR 109 \emph{Discretization in Geometry and
Dynamics}.

\section{Reflection surfaces and combinatorics of curvature lines}

\subsection{Reflection groups of $\mathbb S^3$}
We consider finite subgroups $G\subset\mathrm{Iso}(\mathbb S^3)$ of the group of isometries of the 3-sphere generated by reflections across 
totally geodesic 2-spheres. These groups are called {\em reflection groups}.
Reflection groups and generalisations thereof in any dimension have been investigated and classified by  Coxeter \cite{coxeter}.  We also refer to \cite{Davies}.

In Table \ref{t:ref-group} below, we list all reflection groups faithfully acting on $\bbS^3$. 
The rank of the group indicates the number of generating reflection spheres. The diagram shows their intersections,
 indicating a fundamental polyhedron (or $r$-hedra) $R$ of the reflection group.

\begin{proposition}
  \label{prop:reflection-groups-s3}
  The reflection groups
  acting on $\bbS^3$ are as follows:
\begin{statictable}
  $
  \begin{array}{l||l}\fontsize{8}{9}
    \begin{array}[t]{lllll}
      \text{group} & \text{rank} & \text{order}  & \text{diagram}
      \\ \hline
      \{1\} & 0 & 1
      \\
      \Ztwo & 1 & 2 &
      \\
      D_n & 2 & 2n  & \smalldiagram{group-dihedral}
      \\
      D_n\times \Ztwo & 3 & 4n  & \smalldiagram{group-dihedral-z2}
      \\
      \tet & 3 & 24
      &  \smalldiagram{group-tet}
      \\
      \oct & 3 & 48 
      &  \smalldiagram{group-oct}
      \\
      \ico & 3 & 120 
      & \smalldiagram{group-ico}
      \\
      D_m\times D_n & 4 & 4mn  & \smalldiagram{group-double-dihedral}
      \\
      \tet\times \Ztwo & 4 & 48 
      & \smalldiagram{group-tet-z2}
    \end{array}
    &
    \fontsize{8}{9}\begin{array}[t]{lllll}
      \text{group} & \text{rank} & \text{order}  & \text{diagram}
      \\ \hline
      \oct\times \Ztwo & 4 & 96 
      & \smalldiagram{group-oct-z2}
      \\
      \ico\times \Ztwo & 4 & 240 
      & \smalldiagram{group-ico-z2}
      \\
      \fivecell & 4 & 120 
      & \smalldiagram{group-5cell}
      \\
      \demitesseract & 4 & 192 
      & \smalldiagram{group-demitesseract}
      \\
      \sixteencell & 4 & 384 
      & \smalldiagram{group-16cell}
      \\
      \twentyfourcell & 4 & 1152 
      &  \smalldiagram{group-24cell}
      \\
      \sixhundredcell & 4 & 14400 
      & \smalldiagram{group-600cell}
    \end{array}
  \end{array}
  $
  \captionof{table}{The reflection groups acting on $\bbS^3$.}\label{t:ref-group}
\end{statictable}
\end{proposition}

\begin{proof}
  See \cite{Coxeter_1935}.
\end{proof}

\typeout{== figure/lawson.tex ============================================}\begin{figure}[b]
  \includegraphics[width=0.23\textwidth]{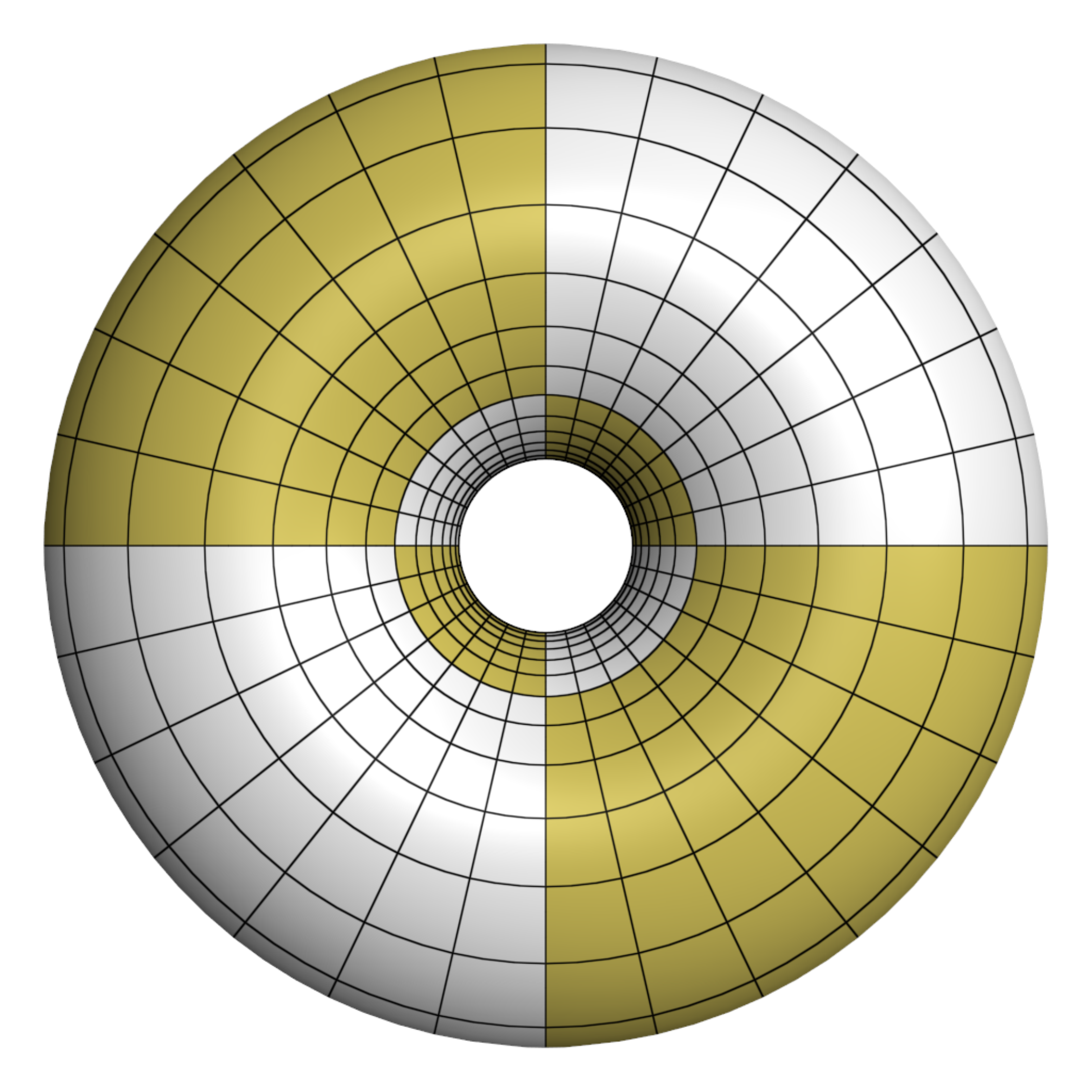}
  \includegraphics[width=0.23\textwidth]{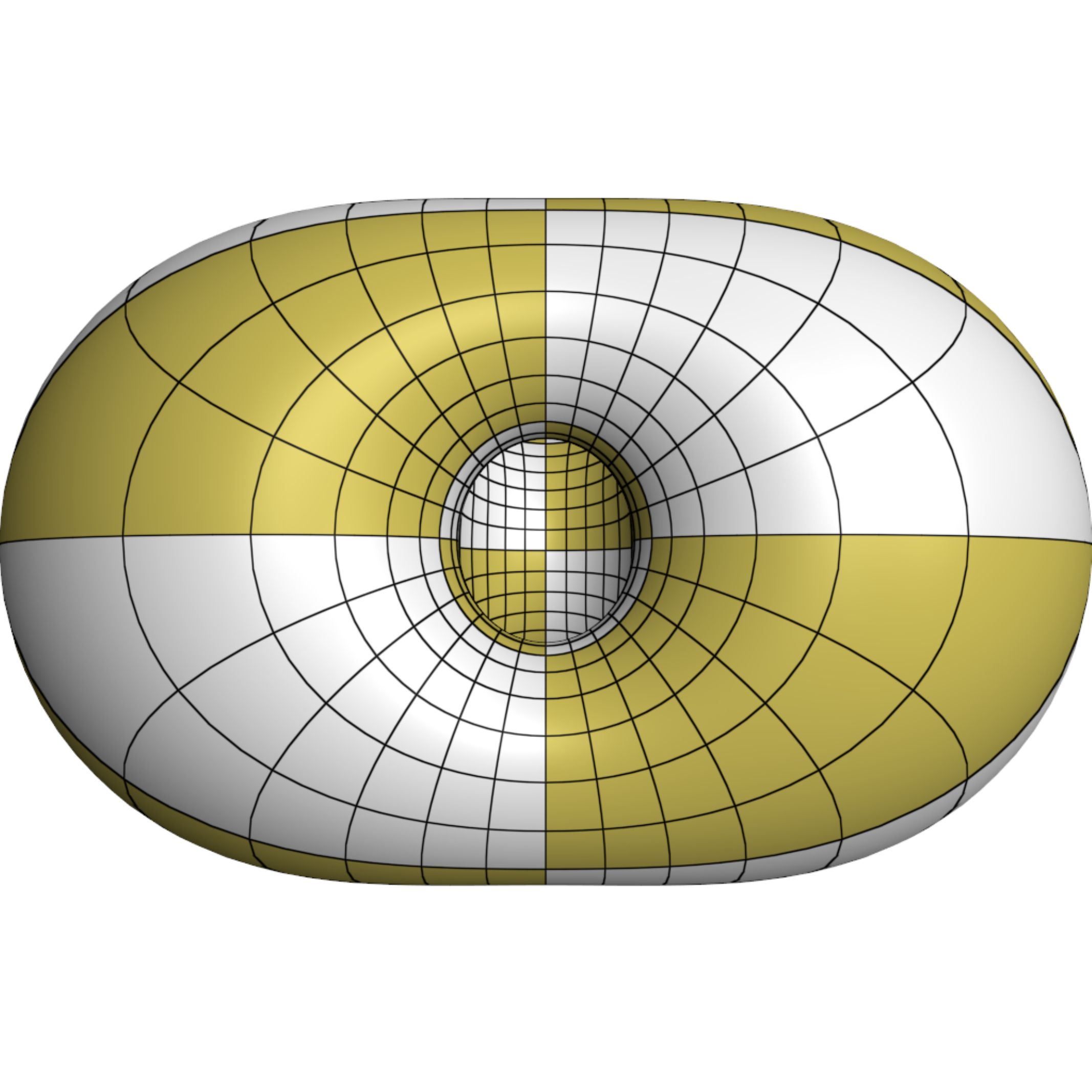}
  \includegraphics[width=0.23\textwidth]{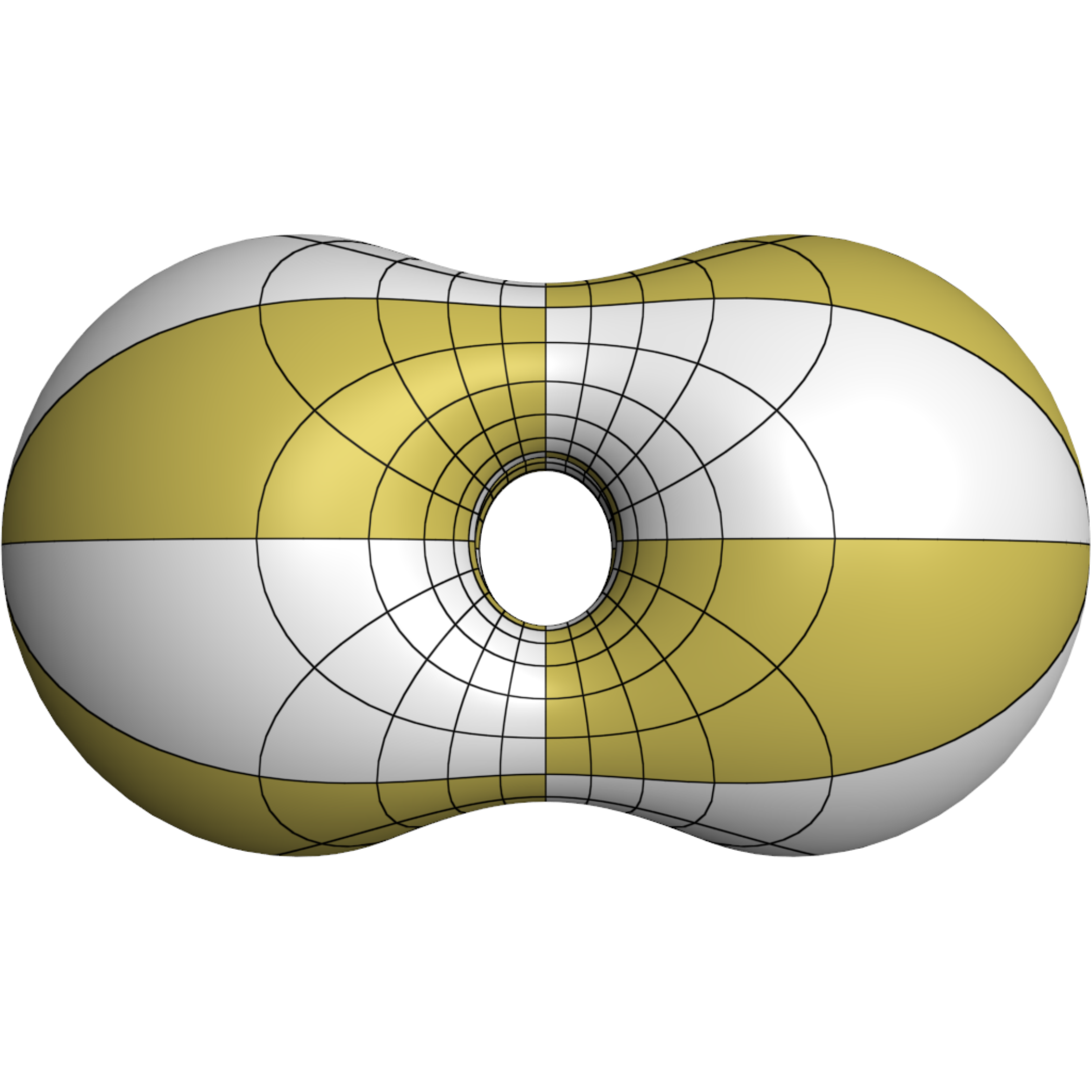}
  \includegraphics[width=0.23\textwidth]{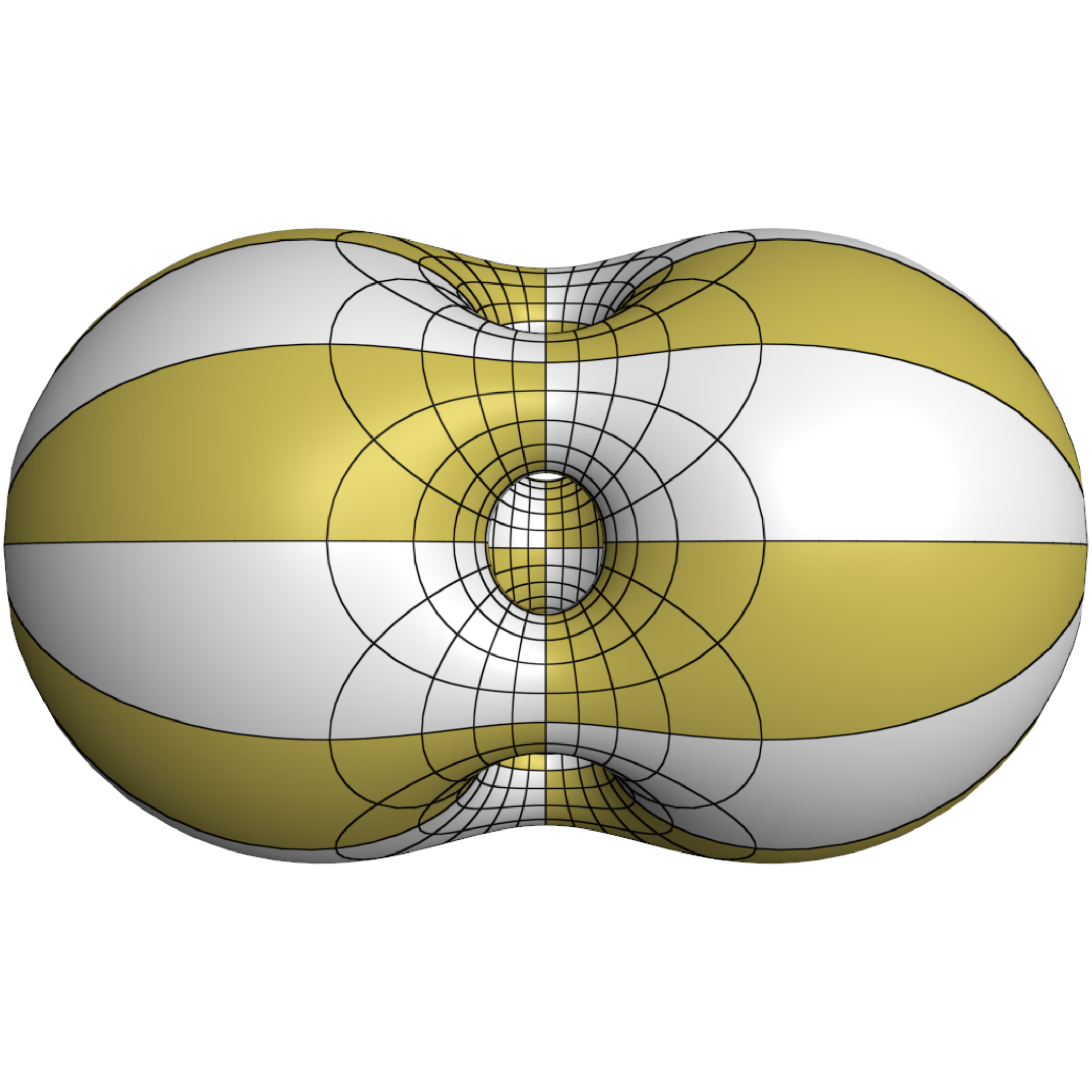}
  \caption{The dihedral family $\Asurface{1}{n}$ for $n=2,\dots, 5$ (Lawson surfaces).}\label{fig:lawson}
\end{figure}

Fundamental polyhedra $R$ have dihedral angles $\frac{\pi}{n}$ with some integer $n$'s. The mark $n$ in the diagrams in Table~\ref{t:ref-group} denotes the dihedral angle $\frac{\pi}{n}$ at the corresponding edge.
Unmarked edges in the above table have integer $n=2$.
The groups in the families
$\bbZ_2=D_1$, $D_n$, $D_n\times \bbZ_2$ and $D_m\times D_n$
are referred to as
dihedral,
and the remaining eleven groups as non-dihedral.

\subsection{Reflection Surfaces}

A \emph{star umbilic} on a surface in $\bbS^3$
is an isolated umbilic with curvature line
foliations as shown:
\vspace{1.8ex}

\begin{statictable}
  $
  \begin{array}{cccc}
  \def\svgwidth{0.2\textwidth}%
\begingroup%
  \makeatletter%
  \providecommand\color[2][]{%
    \errmessage{(Inkscape) Color is used for the text in Inkscape, but the package 'color.sty' is not loaded}%
    \renewcommand\color[2][]{}%
  }%
  \providecommand\transparent[1]{%
    \errmessage{(Inkscape) Transparency is used (non-zero) for the text in Inkscape, but the package 'transparent.sty' is not loaded}%
    \renewcommand\transparent[1]{}%
  }%
  \providecommand\rotatebox[2]{#2}%
  \newcommand*\fsize{\dimexpr\f@size pt\relax}%
  \newcommand*\lineheight[1]{\fontsize{\fsize}{#1\fsize}\selectfont}%
  \ifx\svgwidth\undefined%
    \setlength{\unitlength}{27.85775288bp}%
    \ifx\svgscale\undefined%
      \relax%
    \else%
      \setlength{\unitlength}{\unitlength * \real{\svgscale}}%
    \fi%
  \else%
    \setlength{\unitlength}{\svgwidth}%
  \fi%
  \global\let\svgwidth\undefined%
  \global\let\svgscale\undefined%
  \makeatother%
  \begin{picture}(1,0.95139184)%
    \lineheight{1}%
    \setlength\tabcolsep{0pt}%
    \put(0,0){\includegraphics[width=\unitlength,page=1]{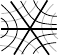}}%
  \end{picture}%
\endgroup%

    &
  \def\svgwidth{0.2\textwidth}%
\begingroup%
  \makeatletter%
  \providecommand\color[2][]{%
    \errmessage{(Inkscape) Color is used for the text in Inkscape, but the package 'color.sty' is not loaded}%
    \renewcommand\color[2][]{}%
  }%
  \providecommand\transparent[1]{%
    \errmessage{(Inkscape) Transparency is used (non-zero) for the text in Inkscape, but the package 'transparent.sty' is not loaded}%
    \renewcommand\transparent[1]{}%
  }%
  \providecommand\rotatebox[2]{#2}%
  \newcommand*\fsize{\dimexpr\f@size pt\relax}%
  \newcommand*\lineheight[1]{\fontsize{\fsize}{#1\fsize}\selectfont}%
  \ifx\svgwidth\undefined%
    \setlength{\unitlength}{27.43439888bp}%
    \ifx\svgscale\undefined%
      \relax%
    \else%
      \setlength{\unitlength}{\unitlength * \real{\svgscale}}%
    \fi%
  \else%
    \setlength{\unitlength}{\svgwidth}%
  \fi%
  \global\let\svgwidth\undefined%
  \global\let\svgscale\undefined%
  \makeatother%
  \begin{picture}(1,0.94743645)%
    \lineheight{1}%
    \setlength\tabcolsep{0pt}%
    \put(0,0){\includegraphics[width=\unitlength,page=1]{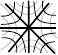}}%
  \end{picture}%
\endgroup%

    &
  \def\svgwidth{0.2\textwidth}%
\begingroup%
  \makeatletter%
  \providecommand\color[2][]{%
    \errmessage{(Inkscape) Color is used for the text in Inkscape, but the package 'color.sty' is not loaded}%
    \renewcommand\color[2][]{}%
  }%
  \providecommand\transparent[1]{%
    \errmessage{(Inkscape) Transparency is used (non-zero) for the text in Inkscape, but the package 'transparent.sty' is not loaded}%
    \renewcommand\transparent[1]{}%
  }%
  \providecommand\rotatebox[2]{#2}%
  \newcommand*\fsize{\dimexpr\f@size pt\relax}%
  \newcommand*\lineheight[1]{\fontsize{\fsize}{#1\fsize}\selectfont}%
  \ifx\svgwidth\undefined%
    \setlength{\unitlength}{23.92494414bp}%
    \ifx\svgscale\undefined%
      \relax%
    \else%
      \setlength{\unitlength}{\unitlength * \real{\svgscale}}%
    \fi%
  \else%
    \setlength{\unitlength}{\svgwidth}%
  \fi%
  \global\let\svgwidth\undefined%
  \global\let\svgscale\undefined%
  \makeatother%
  \begin{picture}(1,0.98058653)%
    \lineheight{1}%
    \setlength\tabcolsep{0pt}%
    \put(0,0){\includegraphics[width=\unitlength,page=1]{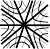}}%
  \end{picture}%
\endgroup%

    &
    \raisebox{6ex}{$\cdots$}
    \\
    \scriptstyle
    \text{umbilic order} = 1
    &
    \scriptstyle
    \text{umbilic order} = 2
    &
    \scriptstyle
    \text{umbilic order} = 3
    &
    \scriptstyle
    \text{umbilic order} = k
    \\
    \scriptstyle
    \text{index} = -\half
    &
    \scriptstyle
    \text{index} = -1
    &
    \scriptsize
    \text{index} = -\tfrac{3}{2}
    &
    \scriptstyle
    \text{index} = -\tfrac{k}{2}
  \end{array}
  $
  \captionof{table}{
    \label{tab:umbilic}
    Curvature line foliations at star umbilics.}
\end{statictable}

An immersion $f\colon M\to\bbS^3$ from an oriented surface $M$ induces a Riemann surface structure on $M$ such that $f$ is conformal.
Let $K$ be the canonical bundle of the  Riemann surface $M$. The Hopf differential 
of the surface $f$ is
the complex bilinear part of the second fundamental form, i.e.,
\[Q=II^{(2,0)}\in\Gamma(M,K^{\otimes2}).\]
Therefore, the zeros of $Q$ are exactly the umbilics of $f.$ 
We assume that the zeros of $Q$ are isolated.
Then the index of a zero $u$ of $Q$ is  defined as follows: consider a local holomorphic coordinate $z$ centered at $u$
and a smooth complex valued function $q$ such that locally $Q=q(dz)^2.$ Then the index of $Q$ at $u$ is the winding number of $q/|q|\colon S^1\cong \gamma\to S^1$, where $\gamma$ is a simple oriented  loop around the isolated singularity $u.$
If the index of a zero $u$ of $Q$ is positive, then $u$ is a star umbilic. We define the {\em umbilic order}  $o_u$ of $f$ at an umbilic $u$ to be the index of the zero of $Q$ at $u$.

\begin{remark}\label{rem:ind}
The \emph{index} of an star umbilic of umbilic order $k$
is the winding number $-k/2$ of the
vector field tangent to a curvature line foliation
along a small counterclockwise simple closed curve around the umbilic.
\end{remark}
\begin{definition}\label{def:refsurf} 
  A \emph{reflection surface} in $\bbS^3$ is a smooth
  compact connected embedded orientable surface $S\subset\bbS^3$
  which is invariant under the action
  of a reflection group $G$ acting on $\bbS^3$ such that
  \begin{itemize}
  \item
    the fundamental region $P\subset R$ of $S$ inside  a fundamental polyhedron $R$  with respect to the action of $G$
    on $\bbS^3$ is compact and simply connected;
  \item
    each umbilic on $S$ is a star umbilic.
  \end{itemize}
\end{definition}

In particular,  totally umbilical spheres are not reflection surfaces by the above definition. This allows us to exclude degenerate cases from further considerations. The examples of reflection surfaces we are mainly interested in are given by minimal surfaces in the 3-sphere.  These will be examined in detail in the next section.

The following proposition explains how a fundamental polygon
of a reflection surface lies in a fundamental polyhedron.

\begin{proposition}
  \label{prop:polygon}
  Given a reflection surface $S$ in $\bbS^3$ with symmetry group $G$,
  let $R$ be a fundamental polyhedron of the action of $G$ on $\bbS^3$.
  Let $P\coloneq S\cap R$ be a fundamental region of the action
  of $G$ on $S$.
  Then $P$ is an embedded topological disk with embedded boundary
  which is the union of curvature lines. Moreover,
  \begin{enumerate}
  \item
    \label{item:polygon-1}
    The intersection of each face interior of $R$ with $S$ is a nonempty finite disjoint union of curves. 
     \item
    \label{item:polygon-1b}  No two curves in a face of $R$ share an endpoint.
  \item
    \label{item:polygon-2}
    The intersection of each edge interior of $R$ with $S$ is a finite set of points.
  \item
    \label{item:polygon-3}
    The intersection of each vertex of $R$ with $S$ is empty.
  \end{enumerate}
\end{proposition}

\begin{proof}
  The surface reflects in each face of the polyhedron.
  The surface cannot be contained in a face, because then it would be
  totally umbilic, contradicting that its umbilics are isolated.
  By the implicit function theorem, the surface intersects each face
  along a curve.  Since the surface reflects in the faces,
 each such intersection curve is automatically a curvature line.
  To prove~\eqref{item:polygon-1}
 it therefore remains to show that the intersection with each face is non-empty. 
 In order to exclude non-empty intersection consider the totally geodesic 2-sphere containing the given face.
 The compact embedded and connected surface reflects across this 2-sphere. If the intersection of $P$ 
 with the face  would be empty,  the intersection of the compact surface with the 2-sphere would be empty as well, contradicting connectedness of $S$.
 
Since $S$ is compact,~\eqref{item:polygon-2} follows if we can show that $S$ intersects each edge transversally.
But transversality simply follows from the fact that the surface is embedded and has a well-defined tangential plane at every point.

 Likewise embeddedness of the surface
  implies~\eqref{item:polygon-3}, 
  and it also shows that two boundary curves in the same face do not intersect ~\eqref{item:polygon-1b}.
\end{proof}


\subsection{Fundamental quadrilaterals and pentagons}

The following proposition lists the number of ways a $p$-gon  can be placed in a marked fundamental $r$-hedra $R$,
for $p=4,5$ and $r=2,3,4.$ For $p>5$,  a corresponding list becomes  more involved, while for $p<4$ there are no  examples induced by reflection surfaces
due to Corollary
\ref{cor:umbilic-excess} below. 

\begin{proposition}
  Let $S$ be a reflection surface
  with fundamental $r$-hedra $R$ and
  fundamental $p$-gon $P$, where $r$ denotes the rank of the reflection group, see Table \ref{t:ref-group}.
  \begin{enumerate}
  \item
     For $r=2,\,3,\,4$,
     there are respectively $1,\,3,\,3$ ways a 4-gon (quadrilateral)
     can be placed in a marked $r$-hedra $R$, as shown:
     \begin{equation}
       \fontsize{7}{8}
  \def\svgwidth{.95\textwidth}%
\begingroup%
  \makeatletter%
  \providecommand\color[2][]{%
    \errmessage{(Inkscape) Color is used for the text in Inkscape, but the package 'color.sty' is not loaded}%
    \renewcommand\color[2][]{}%
  }%
  \providecommand\transparent[1]{%
    \errmessage{(Inkscape) Transparency is used (non-zero) for the text in Inkscape, but the package 'transparent.sty' is not loaded}%
    \renewcommand\transparent[1]{}%
  }%
  \providecommand\rotatebox[2]{#2}%
  \newcommand*\fsize{\dimexpr\f@size pt\relax}%
  \newcommand*\lineheight[1]{\fontsize{\fsize}{#1\fsize}\selectfont}%
  \ifx\svgwidth\undefined%
    \setlength{\unitlength}{266.45892101bp}%
    \ifx\svgscale\undefined%
      \relax%
    \else%
      \setlength{\unitlength}{\unitlength * \real{\svgscale}}%
    \fi%
  \else%
    \setlength{\unitlength}{\svgwidth}%
  \fi%
  \global\let\svgwidth\undefined%
  \global\let\svgscale\undefined%
  \makeatother%
  \begin{picture}(1,0.09111233)%
    \lineheight{1}%
    \setlength\tabcolsep{0pt}%
    \put(0,0){\includegraphics[width=\unitlength,page=1]{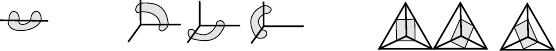}}%
  \end{picture}%
\endgroup%

     \end{equation}
   \item
     For $r=2,\,3,\,4$,
     there are $0,\,3,\,12$ ways a  5-gon (pentagon) 
     can be placed in a marked $r$-hedra $R$, as shown:
     \begin{equation}
       \fontsize{7}{8}
  \def\svgwidth{.95\textwidth}%
\begingroup%
  \makeatletter%
  \providecommand\color[2][]{%
    \errmessage{(Inkscape) Color is used for the text in Inkscape, but the package 'color.sty' is not loaded}%
    \renewcommand\color[2][]{}%
  }%
  \providecommand\transparent[1]{%
    \errmessage{(Inkscape) Transparency is used (non-zero) for the text in Inkscape, but the package 'transparent.sty' is not loaded}%
    \renewcommand\transparent[1]{}%
  }%
  \providecommand\rotatebox[2]{#2}%
  \newcommand*\fsize{\dimexpr\f@size pt\relax}%
  \newcommand*\lineheight[1]{\fontsize{\fsize}{#1\fsize}\selectfont}%
  \ifx\svgwidth\undefined%
    \setlength{\unitlength}{291.93204851bp}%
    \ifx\svgscale\undefined%
      \relax%
    \else%
      \setlength{\unitlength}{\unitlength * \real{\svgscale}}%
    \fi%
  \else%
    \setlength{\unitlength}{\svgwidth}%
  \fi%
  \global\let\svgwidth\undefined%
  \global\let\svgscale\undefined%
  \makeatother%
  \begin{picture}(1,0.19297031)%
    \lineheight{1}%
    \setlength\tabcolsep{0pt}%
    \put(0,0){\includegraphics[width=\unitlength,page=1]{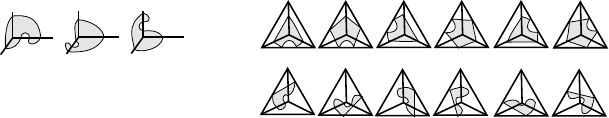}}%
  \end{picture}%
\endgroup%

     \end{equation}
  \end{enumerate}
\end{proposition}

\begin{proof}
  Number the $r$ faces of the fundamental $r$-hedron $R$ by $1,\dots,r$.
  Number each edge of the fundamental $p$-gon $P$
  with the number of the face of $R$ in which it lies.
  Listing these integers cyclically, we obtain
  a cycle $(n_1,\dots,n_p)$.
  Two such cycles are considered to be the same up to
  the group generated by
  \begin{itemize}
    \item
    rotations $(n_1,\dots,n_p)\mapsto(n_2,\dots,n_p,n_1)$,
  \item
    reversals $(n_1,\dots,n_p)\mapsto(n_p,\dots,n_2,\,n_1)$.
  \end{itemize}
  Moreover, the cycle satisfies
  \begin{itemize}
  \item
    no two consecutive integers in the cycle ($(n_k,\,n_{k-1})$ or $(n_p,\,n_1)$)
    are equal;
  \item
    each of the integers $1,\dots,r$ appears at least once in the cycle (Proposition \ref{prop:polygon} $(1)$).
  \end{itemize}
  The number of ways to place $P$ into the marked $R$ is the number
  of these cycles.
  The number of ways to place $P$ into the unmarked $R$ is the number
  of these cycles modulo renumbering.

  \textit{$4$-gons.}
  \begin{itemize}
    \item If $r\in\{0,\,1\}$, there are no cycles.
    \item If $r=2$, the unique cycle is $(1212)$.
    \item
If $r=3$, the unique cycle modulo renumbering is $(1213)$, of which there are $3$ permutations.
    \item
If $r=4$, the unique cycle modulo renumbering is $(1234)$, of which there are $3$ permutations.
  \end{itemize}

\textit{$5$-gons.}
  \begin{itemize}
    \item If $r\in\{0,\,1,\,2\}$, there are no cycles.
    \item
If $r=3$, the unique cycle modulo renumbering is $(12123)$, of which there are $3$ permutations.
    \item
      If $r=4$, the unique cycle modulo renumbering is $(12134)$, of which there are $12$ permutations.
      \qedhere
      \end{itemize}
\end{proof}

Note that  when edge integers (see Definition \ref{def:vertex-and-edge-integers} below) are assigned to the fundamental polyhedron, renumbering of the faces
gives rise to different reflection surfaces in general, e.g., the surfaces $\Bsurface{k}{\ell}$ and $ \Bsurface{\ell}{k}$
are different for $k\neq\ell,$ see for example the proof of Theorem \ref{thm:dihedral-classification} and Figure \ref{fig:b32} and
Figure \ref{fig:b23} below.


\subsection{Genus}

The vertices and edges of a fundamental polygon
of a reflection surface are assigned integers as follows.

\typeout{== figure/button-a.tex ============================================}\begin{figure}[b]
  \includegraphics[width=0.375\textwidth]{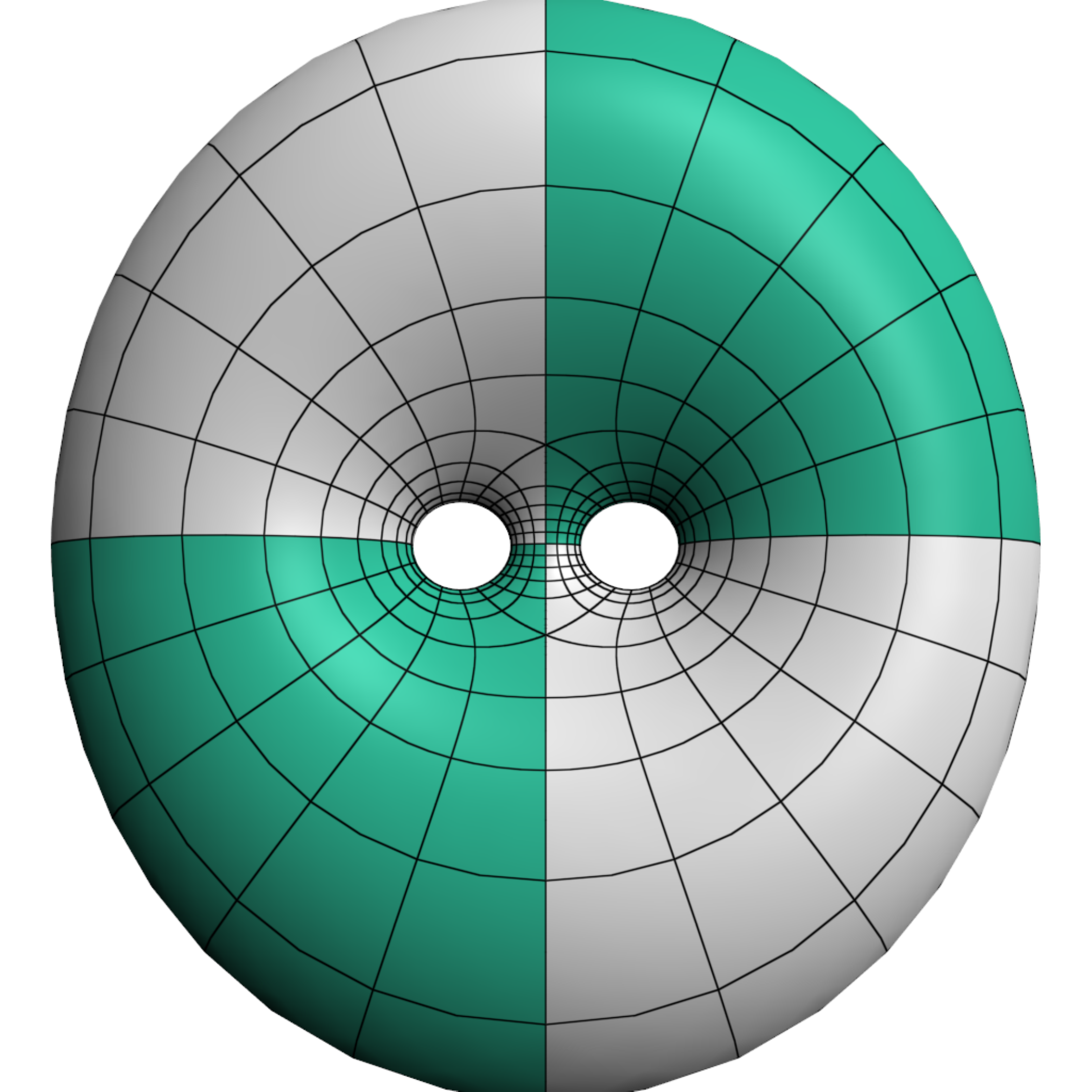}
  \includegraphics[width=0.375\textwidth]{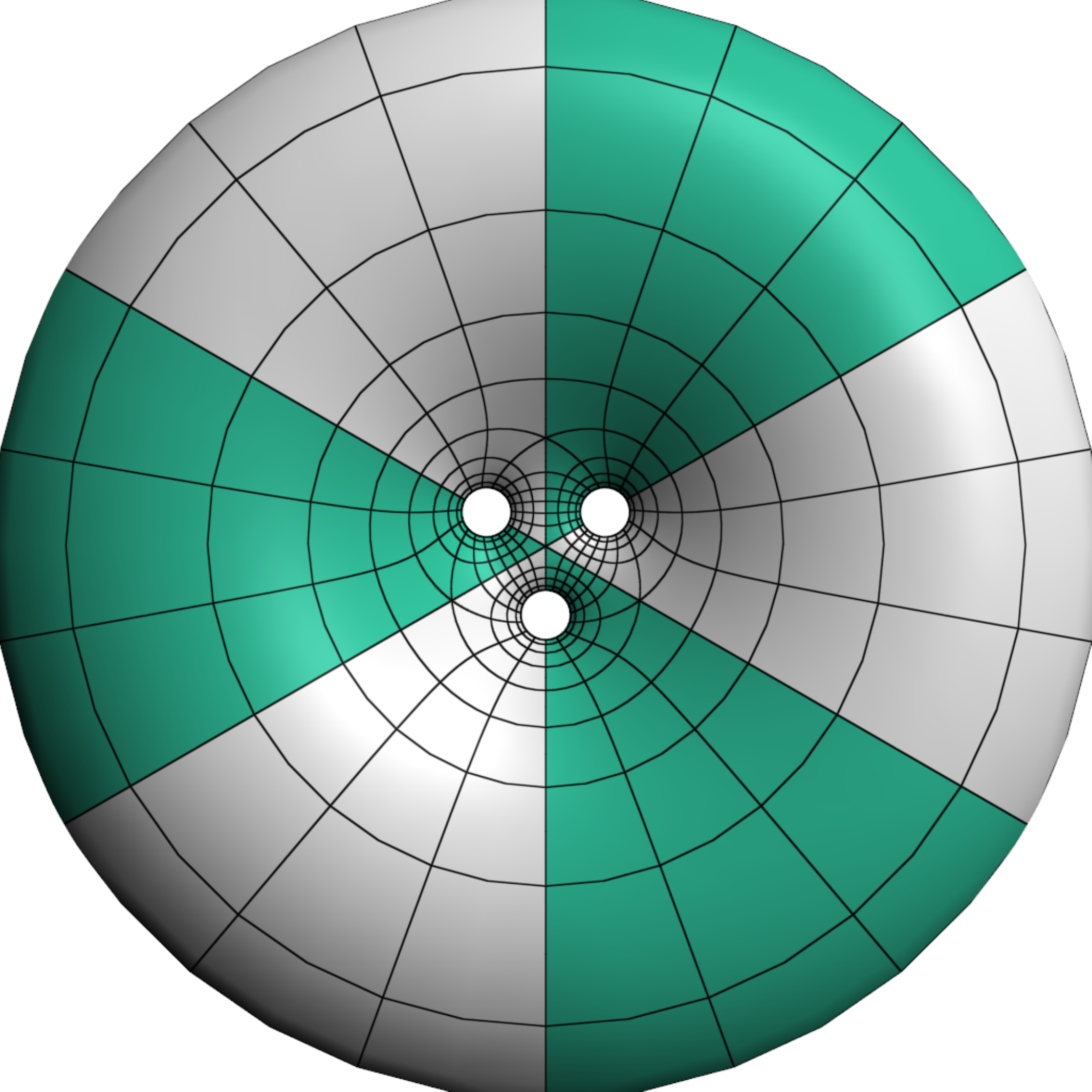}
  \includegraphics[width=0.375\textwidth]{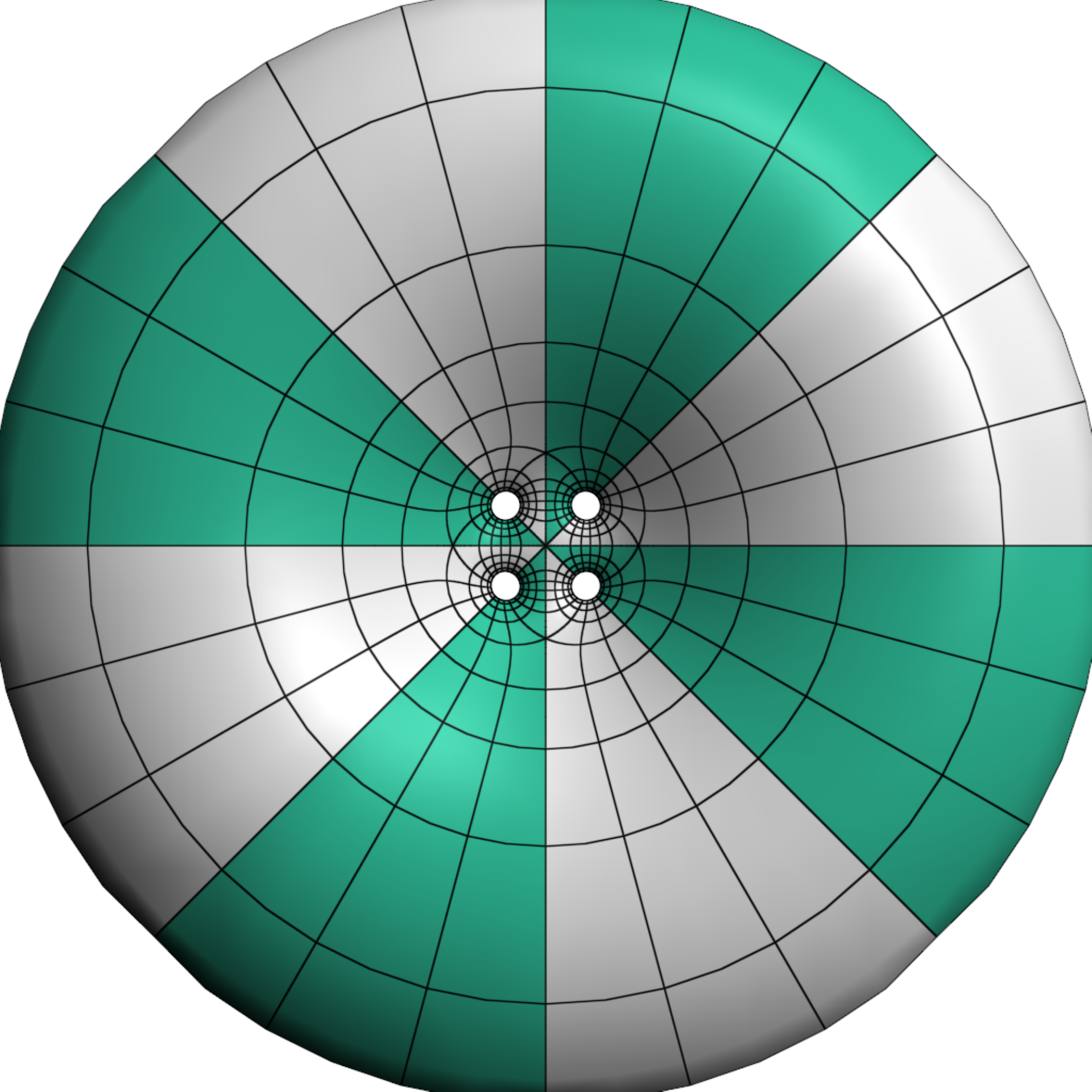}
  \includegraphics[width=0.375\textwidth]{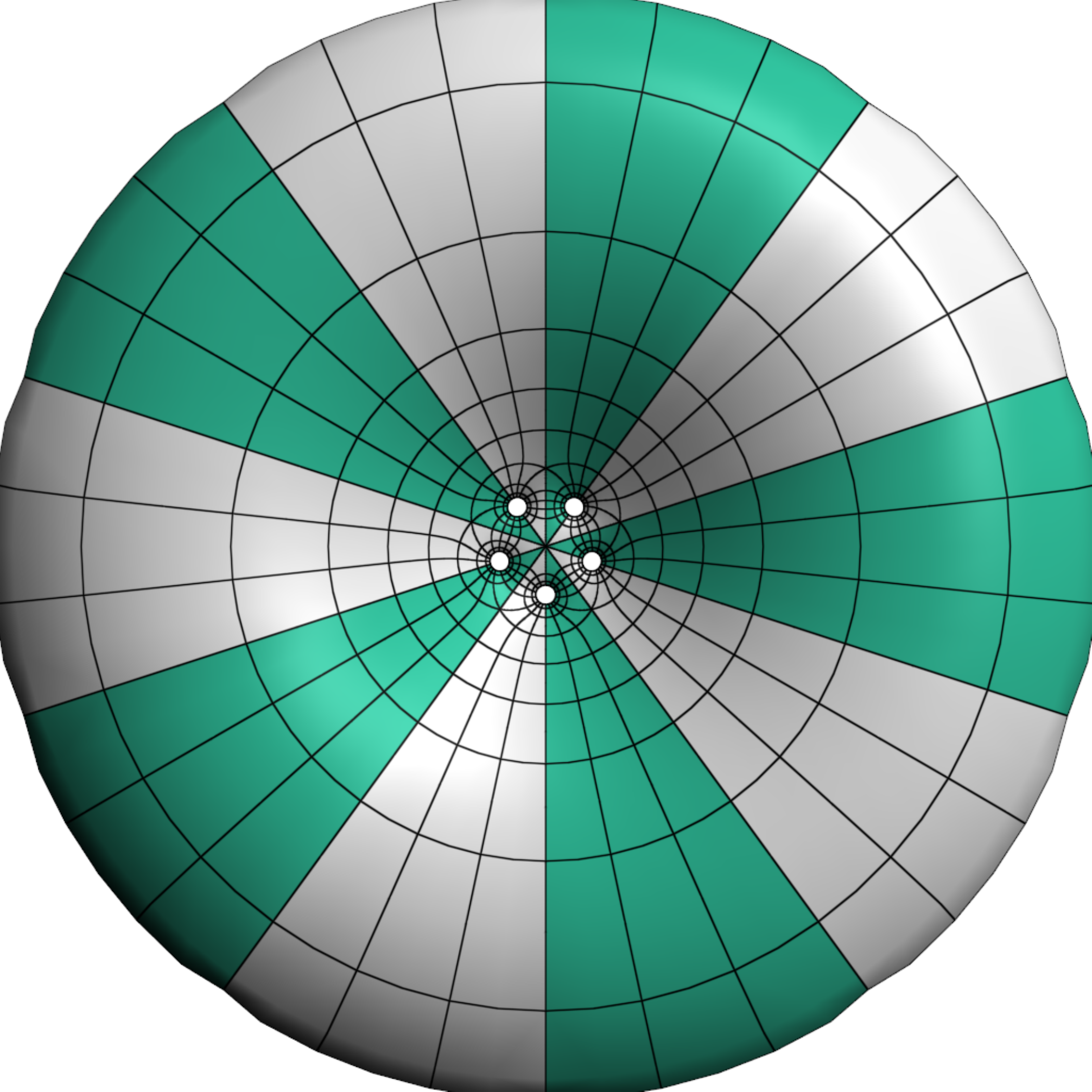}
  \caption{The dihedral family $\Bsurface{n}{1}$ for $n=2,\dots, 5$.}\label{fig:bn1}
\end{figure}

\begin{definition}
  \label{def:vertex-and-edge-integers}
  \theoremname{Vertex and edge integers}
  Let $S$ be a reflection surface with reflection group $G$,
  fundamental polyhedron $R$, and fundamental polygon $P$.  
   \begin{itemize}
  \item
    Each vertex $v$ of $P$ lies on an edge $e$ of $R$.
    The two edges $e_1$ and $e_2$ of $P$ incident to $v$
    lie in two distinct faces $f_1$ and $f_2$ of $R$.
    Then $f_1$ and $f_2$ meet along $e$
    at an interior dihedral angle $\pi/n$, $n\in\bbN_{\ge 1}$.
    Assign to $v$ the \emph{vertex integer} $n$.
  \item
    Each edge $e$ of $P$ lies in a face $f$ of $R$.
    The two endpoints of $e$ lie on two edges $e_1$ and $e_2$ of $R$.
    Let $f_1$ and $f_2$ be the faces of $R$ such that
    $f_1\cap f = e_1$ and $f_2\cap f = e_2$.
    Then $f_1$ and $f_2$ meet along an edge of $R$
    at an interior dihedral angle $\pi/m$, $m\in\bbN_{\ge 1}$,
    and we assign to $e$ the \emph{edge integer} $m$.
  \end{itemize}
\end{definition}

The genus of a reflection surface $S$
with finite reflection group $G$
can be computed using that
the tessellation of $\bbS^3$ induced by $G$
induces a tessellation of $S$ into polygons.

\begin{theorem}
  \label{prop:genus}
  Let $S$ be a reflection surface
  with finite reflection group $G$ of order $\abs{G}<\infty$,
  and let $P$ be a fundamental polygon with
  vertex integers $(n_1,\dots n_p)$
  at its $p\ge 4$ vertices.
  Then the genus of $S$ is
  \begin{equation}
    \label{eq:genus}
    \textstyle
    g = 1 + \frac{\abs{G}}{4}\bigl( p - 2 - \sum_{k=1}^{p}\frac{1}{n_k}\bigr)\,.
  \end{equation}
\end{theorem}

\begin{proof}
  Let $V,\,E,\,F$ be the number of vertices, edges
  and faces of $S$. Then
  \begin{itemize}
  \item
    the size of the orbit of vertex $k$ of $S$ is
    $\half\abs{G}/n_k$,
    so $V = \half\abs{G}\sum_{k=1}^{p}\frac{1}{n_k}$;
  \item
    the size of the orbit of an edge of $S$ is
    $\half\abs{G}$, so $E = \half\abs{G} p$;
  \item
    $F = \abs{G}$
  \end{itemize}
  As $\chi = V - E + F$
  is the Euler characteristic of $S$,
and its genus satisfies $g = 1 - \half\chi$, the result follows.
\end{proof}

\subsection{Curvature line polygons}

On a surface $S$,
let $V$ be the tangent vector field
to a curvature line foliation, and
let $\gamma$ be a
counterclockwise
simple closed curve on $S$ bounding
a topological disk $D$.
The \emph{winding number} of $V$ along $\gamma$
is equal to the sum of the
indices of the vector field at the umbilics in $D$.

\begin{definition}
  \label{def:curvature-line-polygon}
  \theoremname{Curvature line polygon}
  A \emph{curvature line polygon} on a surface is
  a $p$-gon bounding a topological disk whose edges
  are curvature lines, and all umbilics in $P\cup\del P$ are star umbilics.
\end{definition}
The fundamental $p$-gon of a reflection surface is a curvature line polygon by Proposition \ref{prop:polygon}.

\begin{definition}
  \label{def:umbilic-excess}
  \theoremname{Umbilic excess}
  The \emph{umbilic excess}
  of a curvature line $p$-gon $P$
  is $\kappa = \Sigma \kappa_u$, where we sum over all umbilics $u$ with umbilic oder $o_u$ of $P$ and $\kappa_u\in\half\bbN_{\ge 0}$
  is as follows:
  \[\kappa_u=\begin{cases} 
    \frac{(o_u+2-n_u)}{2n_u} &\text{ if } u \text{ is a vertex umbilic;}\\
  \tfrac{1}{2}o_u &\text{ if } u \text{ is an edge umbilic;}\\
   o_u &\text{ if } u \text{ is a face umbilic.}\\
  \end{cases}\]
  
  The following table shows some basic examples for the umbilic excess for umbilics lying on vertices, edges and faces of the $p$-gon $P$.
  
 \begin{statictable}
    $
    \fontsize{7}{8}
  \def\svgwidth{0.75\textwidth}%
\begingroup%
  \makeatletter%
  \providecommand\color[2][]{%
    \errmessage{(Inkscape) Color is used for the text in Inkscape, but the package 'color.sty' is not loaded}%
    \renewcommand\color[2][]{}%
  }%
  \providecommand\transparent[1]{%
    \errmessage{(Inkscape) Transparency is used (non-zero) for the text in Inkscape, but the package 'transparent.sty' is not loaded}%
    \renewcommand\transparent[1]{}%
  }%
  \providecommand\rotatebox[2]{#2}%
  \newcommand*\fsize{\dimexpr\f@size pt\relax}%
  \newcommand*\lineheight[1]{\fontsize{\fsize}{#1\fsize}\selectfont}%
  \ifx\svgwidth\undefined%
    \setlength{\unitlength}{149.16362999bp}%
    \ifx\svgscale\undefined%
      \relax%
    \else%
      \setlength{\unitlength}{\unitlength * \real{\svgscale}}%
    \fi%
  \else%
    \setlength{\unitlength}{\svgwidth}%
  \fi%
  \global\let\svgwidth\undefined%
  \global\let\svgscale\undefined%
  \makeatother%
  \begin{picture}(1,0.61818902)%
    \lineheight{1}%
    \setlength\tabcolsep{0pt}%
    \put(0.25544665,0.4743524){\color[rgb]{0,0,0}\makebox(0,0)[lt]{\lineheight{1.25}\smash{\begin{tabular}[t]{l}$\kappa_\ell=0$\end{tabular}}}}%
    \put(0.87691472,0.52178346){\color[rgb]{0,0,0}\makebox(0,0)[lt]{\lineheight{1.25}\smash{\begin{tabular}[t]{l}$\cdots$\end{tabular}}}}%
    \put(0.87586448,0.34821287){\color[rgb]{0,0,0}\makebox(0,0)[lt]{\lineheight{1.25}\smash{\begin{tabular}[t]{l}$\cdots$\end{tabular}}}}%
    \put(0.87458077,0.10747372){\color[rgb]{0,0,0}\makebox(0,0)[lt]{\lineheight{1.25}\smash{\begin{tabular}[t]{l}$\cdots$\end{tabular}}}}%
    \put(0.47646611,0.47217427){\color[rgb]{0,0,0}\makebox(0,0)[lt]{\lineheight{1.25}\smash{\begin{tabular}[t]{l}$\kappa_\ell=\tfrac{1}{2}$\end{tabular}}}}%
    \put(0.718149,0.47117969){\color[rgb]{0,0,0}\makebox(0,0)[lt]{\lineheight{1.25}\smash{\begin{tabular}[t]{l}$\kappa_\ell=1$\end{tabular}}}}%
    \put(0,0.53218447){\color[rgb]{0,0,0}\makebox(0,0)[lt]{\lineheight{1.25}\smash{\begin{tabular}[t]{l}vertex umbilic\end{tabular}}}}%
    \put(0,0){\includegraphics[width=\unitlength,page=1]{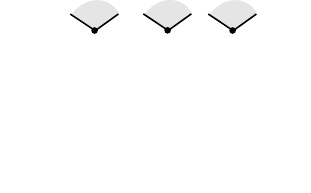}}%
    \put(0.25530874,0.27073388){\color[rgb]{0,0,0}\makebox(0,0)[lt]{\lineheight{1.25}\smash{\begin{tabular}[t]{l}$\kappa_\ell=0$\end{tabular}}}}%
    \put(0.47686944,0.26973933){\color[rgb]{0,0,0}\makebox(0,0)[lt]{\lineheight{1.25}\smash{\begin{tabular}[t]{l}$\kappa_\ell=\tfrac{1}{2}$\end{tabular}}}}%
    \put(0.71763773,0.26915525){\color[rgb]{0,0,0}\makebox(0,0)[lt]{\lineheight{1.25}\smash{\begin{tabular}[t]{l}$\kappa_\ell=1$\end{tabular}}}}%
    \put(0.25505265,0.01055775){\color[rgb]{0,0,0}\makebox(0,0)[lt]{\lineheight{1.25}\smash{\begin{tabular}[t]{l}$\kappa_\ell=0$\end{tabular}}}}%
    \put(0.47663068,0.00837958){\color[rgb]{0,0,0}\makebox(0,0)[lt]{\lineheight{1.25}\smash{\begin{tabular}[t]{l}$\kappa_\ell=1$\end{tabular}}}}%
    \put(0.71830382,0.00738512){\color[rgb]{0,0,0}\makebox(0,0)[lt]{\lineheight{1.25}\smash{\begin{tabular}[t]{l}$\kappa_\ell=2$\end{tabular}}}}%
    \put(0.00006621,0.33074412){\color[rgb]{0,0,0}\makebox(0,0)[lt]{\lineheight{1.25}\smash{\begin{tabular}[t]{l}edge umbilic\end{tabular}}}}%
    \put(0.00229889,0.11123354){\color[rgb]{0,0,0}\makebox(0,0)[lt]{\lineheight{1.25}\smash{\begin{tabular}[t]{l}face umbilic\end{tabular}}}}%
    \put(0,0){\includegraphics[width=\unitlength,page=2]{umbilic-excess_svg-tex.pdf}}%
  \end{picture}%
\endgroup%

    $
    \captionof{table}{The umbilic excess for an umbilic $u$ at
      a vertex, edge, or face of a polygon.}
    \label{tab:alpha-beta}
  \end{statictable}
\end{definition}

The following theorem generalizes the formula for the degree of the square of the canonical bundle of a reflection
surface to a  curvature line $p$-gon.
\begin{theorem}
  \label{thm:umbilic-excess}
  \theoremname{Umbilic excess}
  Let $P$ be a curvature line $p$-gon (Definition \ref{def:curvature-line-polygon})
  with umbilic excess $\kappa$. Then
  \begin{equation}
    \label{eq:umbilic-excess}
    \kappa = \half(p-4)
    \spaceperiod
  \end{equation}
\end{theorem}

\begin{proof}
  Draw a curvilinear polygon $\gamma\subset P\cup\del P$
  with curvature line edges meeting at angles of $\pm\pi/2$,
  as shown in the following example:
  \begin{equation}
    \label{eq:umbilic-excess-shortcut}
    \fontsize{7}{8}
  \def\svgwidth{0.2\textwidth}%
\begingroup%
  \makeatletter%
  \providecommand\color[2][]{%
    \errmessage{(Inkscape) Color is used for the text in Inkscape, but the package 'color.sty' is not loaded}%
    \renewcommand\color[2][]{}%
  }%
  \providecommand\transparent[1]{%
    \errmessage{(Inkscape) Transparency is used (non-zero) for the text in Inkscape, but the package 'transparent.sty' is not loaded}%
    \renewcommand\transparent[1]{}%
  }%
  \providecommand\rotatebox[2]{#2}%
  \newcommand*\fsize{\dimexpr\f@size pt\relax}%
  \newcommand*\lineheight[1]{\fontsize{\fsize}{#1\fsize}\selectfont}%
  \ifx\svgwidth\undefined%
    \setlength{\unitlength}{113.99390039bp}%
    \ifx\svgscale\undefined%
      \relax%
    \else%
      \setlength{\unitlength}{\unitlength * \real{\svgscale}}%
    \fi%
  \else%
    \setlength{\unitlength}{\svgwidth}%
  \fi%
  \global\let\svgwidth\undefined%
  \global\let\svgscale\undefined%
  \makeatother%
  \begin{picture}(1,0.95061992)%
    \lineheight{1}%
    \setlength\tabcolsep{0pt}%
    \put(0,0){\includegraphics[width=\unitlength,page=1]{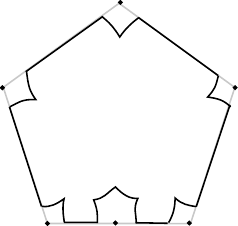}}%
  \end{picture}%
\endgroup%

  \end{equation}
  where
  \begin{itemize}
  \item
    near each vertex of $P$, $\gamma$ is
    as shown in the first row of~\eqref{eq:umbilic-excess-part};
  \item
    near each edge umbilic of $P$, $\gamma$ is
    as shown in the second row of~\eqref{eq:umbilic-excess-part};
  \end{itemize}
and $\mathring{\kappa}$ denotes the umbilic excess of the respective umbilic.

  \begin{equation}
    \label{eq:umbilic-excess-part}
    \fontsize{7}{8}
  \def\svgwidth{0.7\textwidth}%
\begingroup%
  \makeatletter%
  \providecommand\color[2][]{%
    \errmessage{(Inkscape) Color is used for the text in Inkscape, but the package 'color.sty' is not loaded}%
    \renewcommand\color[2][]{}%
  }%
  \providecommand\transparent[1]{%
    \errmessage{(Inkscape) Transparency is used (non-zero) for the text in Inkscape, but the package 'transparent.sty' is not loaded}%
    \renewcommand\transparent[1]{}%
  }%
  \providecommand\rotatebox[2]{#2}%
  \newcommand*\fsize{\dimexpr\f@size pt\relax}%
  \newcommand*\lineheight[1]{\fontsize{\fsize}{#1\fsize}\selectfont}%
  \ifx\svgwidth\undefined%
    \setlength{\unitlength}{484.31344984bp}%
    \ifx\svgscale\undefined%
      \relax%
    \else%
      \setlength{\unitlength}{\unitlength * \real{\svgscale}}%
    \fi%
  \else%
    \setlength{\unitlength}{\svgwidth}%
  \fi%
  \global\let\svgwidth\undefined%
  \global\let\svgscale\undefined%
  \makeatother%
  \begin{picture}(1,0.40196136)%
    \lineheight{1}%
    \setlength\tabcolsep{0pt}%
    \put(0,0){\includegraphics[width=\unitlength,page=1]{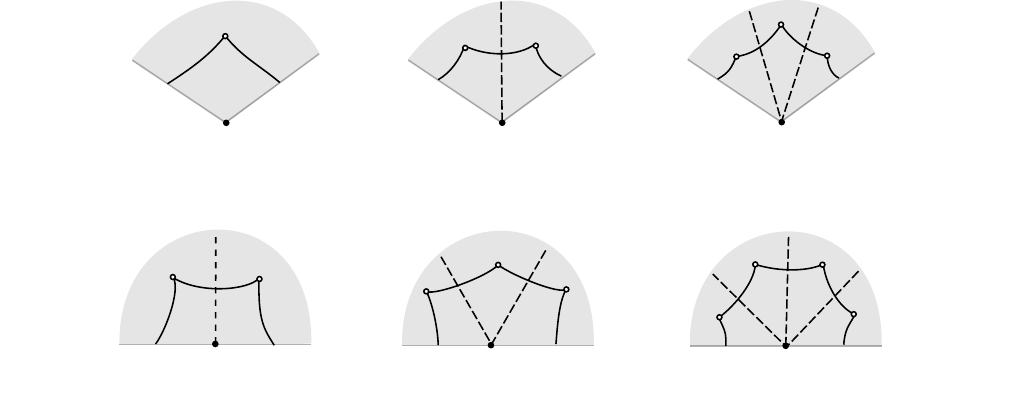}}%
    \put(-0.00035855,0.11031447){\color[rgb]{0,0,0}\makebox(0,0)[lt]{\lineheight{1.25}\smash{\begin{tabular}[t]{l}edge\end{tabular}}}}%
    \put(0.13003109,0.29571205){\color[rgb]{0,0,0}\makebox(0,0)[lt]{\lineheight{1.25}\smash{\begin{tabular}[t]{l}$+$\end{tabular}}}}%
    \put(0.19099824,0.23684366){\color[rgb]{0,0,0}\makebox(0,0)[lt]{\lineheight{1.25}\smash{\begin{tabular}[t]{l}$\mathring{\kappa}=0$\end{tabular}}}}%
    \put(0.46130285,0.23735543){\color[rgb]{0,0,0}\makebox(0,0)[lt]{\lineheight{1.25}\smash{\begin{tabular}[t]{l}$\mathring{\kappa}=\tfrac{1}{2}$\end{tabular}}}}%
    \put(0.74294578,0.23735543){\color[rgb]{0,0,0}\makebox(0,0)[lt]{\lineheight{1.25}\smash{\begin{tabular}[t]{l}$\mathring{\kappa}=1$\end{tabular}}}}%
    \put(0.456008,0.00296745){\color[rgb]{0,0,0}\makebox(0,0)[lt]{\lineheight{1.25}\smash{\begin{tabular}[t]{l}$\mathring{\kappa}=\tfrac{1}{2}$\end{tabular}}}}%
    \put(0.74647565,0.00296745){\color[rgb]{0,0,0}\makebox(0,0)[lt]{\lineheight{1.25}\smash{\begin{tabular}[t]{l}$\mathring{\kappa}=1$\end{tabular}}}}%
    \put(0.18144399,0.00227454){\color[rgb]{0,0,0}\makebox(0,0)[lt]{\lineheight{1.25}\smash{\begin{tabular}[t]{l}$\mathring{\kappa}=0$\end{tabular}}}}%
    \put(0.28009949,0.29890699){\color[rgb]{0,0,0}\makebox(0,0)[lt]{\lineheight{1.25}\smash{\begin{tabular}[t]{l}$+$\end{tabular}}}}%
    \put(0.39323728,0.30165742){\color[rgb]{0,0,0}\makebox(0,0)[lt]{\lineheight{1.25}\smash{\begin{tabular}[t]{l}$+$\end{tabular}}}}%
    \put(0.56034927,0.30096983){\color[rgb]{0,0,0}\makebox(0,0)[lt]{\lineheight{1.25}\smash{\begin{tabular}[t]{l}$+$\end{tabular}}}}%
    \put(0.67590251,0.30028221){\color[rgb]{0,0,0}\makebox(0,0)[lt]{\lineheight{1.25}\smash{\begin{tabular}[t]{l}$+$\end{tabular}}}}%
    \put(0.83647934,0.30096983){\color[rgb]{0,0,0}\makebox(0,0)[lt]{\lineheight{1.25}\smash{\begin{tabular}[t]{l}$+$\end{tabular}}}}%
    \put(0.42041245,0.02483981){\color[rgb]{0,0,0}\makebox(0,0)[lt]{\lineheight{1.25}\smash{\begin{tabular}[t]{l}$+$\end{tabular}}}}%
    \put(0.53730563,0.02415221){\color[rgb]{0,0,0}\makebox(0,0)[lt]{\lineheight{1.25}\smash{\begin{tabular}[t]{l}$+$\end{tabular}}}}%
    \put(0.13947508,0.02552743){\color[rgb]{0,0,0}\makebox(0,0)[lt]{\lineheight{1.25}\smash{\begin{tabular}[t]{l}$+$\end{tabular}}}}%
    \put(0.2570559,0.02552743){\color[rgb]{0,0,0}\makebox(0,0)[lt]{\lineheight{1.25}\smash{\begin{tabular}[t]{l}$+$\end{tabular}}}}%
    \put(0.70479376,0.02483981){\color[rgb]{0,0,0}\makebox(0,0)[lt]{\lineheight{1.25}\smash{\begin{tabular}[t]{l}$+$\end{tabular}}}}%
    \put(0.8223745,0.02552743){\color[rgb]{0,0,0}\makebox(0,0)[lt]{\lineheight{1.25}\smash{\begin{tabular}[t]{l}$+$\end{tabular}}}}%
    \put(0,0.30833087){\color[rgb]{0,0,0}\makebox(0,0)[lt]{\lineheight{1.25}\smash{\begin{tabular}[t]{l}vertex\end{tabular}}}}%
    \put(0.96664006,0.11581533){\color[rgb]{0,0,0}\makebox(0,0)[lt]{\lineheight{1.25}\smash{\begin{tabular}[t]{l}$\dots$\end{tabular}}}}%
    \put(0.96638505,0.32712372){\color[rgb]{0,0,0}\makebox(0,0)[lt]{\lineheight{1.25}\smash{\begin{tabular}[t]{l}$\dots$\end{tabular}}}}%
    \put(0.21595214,0.38159922){\color[rgb]{0,0,0}\makebox(0,0)[lt]{\lineheight{1.25}\smash{\begin{tabular}[t]{l}$-$\end{tabular}}}}%
    \put(0.15888065,0.14667011){\color[rgb]{0,0,0}\makebox(0,0)[lt]{\lineheight{1.25}\smash{\begin{tabular}[t]{l}$-$\end{tabular}}}}%
    \put(0.26202169,0.14529487){\color[rgb]{0,0,0}\makebox(0,0)[lt]{\lineheight{1.25}\smash{\begin{tabular}[t]{l}$-$\end{tabular}}}}%
    \put(0.39821488,0.12055276){\color[rgb]{0,0,0}\makebox(0,0)[lt]{\lineheight{1.25}\smash{\begin{tabular}[t]{l}$-$\end{tabular}}}}%
    \put(0.48932571,0.15321116){\color[rgb]{0,0,0}\makebox(0,0)[lt]{\lineheight{1.25}\smash{\begin{tabular}[t]{l}$-$\end{tabular}}}}%
    \put(0.5673662,0.12226883){\color[rgb]{0,0,0}\makebox(0,0)[lt]{\lineheight{1.25}\smash{\begin{tabular}[t]{l}$-$\end{tabular}}}}%
    \put(0.69050655,0.09717412){\color[rgb]{0,0,0}\makebox(0,0)[lt]{\lineheight{1.25}\smash{\begin{tabular}[t]{l}$-$\end{tabular}}}}%
    \put(0.72970021,0.16180919){\color[rgb]{0,0,0}\makebox(0,0)[lt]{\lineheight{1.25}\smash{\begin{tabular}[t]{l}$-$\end{tabular}}}}%
    \put(0.81117865,0.15871202){\color[rgb]{0,0,0}\makebox(0,0)[lt]{\lineheight{1.25}\smash{\begin{tabular}[t]{l}$-$\end{tabular}}}}%
    \put(0.86378353,0.10886345){\color[rgb]{0,0,0}\makebox(0,0)[lt]{\lineheight{1.25}\smash{\begin{tabular}[t]{l}$-$\end{tabular}}}}%
    \put(0.82458991,0.35713068){\color[rgb]{0,0,0}\makebox(0,0)[lt]{\lineheight{1.25}\smash{\begin{tabular}[t]{l}$-$\end{tabular}}}}%
    \put(0.77576982,0.38669779){\color[rgb]{0,0,0}\makebox(0,0)[lt]{\lineheight{1.25}\smash{\begin{tabular}[t]{l}$-$\end{tabular}}}}%
    \put(0.72110209,0.36538197){\color[rgb]{0,0,0}\makebox(0,0)[lt]{\lineheight{1.25}\smash{\begin{tabular}[t]{l}$-$\end{tabular}}}}%
    \put(0.44642959,0.37019522){\color[rgb]{0,0,0}\makebox(0,0)[lt]{\lineheight{1.25}\smash{\begin{tabular}[t]{l}$-$\end{tabular}}}}%
    \put(0.53546002,0.36743305){\color[rgb]{0,0,0}\makebox(0,0)[lt]{\lineheight{1.25}\smash{\begin{tabular}[t]{l}$-$\end{tabular}}}}%
    \put(0,0){\includegraphics[width=\unitlength,page=2]{umbilic-excess-part_svg-tex.pdf}}%
  \end{picture}%
\endgroup%

  \end{equation}

\typeout{== figure/buttonhole-a.tex ============================================}\begin{figure}[b]
  \includegraphics[width=0.375\textwidth]{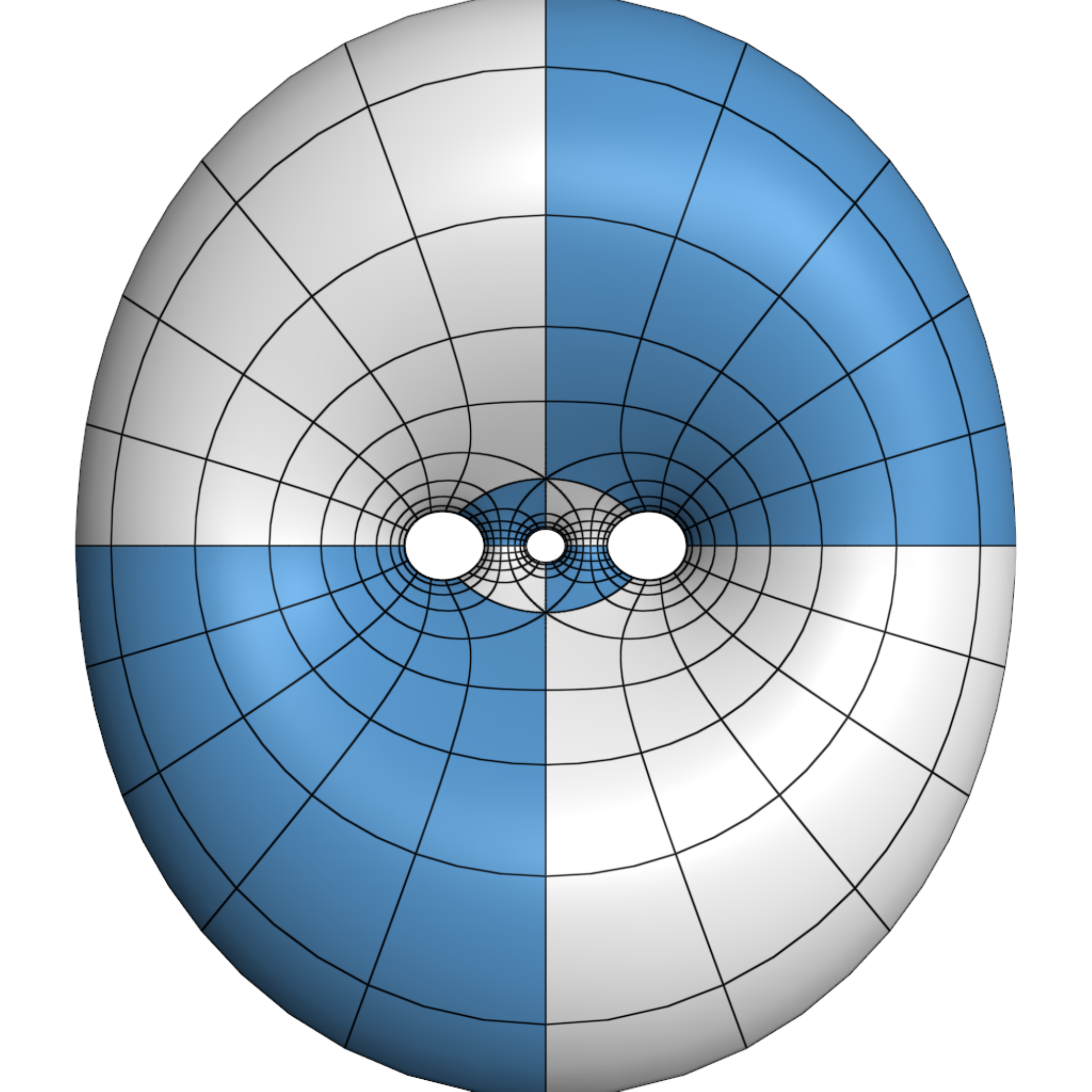}
  \includegraphics[width=0.375\textwidth]{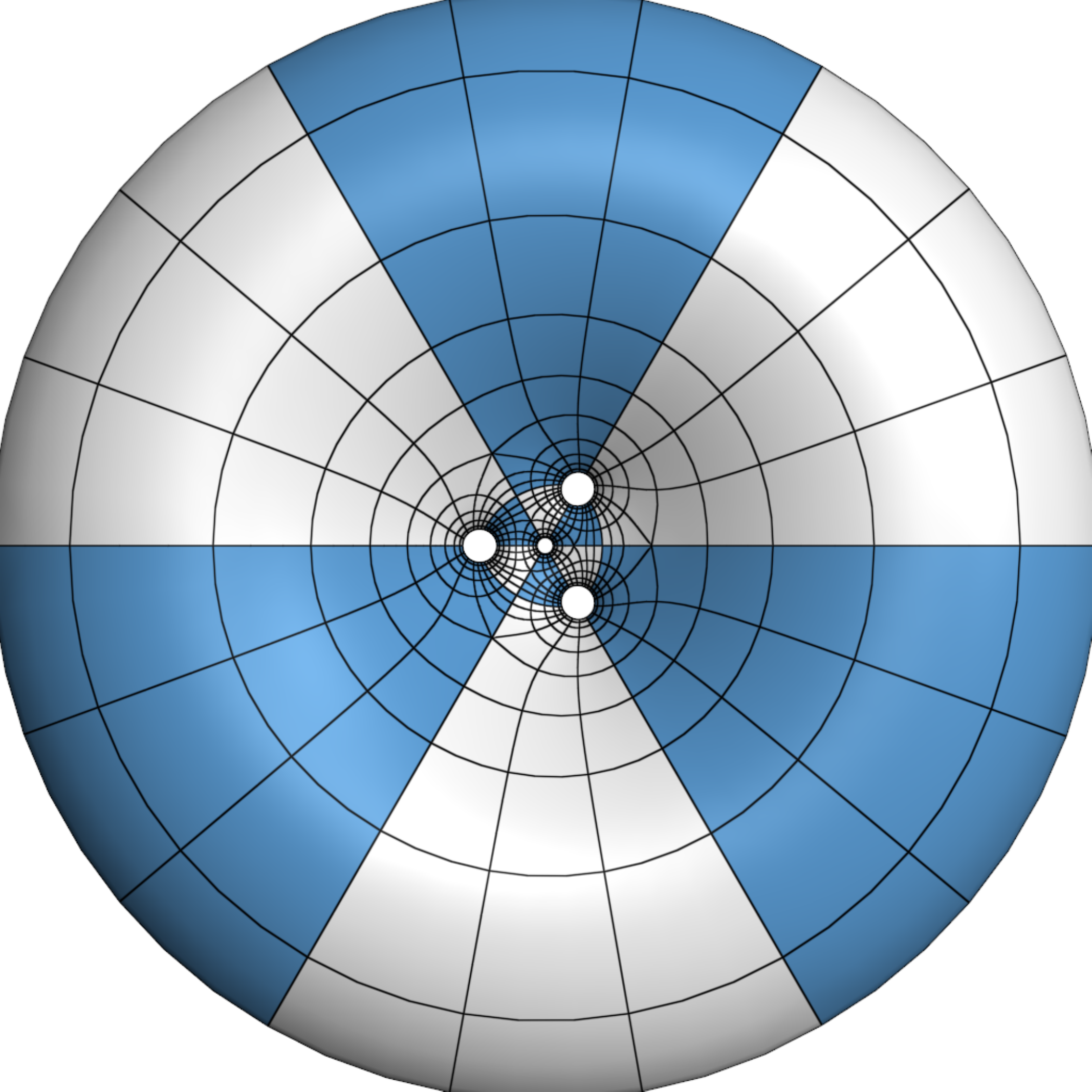}
  \includegraphics[width=0.375\textwidth]{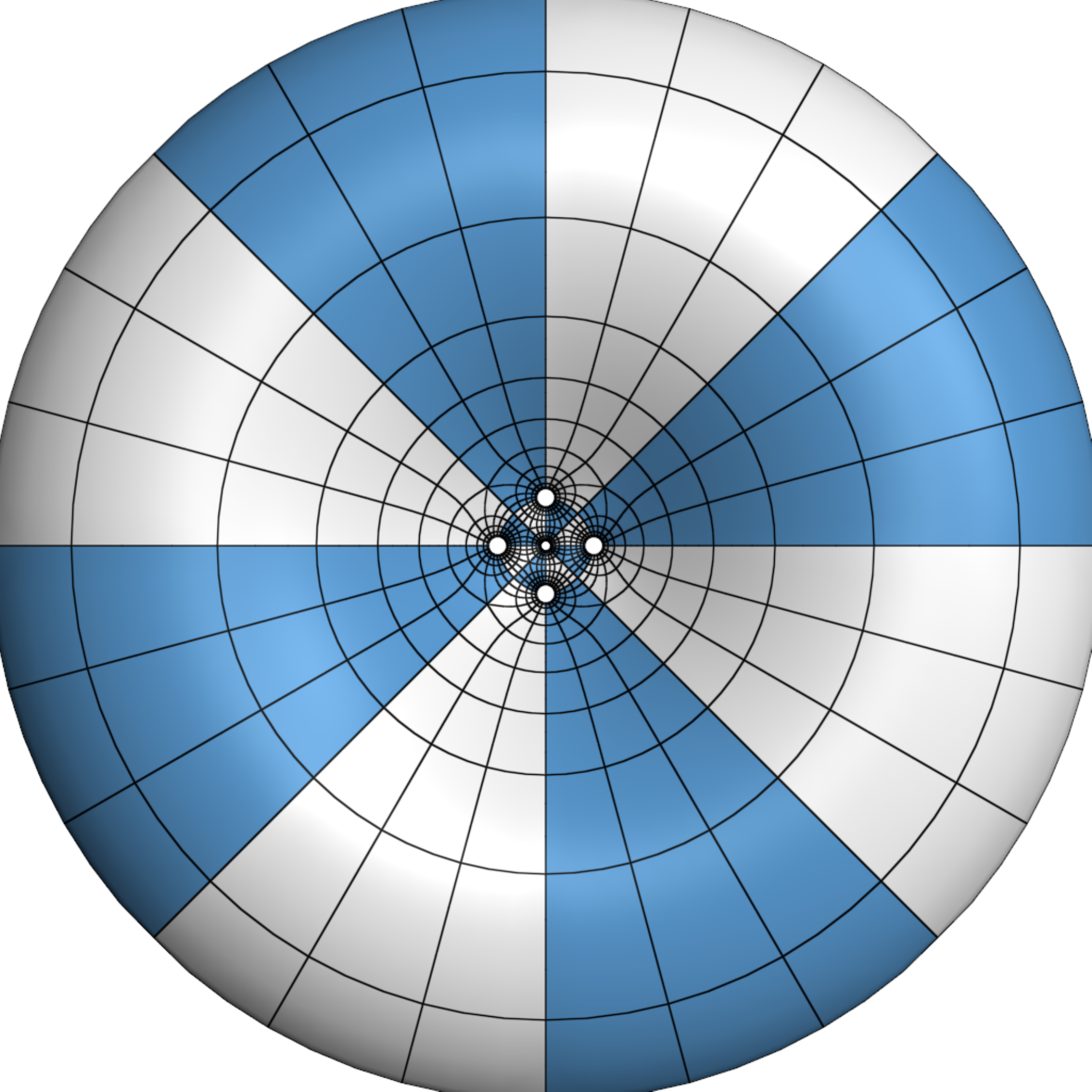}
  \includegraphics[width=0.375\textwidth]{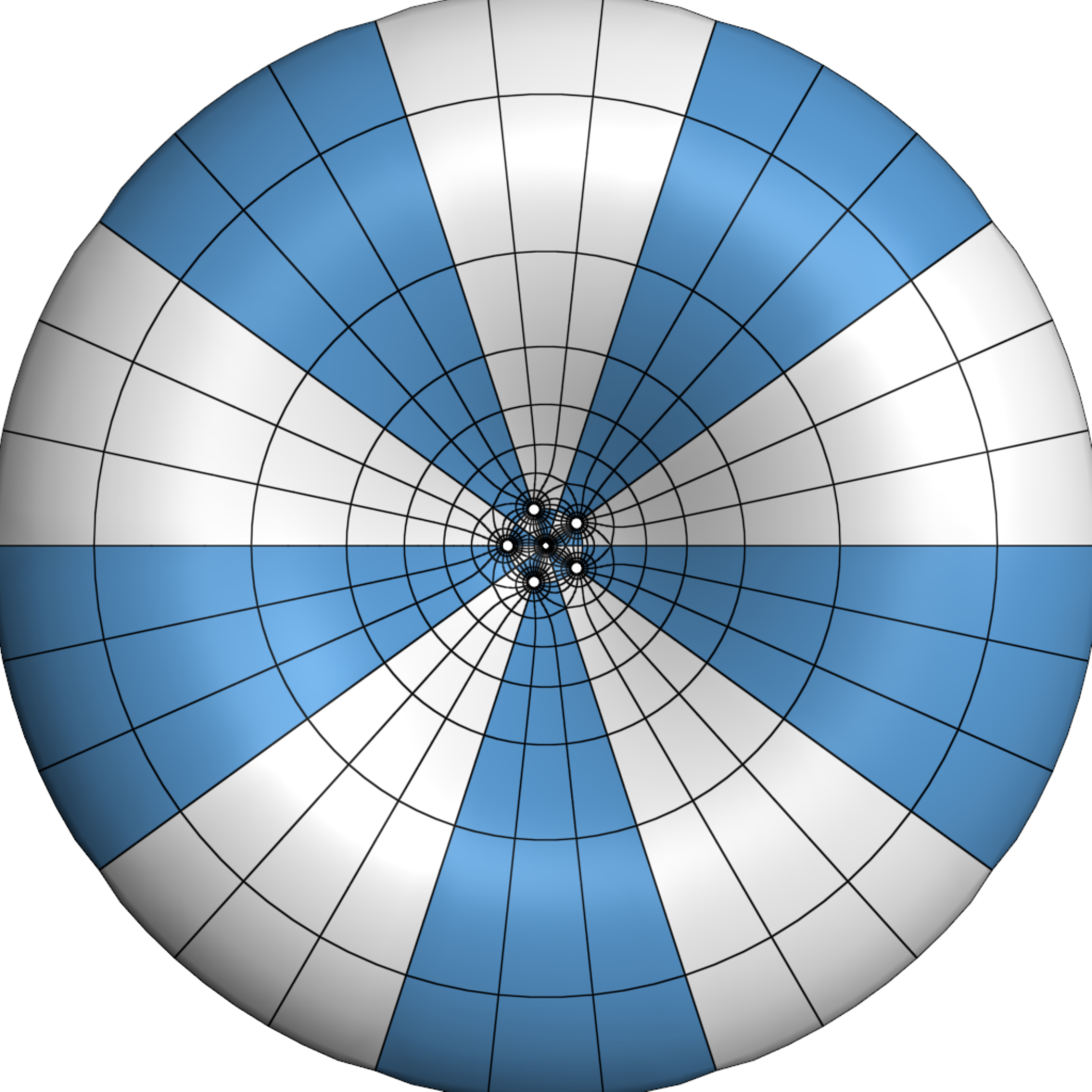}
  \caption{The dihedral family $\Bsurface{2}{n}$ for $n=2,\dots, 5$.}
\end{figure}

  Going along $\gamma$ counterclockwise, let
  \begin{itemize}
  \item
    $\alpha$: the sum of the turning angles at the
    vertices of $\gamma$;
    each angle is $\pm\pi/2$;
  \item
    $\beta$: the sum of the turning angles of
    the edges of $\gamma$.
  \end{itemize}
  Let
  $\kappa_0$,
  $\kappa_1$ and
  $\kappa_2$
  be the contributions to the
  umbilic excess $\kappa$ of $P$ at
  the vertex, edge and face umbilics
  respectively, so $\kappa=\kappa_0+\kappa_1+\kappa_2$.

  Since $\gamma$ is a simple closed curve,
  \begin{equation}
    \label{eq;umbilic-excess-1}
    \alpha + \beta = 2\pi
    \spaceperiod
  \end{equation}

  Each vertex of $P$ with umbilic excess $\mathring{\kappa}$
  contributes
  $\pi(\half - \mathring{\kappa})$
  to $\alpha$
  (see the first row of
  ~\eqref{eq:umbilic-excess-part}),
  so the $p$ vertices of $P$
  contribute
  $\pi(\tfrac{p}{2} - \kappa_0)$
  to $\alpha$.

  Each edge umbilic of $P$ with umbilic excess
  $\mathring{\kappa}$
  contributes
  $-\pi\mathring{\kappa}$
  to $\alpha$
  (see the second row of
  ~\eqref{eq:umbilic-excess-part}),
  so the  edge umbilics of $P$
  contribute in total
  $-\pi\kappa_1$
  to $\alpha$.

  Hence
  \begin{equation}
    \label{eq;umbilic-excess-2}
    \alpha = \pi(\tfrac{p}{2} - \kappa_0 - \kappa_1)
    \spaceperiod
  \end{equation}

  Since $\gamma$ encloses the face umbilics, then
  \begin{equation}
    \label{eq;umbilic-excess-3}
    \beta = -\pi \kappa_2
    \spacecomma
  \end{equation}
  see Remark \ref{rem:ind}. The result follows from~\eqref{eq;umbilic-excess-1},~\eqref{eq;umbilic-excess-2}
  and~\eqref{eq;umbilic-excess-3},
\end{proof}

\begin{corollary}
  \label{cor:umbilic-excess}
  If $P$ is a curvature line $p$-gon
  (Definition \ref{def:curvature-line-polygon}),
  then $p\ge 4$.
\end{corollary}

\begin{proof}
  This follows from Theorem \ref{thm:umbilic-excess}
  and the fact that $\kappa\ge 0$.
\end{proof}

\begin{theorem}
  \label{thm:finite-curvature-lines}
  \theoremname{Finite curvature lines}
  Let $D$ be a compact topological disk with boundary
  on which every umbilic is a star umbilic.
  Then every curvature line starting at a point $p$
  in the interior $D$, extended in either direction,
  either (a) ends at an umbilic in $D$,
  or (b) exits through the boundary of $D$.
\end{theorem}

\begin{proof}
  Let $\gamma$ be a curvature line starting at $p$,
  extended in one direction,
  which does not end at an umbilic in $D$.
  If $\gamma$ does not exit through the boundary of $D$,
  then $\gamma$ has a limit point $L$ in $D$
  because $D$ is compact.

  In the case $L$ is in the interior of $D$
  and is not an umbilic, then near $L$ the curvature
  lines make a checkerboard pattern: the gray
  lines in~\eqref{eq:finite-curvature-lines}.
  Therefore there exists a $2$-gon,
  one of whose edges is a segment of $\gamma$,
  the other a crossing curvature line,
  which bounds a disk as shown:
  \begin{equation}
    \label{eq:finite-curvature-lines}
    \fontsize{7}{8}
  \def\svgwidth{0.25\textwidth}%
\begingroup%
  \makeatletter%
  \providecommand\color[2][]{%
    \errmessage{(Inkscape) Color is used for the text in Inkscape, but the package 'color.sty' is not loaded}%
    \renewcommand\color[2][]{}%
  }%
  \providecommand\transparent[1]{%
    \errmessage{(Inkscape) Transparency is used (non-zero) for the text in Inkscape, but the package 'transparent.sty' is not loaded}%
    \renewcommand\transparent[1]{}%
  }%
  \providecommand\rotatebox[2]{#2}%
  \newcommand*\fsize{\dimexpr\f@size pt\relax}%
  \newcommand*\lineheight[1]{\fontsize{\fsize}{#1\fsize}\selectfont}%
  \ifx\svgwidth\undefined%
    \setlength{\unitlength}{66.59794607bp}%
    \ifx\svgscale\undefined%
      \relax%
    \else%
      \setlength{\unitlength}{\unitlength * \real{\svgscale}}%
    \fi%
  \else%
    \setlength{\unitlength}{\svgwidth}%
  \fi%
  \global\let\svgwidth\undefined%
  \global\let\svgscale\undefined%
  \makeatother%
  \begin{picture}(1,0.80641327)%
    \lineheight{1}%
    \setlength\tabcolsep{0pt}%
    \put(0,0){\includegraphics[width=\unitlength,page=1]{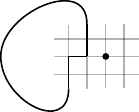}}%
  \end{picture}%
\endgroup%

  \end{equation}
  This contradicts Corollary~\ref{cor:umbilic-excess}.

  The proofs for other cases ($L$ is in the interior of $D$
  and is a star umbilic,
  or $L$ is on the boundary of $D$ and is either a non-umbilic
  or is a star umbilic) are similar.
\end{proof}

\subsection{Foliations}

\begin{theorem}
  \label{thm:foliation}
  \theoremname{Quadrilaterals and pentagons foliations}
  The curvature line foliations of
  quadrilaterals or pentagons bounded by curvature lines
  with star umbilics are as shown:
  \begin{statictable}
    $
    {\fontsize{7}{8}
  \def\svgwidth{0.5\textwidth}%
\begingroup%
  \makeatletter%
  \providecommand\color[2][]{%
    \errmessage{(Inkscape) Color is used for the text in Inkscape, but the package 'color.sty' is not loaded}%
    \renewcommand\color[2][]{}%
  }%
  \providecommand\transparent[1]{%
    \errmessage{(Inkscape) Transparency is used (non-zero) for the text in Inkscape, but the package 'transparent.sty' is not loaded}%
    \renewcommand\transparent[1]{}%
  }%
  \providecommand\rotatebox[2]{#2}%
  \newcommand*\fsize{\dimexpr\f@size pt\relax}%
  \newcommand*\lineheight[1]{\fontsize{\fsize}{#1\fsize}\selectfont}%
  \ifx\svgwidth\undefined%
    \setlength{\unitlength}{91.14234339bp}%
    \ifx\svgscale\undefined%
      \relax%
    \else%
      \setlength{\unitlength}{\unitlength * \real{\svgscale}}%
    \fi%
  \else%
    \setlength{\unitlength}{\svgwidth}%
  \fi%
  \global\let\svgwidth\undefined%
  \global\let\svgscale\undefined%
  \makeatother%
  \begin{picture}(1,0.36881884)%
    \lineheight{1}%
    \setlength\tabcolsep{0pt}%
    \put(0,0){\includegraphics[width=\unitlength,page=1]{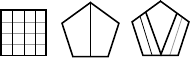}}%
    \put(0.07807849,0.01594898){\color[rgb]{0,0,0}\makebox(0,0)[lt]{\lineheight{1.25}\smash{\begin{tabular}[t]{l}$(J)$\end{tabular}}}}%
    \put(0.4287801,0.00821751){\color[rgb]{0,0,0}\makebox(0,0)[lt]{\lineheight{1.25}\smash{\begin{tabular}[t]{l}$(K)$\end{tabular}}}}%
    \put(0.81011188,0.00795694){\color[rgb]{0,0,0}\makebox(0,0)[lt]{\lineheight{1.25}\smash{\begin{tabular}[t]{l}$(L)$\end{tabular}}}}%
    \put(0,0){\includegraphics[width=\unitlength,page=2]{polygon-foliations_svg-tex.pdf}}%
  \end{picture}%
\endgroup%
}
    $
    \captionof{table}{
      \label{tab:foliation}
      (J): foliation of $4$-gon by curvature lines.
      (K): foliation of a $5$-gon
      by curvature lines, with the extra umbilic at a vertex.
      (K): foliation of a $5$-gon
      by curvature lines, with the extra umbilic on an edge.}
  \end{statictable}\label{poly-foli}
\end{theorem}

\begin{proof}
  Let horizontal and vertical curvature lines
  be called $\alpha$ and $\beta$ respectively.
  If two curvature lines meet at a non-umbilic,
  then one is of type $\alpha$ and the other of type $\beta$.

  \begin{statictable}
    \vspace{1ex}
    \fontsize{7}{8}
  \def\svgwidth{0.55\textwidth}%
\begingroup%
  \makeatletter%
  \providecommand\color[2][]{%
    \errmessage{(Inkscape) Color is used for the text in Inkscape, but the package 'color.sty' is not loaded}%
    \renewcommand\color[2][]{}%
  }%
  \providecommand\transparent[1]{%
    \errmessage{(Inkscape) Transparency is used (non-zero) for the text in Inkscape, but the package 'transparent.sty' is not loaded}%
    \renewcommand\transparent[1]{}%
  }%
  \providecommand\rotatebox[2]{#2}%
  \newcommand*\fsize{\dimexpr\f@size pt\relax}%
  \newcommand*\lineheight[1]{\fontsize{\fsize}{#1\fsize}\selectfont}%
  \ifx\svgwidth\undefined%
    \setlength{\unitlength}{310.63040969bp}%
    \ifx\svgscale\undefined%
      \relax%
    \else%
      \setlength{\unitlength}{\unitlength * \real{\svgscale}}%
    \fi%
  \else%
    \setlength{\unitlength}{\svgwidth}%
  \fi%
  \global\let\svgwidth\undefined%
  \global\let\svgscale\undefined%
  \makeatother%
  \begin{picture}(1,0.29103652)%
    \lineheight{1}%
    \setlength\tabcolsep{0pt}%
    \put(0.45430652,0.21010664){\color[rgb]{0,0,0}\makebox(0,0)[lt]{\lineheight{1.25}\smash{\begin{tabular}[t]{l}$\alpha$\end{tabular}}}}%
    \put(0.59268705,0.13129035){\color[rgb]{0,0,0}\makebox(0,0)[lt]{\lineheight{1.25}\smash{\begin{tabular}[t]{l}$\alpha$\end{tabular}}}}%
    \put(0.35762694,0.13053228){\color[rgb]{0,0,0}\makebox(0,0)[lt]{\lineheight{1.25}\smash{\begin{tabular}[t]{l}$\alpha$\end{tabular}}}}%
    \put(0.12909404,0.04213799){\color[rgb]{0,0,0}\makebox(0,0)[lt]{\lineheight{1.25}\smash{\begin{tabular}[t]{l}$\alpha$\end{tabular}}}}%
    \put(0.11971676,0.27684527){\color[rgb]{0,0,0}\makebox(0,0)[lt]{\lineheight{1.25}\smash{\begin{tabular}[t]{l}$\alpha$\end{tabular}}}}%
    \put(0.55224786,0.25071657){\color[rgb]{0,0,0}\makebox(0,0)[lt]{\lineheight{1.25}\smash{\begin{tabular}[t]{l}$\beta$\end{tabular}}}}%
    \put(0.39956602,0.25015914){\color[rgb]{0,0,0}\makebox(0,0)[lt]{\lineheight{1.25}\smash{\begin{tabular}[t]{l}$\beta$\end{tabular}}}}%
    \put(-0.00091613,0.16147651){\color[rgb]{0,0,0}\makebox(0,0)[lt]{\lineheight{1.25}\smash{\begin{tabular}[t]{l}$\beta$\end{tabular}}}}%
    \put(0.24705537,0.16075196){\color[rgb]{0,0,0}\makebox(0,0)[lt]{\lineheight{1.25}\smash{\begin{tabular}[t]{l}$\beta$\end{tabular}}}}%
    \put(0.47237732,0.04115718){\color[rgb]{0,0,0}\makebox(0,0)[lt]{\lineheight{1.25}\smash{\begin{tabular}[t]{l}$\beta$\end{tabular}}}}%
    \put(0.46946415,0.00233465){\color[rgb]{0,0,0}\makebox(0,0)[lt]{\lineheight{1.25}\smash{\begin{tabular}[t]{l}$(K)$\end{tabular}}}}%
    \put(0.11883853,0.0034096){\color[rgb]{0,0,0}\makebox(0,0)[lt]{\lineheight{1.25}\smash{\begin{tabular}[t]{l}$(J)$\end{tabular}}}}%
    \put(0,0){\includegraphics[width=\unitlength,page=1]{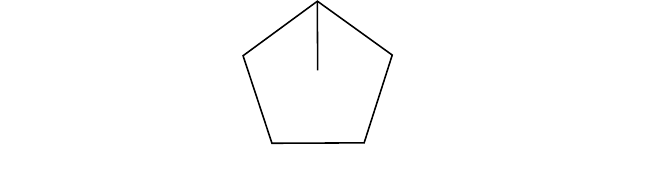}}%
    \put(0.79772815,0.10233309){\color[rgb]{0,0,0}\makebox(0,0)[lt]{\lineheight{1.25}\smash{\begin{tabular}[t]{l}$\alpha$\end{tabular}}}}%
    \put(0.70283829,0.13099865){\color[rgb]{0,0,0}\makebox(0,0)[lt]{\lineheight{1.25}\smash{\begin{tabular}[t]{l}$\alpha$\end{tabular}}}}%
    \put(0.89555834,0.25001527){\color[rgb]{0,0,0}\makebox(0,0)[lt]{\lineheight{1.25}\smash{\begin{tabular}[t]{l}$\alpha$\end{tabular}}}}%
    \put(0.86769466,0.04274402){\color[rgb]{0,0,0}\makebox(0,0)[lt]{\lineheight{1.25}\smash{\begin{tabular}[t]{l}$\alpha$\end{tabular}}}}%
    \put(0.86948744,0.10195867){\color[rgb]{0,0,0}\makebox(0,0)[lt]{\lineheight{1.25}\smash{\begin{tabular}[t]{l}$\beta$\end{tabular}}}}%
    \put(0.79112402,0.04071097){\color[rgb]{0,0,0}\makebox(0,0)[lt]{\lineheight{1.25}\smash{\begin{tabular}[t]{l}$\beta$\end{tabular}}}}%
    \put(0.94738166,0.13046596){\color[rgb]{0,0,0}\makebox(0,0)[lt]{\lineheight{1.25}\smash{\begin{tabular}[t]{l}$\beta$\end{tabular}}}}%
    \put(0.745509,0.25024063){\color[rgb]{0,0,0}\makebox(0,0)[lt]{\lineheight{1.25}\smash{\begin{tabular}[t]{l}$\beta$\end{tabular}}}}%
    \put(0.82502807,0.00268658){\color[rgb]{0,0,0}\makebox(0,0)[lt]{\lineheight{1.25}\smash{\begin{tabular}[t]{l}$(L)$\end{tabular}}}}%
    \put(0,0){\includegraphics[width=\unitlength,page=2]{curvature-line-orientations_svg-tex.pdf}}%
  \end{picture}%
\endgroup%

    \captionof{table}{
             \label{tab:curvature-line-orientations}
             Curvature lines
             on quadrilaterals and pentagons,
             in the two curvature line foliations
             $\alpha$ and $\beta$.
           }
  \end{statictable}

  \textit{Quadrilaterals.}
  Let $P$ be a quadrilateral bounded by curvature lines.
  By Theorem \ref{thm:umbilic-excess}, the umbilic excess is $\kappa=0$.
  Hence
  \begin{itemize}
  \item
    there are no curvature lines in the interior
    of $P$ starting at a vertex,
  \item
    there are no umbilics on the edges of $P$, and
  \item
    there are no umbilics in the interior of $P$.
  \end{itemize}
  Hence the edges of $P$ are cyclically of type
  $\alpha$, $\beta$, $\alpha$, $\beta$,
  as shown in Table \ref{tab:curvature-line-orientations}(J).

  Let $\gamma$ be a curvature line of type $\alpha$
  through a point in the interior of $P$.
  Then $\gamma$ does not end at a vertex of $P$, so
  by Theorem \ref{thm:finite-curvature-lines},
  $\gamma$ starts and ends 
  at edges of type $\beta$.
  By Corollary \ref{cor:umbilic-excess},
  $\gamma$ cannot connect an edge to itself;
  otherwise there is a $2$-gon bounding a disk.
  Hence $\gamma$ connects opposite edges of $P$
  in a ``checkerboard'' pattern
  as in Table \ref{tab:foliation}(J).

  \textit{Pentagons.}
  Let $P$ be a quadrilateral bounded by curvature lines.
  By Theorem \ref{thm:umbilic-excess}, the umbilic excess is $\kappa=1/2$.
  Hence
  \begin{itemize}
  \item
    there are no umbilics in the interior of $P$,
  \item
    either
    there is one vertex $v$ of $P$ from which one curvature line
    emanates into the interior $P$,
    as in Table \ref{tab:foliation}(K),
    or
  \item
    there is a simple umbilic on an edge of $P$,
    as in Table \ref{tab:foliation}(L).
  \end{itemize}

  Note that two intersecting curvature lines $\gamma_1$ and $\gamma_2$ have opposite types (if the point of intersection is not an umbilic):
  one is of type $\alpha$ and the other of type $\beta$.
  Hence $\gamma_1$ and $\gamma_2$ exit the pentagon
  through edges of type $\beta$ and $\alpha$ respectively.

  In case $K$, let $\gamma$ be the curvature line through
  the vertex $v$ of $P$ into the interior of $P$,
  marked as type $\alpha$,
  where the edges of the pentagon  are marked
  as shown in Table \ref{tab:curvature-line-orientations}(K).

  By Theorem \ref{thm:finite-curvature-lines},
  $\gamma$ ends at a boundary edge of $P$.
  Since $\gamma$ is of type $\alpha$, this edge is
  of type $\beta$.
  Since $2$-gons bounding a disk
  are excluded by  Corollary \ref{cor:umbilic-excess},
  $\gamma$ ends at the edge of $P$
  opposite to $v$.
  This divides $P$ into two quadrilaterals,
  each of which has a checkerboard
  curvature line pattern as shown in Table \ref{tab:foliation}(K).

  In case $L$, let $\gamma_1$ and $\gamma_2$
  be the two curvature lines emanating from
  the edge umbilic into the interior of $P$,
  marked as $\alpha$ and $\beta$ respectively,
  as shown in Table  \ref{tab:curvature-line-orientations}(L).

  It follows that
  $\gamma_1$ and $\gamma_2$
  are as in Table \ref{tab:foliation}(L);
  otherwise $\gamma_1$ and $\gamma_2$ would intersect,
  bounding a $2$-gon, in contradiction to  Corollary \ref{cor:umbilic-excess}.
  This divides $P$ into three quadrilaterals,
  each of which has a checkerboard
  curvature line pattern as shown in Table \ref{tab:foliation}(L).
\end{proof}

\typeout{== figure/double-pyramid-a.tex ============================================}\begin{figure}[b]
  \includegraphics[width=0.375\textwidth]{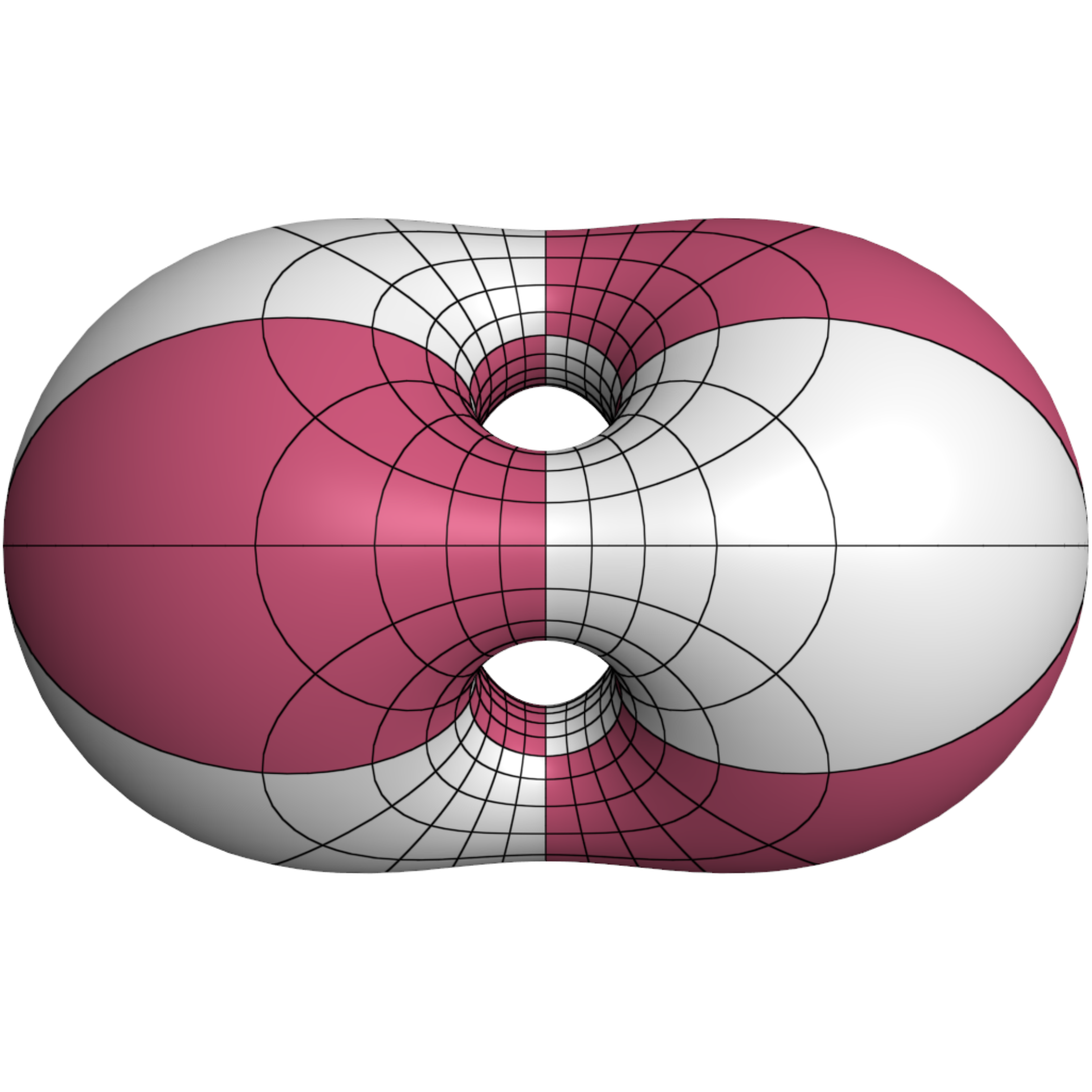}
  \includegraphics[width=0.375\textwidth]{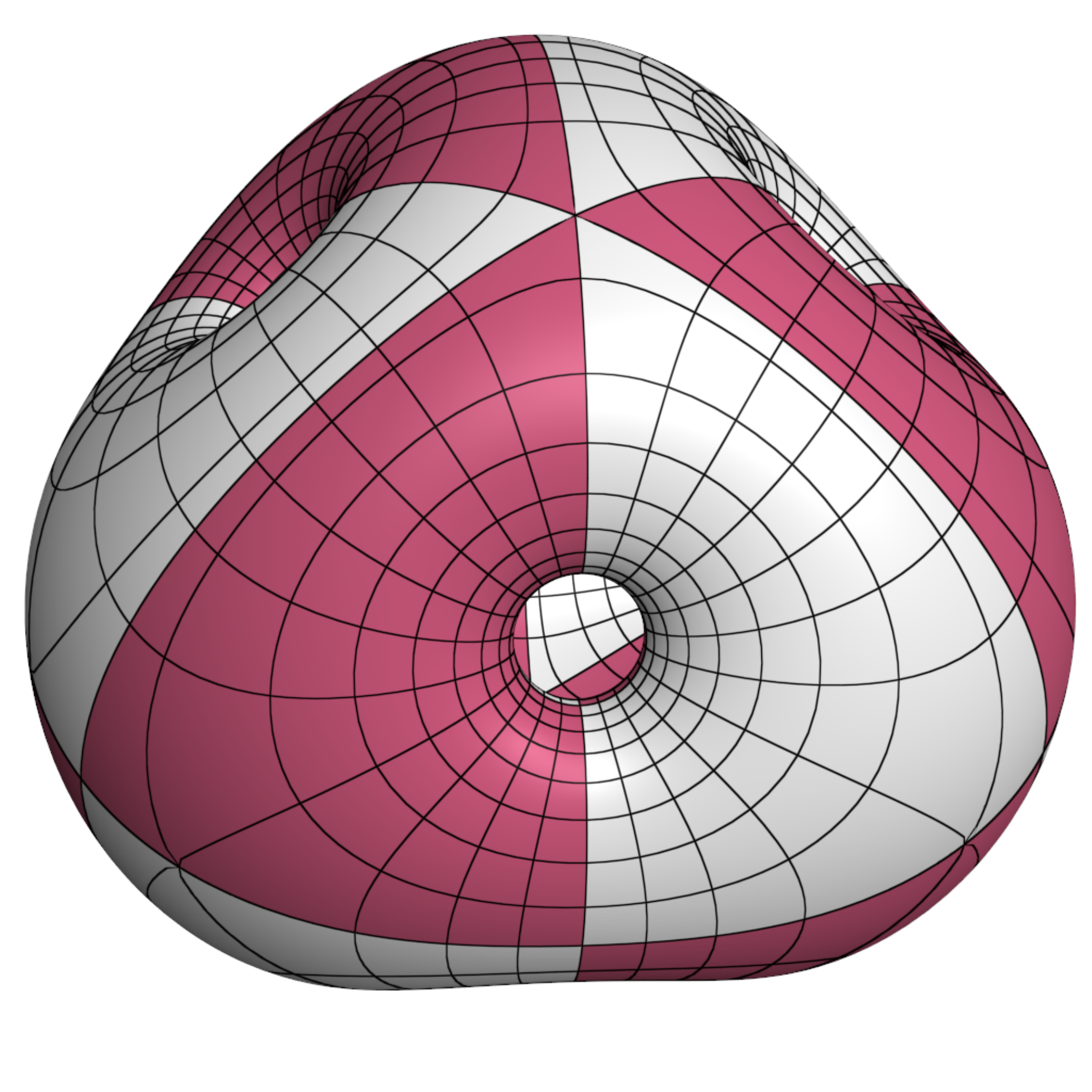}
  \includegraphics[width=0.375\textwidth]{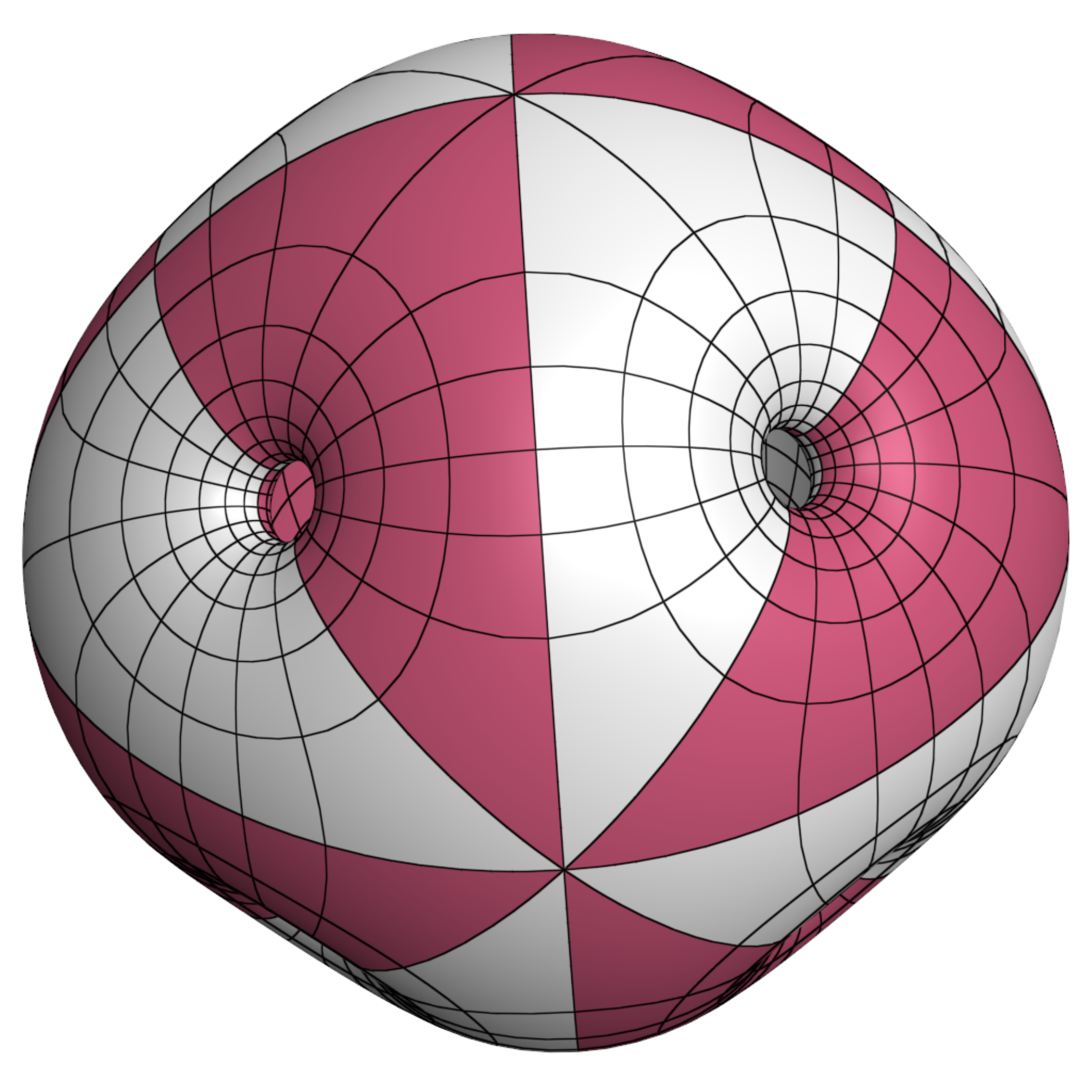}
  \includegraphics[width=0.375\textwidth]{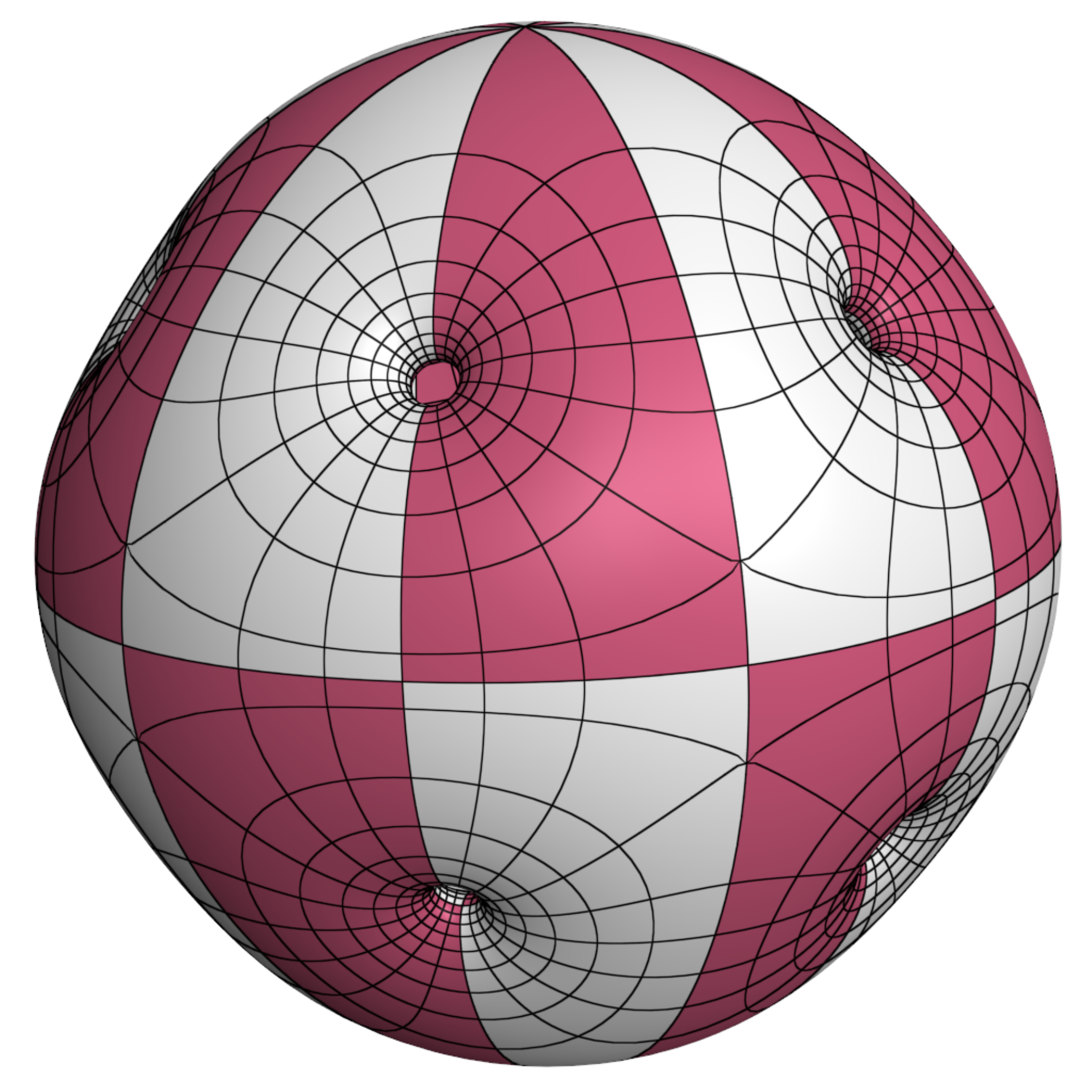}
  \caption{The dihedral family $\Bsurface{n}{2}$ for $n=2,\dots, 5$.}\label{fig:bn2}
\end{figure}

\subsection{Quadrilateral decomposition and closed curvature lines}

A \emph{curvature line quadrilateral decomposition} of a surface $S$ is
a 
cell decomposition of $S$
such that each face is a quadrilateral,
each edge is a curvature line 
and each vertex has an even number of edges emanating from it.

\begin{theorem}
  \label{thm:quadrilateral-decomposition}
  \theoremname{Quadrilateral decomposition}
  Every reflection surface has a curvature line
  quadrilateral decomposition.
\end{theorem}

\begin{proof}
  Let $P$ be a fundamental polygon.
  In $P$, draw every curvature line
  starting at an umbilic.
  This decomposes $P$ into finitely many curvature line 
  quadrilaterals
  (see Definition \ref{def:curvature-line-polygon}).

  Indeed, each sub-polygon satisfies:
  (a) $P$ has no umbilics in its interior
  or on its edges,
  and
  (b) no curvature line in the interior of $P$
  ends at a vertex of $P$.
  Hence the umbilic excess of each sub-polygon is $0$.
  By Theorem \ref{thm:umbilic-excess},
  each sub-polygon is a quadrilateral.
\end{proof}


A reflection surface has \emph{closed curvature lines} if
\begin{itemize}
\item
  every curvature line which starts at an umbilic ends at an umbilic.
\item
  every curvature line which does not pass through an umbilic
  is closed.
\end{itemize}

\begin{theorem}
  \label{thm:closed-curvature-lines}
  \theoremname{Closed curvature lines}
  Let $S$ be a reflection surface in $\bbS^3$.
  Then $S$ has closed curvature lines.
\end{theorem}

\begin{proof}
  Let $R$ be a fundamental polyhedron and
  $P$ be a fundamental polygon.
 In $\bbS^3$,
  every two faces of $R$ meet at internal dihedral angle $\pi/n$ for some $n\in\bbN_{\ge 2}$.

  Let $\gamma$ be a curvature line in $P\cup\del P$.

  Case 1: $\gamma$ does not contain any umbilics.
  By  Theorem \ref{thm:finite-curvature-lines},
  $\gamma$ connects two edges $e_1$ and $e_2$ of $P$,
  meeting them perpendicularly.
  Let $f_1$ and $f_2$ be the two faces of the
  fundamental polyhedron containing $e_1$
  and $e_2$ respectively.

  Let $H$ be the dihedral group generated by
  reflections in $f_1$ and $f_2$, and let
  $\delta$ be the orbit of $\gamma$ under this group action.
  Then $\delta$ is a smooth closed curvature line.

  Case 2: $\gamma$ connects an umbilic to an edge $e$.
  Since the surface reflects in the edge $e$,
  $\gamma$ reflects in the edge smoothly to $\gamma^\prime$.
  The union $\gamma\cup\gamma^\prime$
  is a smooth curvature line connecting two umbilics.

  Case 3: $\gamma$ connects two umbilics.
  In this case, we are done.

  Case 4: $\gamma$ is an edge $e$ of $P$.
  Let $v_1$ and $v_2$ be the endpoints of $e$, i.e.
  vertices of $P$.
  The proof is as above in the three different cases,
  according to whether $e_1$ and $e_2$ are umbilics.\end{proof}

\subsection{Two dihedral families of reflection surfaces}\label{ss:2dihfam}

Define the
two $2$-integer families
of reflection surface types
\begin{equation}
  \Asurface{k}{\ell}
  \quad
  (k,\,\ell)\in\bbN_{\ge 2}\cross \bbN_{\ge 2}
  \qquad\text{and}\qquad
  \Bsurface{k}{\ell}
  \quad
  (k,\,\ell)\in\bbN_{\ge 2}\cross \bbN_{\ge 1}
\end{equation}
with dihedral symmetry 
and respective fundamental quadrilateral and pentagon,
as shown in Table \ref{tab:definition-a-b}:
\begin{statictable}
  $
  \fontsize{7}{8}
  \begin{array}{cccc}
  \def\svgwidth{0.1\textwidth}%
\begingroup%
  \makeatletter%
  \providecommand\color[2][]{%
    \errmessage{(Inkscape) Color is used for the text in Inkscape, but the package 'color.sty' is not loaded}%
    \renewcommand\color[2][]{}%
  }%
  \providecommand\transparent[1]{%
    \errmessage{(Inkscape) Transparency is used (non-zero) for the text in Inkscape, but the package 'transparent.sty' is not loaded}%
    \renewcommand\transparent[1]{}%
  }%
  \providecommand\rotatebox[2]{#2}%
  \newcommand*\fsize{\dimexpr\f@size pt\relax}%
  \newcommand*\lineheight[1]{\fontsize{\fsize}{#1\fsize}\selectfont}%
  \ifx\svgwidth\undefined%
    \setlength{\unitlength}{89.44568722bp}%
    \ifx\svgscale\undefined%
      \relax%
    \else%
      \setlength{\unitlength}{\unitlength * \real{\svgscale}}%
    \fi%
  \else%
    \setlength{\unitlength}{\svgwidth}%
  \fi%
  \global\let\svgwidth\undefined%
  \global\let\svgscale\undefined%
  \makeatother%
  \begin{picture}(1,0.86602541)%
    \lineheight{1}%
    \setlength\tabcolsep{0pt}%
    \put(0,0){\includegraphics[width=\unitlength,page=1]{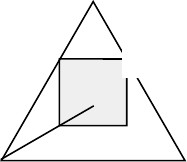}}%
    \put(0.65587314,0.47084927){\color[rgb]{0,0,0}\makebox(0,0)[lt]{\lineheight{1.25}\smash{\begin{tabular}[t]{l}$\ell$\end{tabular}}}}%
    \put(0,0){\includegraphics[width=\unitlength,page=2]{definition-4_svg-tex.pdf}}%
    \put(0.13580098,0.05379311){\color[rgb]{0,0,0}\makebox(0,0)[lt]{\lineheight{1.25}\smash{\begin{tabular}[t]{l}$k$\end{tabular}}}}%
    \put(0,0){\includegraphics[width=\unitlength,page=3]{definition-4_svg-tex.pdf}}%
  \end{picture}%
\endgroup%

    &
    \phantom{\qquad\qquad}&
  \def\svgwidth{0.1\textwidth}%
\begingroup%
  \makeatletter%
  \providecommand\color[2][]{%
    \errmessage{(Inkscape) Color is used for the text in Inkscape, but the package 'color.sty' is not loaded}%
    \renewcommand\color[2][]{}%
  }%
  \providecommand\transparent[1]{%
    \errmessage{(Inkscape) Transparency is used (non-zero) for the text in Inkscape, but the package 'transparent.sty' is not loaded}%
    \renewcommand\transparent[1]{}%
  }%
  \providecommand\rotatebox[2]{#2}%
  \newcommand*\fsize{\dimexpr\f@size pt\relax}%
  \newcommand*\lineheight[1]{\fontsize{\fsize}{#1\fsize}\selectfont}%
  \ifx\svgwidth\undefined%
    \setlength{\unitlength}{74.13105bp}%
    \ifx\svgscale\undefined%
      \relax%
    \else%
      \setlength{\unitlength}{\unitlength * \real{\svgscale}}%
    \fi%
  \else%
    \setlength{\unitlength}{\svgwidth}%
  \fi%
  \global\let\svgwidth\undefined%
  \global\let\svgscale\undefined%
  \makeatother%
  \begin{picture}(1,0.79475608)%
    \lineheight{1}%
    \setlength\tabcolsep{0pt}%
    \put(0,0){\includegraphics[width=\unitlength,page=1]{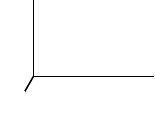}}%
    \put(0,0.02323814){\makebox(0,0)[lt]{\lineheight{1.25}\smash{\begin{tabular}[t]{l}$k$\end{tabular}}}}%
    \put(0,0){\includegraphics[width=\unitlength,page=2]{definition-5a_svg-tex.pdf}}%
  \end{picture}%
\endgroup%

    &
  \def\svgwidth{0.1\textwidth}%
\begingroup%
  \makeatletter%
  \providecommand\color[2][]{%
    \errmessage{(Inkscape) Color is used for the text in Inkscape, but the package 'color.sty' is not loaded}%
    \renewcommand\color[2][]{}%
  }%
  \providecommand\transparent[1]{%
    \errmessage{(Inkscape) Transparency is used (non-zero) for the text in Inkscape, but the package 'transparent.sty' is not loaded}%
    \renewcommand\transparent[1]{}%
  }%
  \providecommand\rotatebox[2]{#2}%
  \newcommand*\fsize{\dimexpr\f@size pt\relax}%
  \newcommand*\lineheight[1]{\fontsize{\fsize}{#1\fsize}\selectfont}%
  \ifx\svgwidth\undefined%
    \setlength{\unitlength}{89.44568722bp}%
    \ifx\svgscale\undefined%
      \relax%
    \else%
      \setlength{\unitlength}{\unitlength * \real{\svgscale}}%
    \fi%
  \else%
    \setlength{\unitlength}{\svgwidth}%
  \fi%
  \global\let\svgwidth\undefined%
  \global\let\svgscale\undefined%
  \makeatother%
  \begin{picture}(1,0.86602541)%
    \lineheight{1}%
    \setlength\tabcolsep{0pt}%
    \put(0,0){\includegraphics[width=\unitlength,page=1]{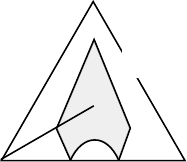}}%
    \put(0.65587314,0.47084927){\color[rgb]{0,0,0}\makebox(0,0)[lt]{\lineheight{1.25}\smash{\begin{tabular}[t]{l}$\ell$\end{tabular}}}}%
    \put(0,0){\includegraphics[width=\unitlength,page=2]{definition-5b_svg-tex.pdf}}%
    \put(0.13580098,0.05379311){\color[rgb]{0,0,0}\makebox(0,0)[lt]{\lineheight{1.25}\smash{\begin{tabular}[t]{l}$k$\end{tabular}}}}%
    \put(0,0){\includegraphics[width=\unitlength,page=3]{definition-5b_svg-tex.pdf}}%
  \end{picture}%
\endgroup%

    \\
    \Asurface{k}{\ell} &
    \phantom{\qquad\qquad}& \Bsurface{k}{1} & \Bsurface{k}{\ell}\ (\ell \ge 2)
  \end{array}
  $
  \captionof{table}{Two $2$-integer families
    $\Asurface{k}{\ell}$
    and $\Bsurface{k}{\ell}$
    with dihedral symmetry 
    and fundamental quadrilateral and pentagon respectively.
  }
  \label{tab:definition-a-b}
\end{statictable}
Then:
\begin{itemize}
\item
  $\Asurface{k}{\ell}$ is symmetric in $k$ and $\ell$,
  while $\Bsurface{k}{\ell}$ is not.
\item
  $\genus(\Asurface{k}{\ell}) = (k-1)(\ell-1)$
  and
  $\genus(\Bsurface{k}{\ell}) = (k-1)(\ell-1) + k$.
\item
  The symmetry group of a reflection surface of type $\Asurface{k}{\ell}$
  is $D_k\cross D_\ell$, see Theorem \ref{thm:quad-classification} below.
  \item
  The symmetry group of a reflection surface of type $\Bsurface{k}{\ell}$
  has as subgroup $D_k\cross D_\ell$.
\item
  The Lawson surface $\xi_{k-1,\ell-1}$
  is a reflection surface of type $\Asurface{k}{\ell}$.
\end{itemize}


\begin{theorem}
  \label{thm:quad-classification}
  The reflection group of a surface of type $\Asurface{k}{\ell}$,
  $(k,\,\ell)\ne(2,\,2)$
  is exactly $D_k\cross D_\ell$.
\end{theorem}

\begin{proof}
  If the surface had a larger symmetry group, its fundamental
  quadrilateral would be tessellated into quadrilateral as in Table
  \ref{tab:foliation}(J).

  The vertex integers around the fundamental quadrilateral,
  in cyclic order, are
  $k,\,2,\,\ell,\,2$.
  Since these reflect, they introduce umbilic on the vertices, edges
  or interior of the
  quadrilateral,  which is impossible.
\end{proof}


The \emph{umbilic structure} of a reflection surface $S$
counts how many umbilics it has of what orders,
written (non-uniquely) as a finite formal sum
\begin{equation}
  \calU(S) \coloneq n_1[o_1] + \dots + n_s[o_s]
\end{equation}
Since the genus $g$ of $S$ is related to the umbilic structure by
\begin{equation}
  \textstyle
  4g-4 = \sum_{k=1}^s n_k o_k
  \spacecomma
\end{equation}
two surfaces with the same umbilic structure have
the same genus.

\begin{theorem}
  \label{thm:dihedral-classification}
  Every reflection surface of type
  \begin{equation}
    \{ \Bsurface{k}{\ell} \,\suchthat\, (k,\,\ell)\not\in \{ (2,\,1),\,(2,\,2)\}
  \end{equation}
  is different from every reflection surface of type $\Asurface{u}{v}$.
\end{theorem}

\begin{proof}
  The umbilic structures of a reflection surface $S$ of type $\Asurface{u}{v}$ is
  \begin{equation}
    \label{eq:A-umbilic-structure}
    \calU(\Asurface{u}{v}) = 2v[u-2] + 2u[v-2]
    \spaceperiod
  \end{equation}

  A surface of type $\Bsurface{k}{\ell}$ may have different
  umbilic structures depending on the position of the extra
  umbilic, denoted as follows:
  \begin{itemize}
  \item
    $\Bsurface{k}{\ell}^\alpha$ with extra umbilic on an edge
    of the pentagon,
  \item
    $\Bsurface{k}{\ell}^\beta$ with extra umbilic at a vertex
    of the pentagon
    with vertex integer $2$,
  \item
    $\Bsurface{k}{\ell}^\gamma$, $k>2$ with extra umbilic at the vertex
    of the pentagon
    with vertex integer $k$.
  \end{itemize}
  The umbilic structures of the reflection surfaces of type $\Bsurface{k}{\ell}^x$,
  $x\in\{\alpha,\,\beta,\,\gamma\}$
  are:
  \begin{equation}
    \label{eq:B-umbilic-structure}
    \calU(\Bsurface{k}{\ell}^\alpha) = 2k\ell[1] + 2\ell[k-2]
    \spacecomma\quad
    \calU(\Bsurface{k}{\ell}^\beta) = k\ell[2] + 2\ell[k-2]
    \spacecomma\quad
    \calU(\Bsurface{k}{\ell}^\gamma) = 2\ell[2k-2]
    \spaceperiod
  \end{equation}

  Comparing the umbilic structures~\eqref{eq:A-umbilic-structure}
  with~\eqref{eq:B-umbilic-structure},
  the reflection surfaces of type $\Asurface{k}{\ell}$
  and those of type $\Bsurface{u}{v}$ with the same umbilic
  structure are:
  \begin{equation}
    \label{eq:same-umbilic-structure}
    \begin{array}{c|c}
      \begin{array}{ll}
        \mathrm{(a)} &
        \calU(\Asurface{2}{3})
        = \calU(\Bsurface{2}{1}^\alpha)
        = 4[1]
        \\
        \mathrm{(b)} &
        \calU(\Asurface{2}{4})
        = \calU(\Bsurface{2}{2}^\beta)
        = 4[2]
        \\
        \mathrm{(c)} &
        \calU(\Asurface{3}{3})
        = \calU(\Bsurface{2}{3}^\alpha)
        = 12[1]
      \end{array}
      &
      \begin{array}{ll}
        \mathrm{(d)} &
        \calU(\Asurface{4}{4})
        = \calU(\Bsurface{2}{8}^\beta)
        = 16[2]
        \\
        \mathrm{(e)} &
        \calU(\Asurface{2k}{2k})
        = \calU(\Bsurface{k}{4k}^\gamma)
        = 8k[2k-2]
        \\
        \mathrm{(f)} &
        \calU(\Asurface{2k}{2})
        = \calU(\Bsurface{k}{2}^\gamma)
        = 4[2k-2]
      \end{array}
    \end{array}
  \end{equation}

  By Theorem \ref{thm:quad-classification}
  the symmetry group of reflection surfaces of type $\Asurface{u}{v}$, $(u,\,v)\ne(2,\,2)$
  is $D_u\cross D_v$.
  Hence if a reflection surface of type $\Bsurface{k}{\ell}$
  is of type $\Asurface{u}{v}$, then $D_k\cross D_\ell$
  is a subgroup of $D_u\cross D_v$, that is
  ($k|u$ and $\ell|v$) or
  ($k|v$ and $\ell|u$).
  This excludes the pairs \eqref{eq:same-umbilic-structure}(c), (d) and (e);
  that is for these pairs, the $B$ surface is different from the corresponding $A$ surface the same umbilic structure.

  To show the pairs \eqref{eq:same-umbilic-structure}(f)
  are different,
  assume $\Asurface{2k}{2}$ and $\Bsurface{k}{2}^\gamma$
  are the same.
  Since the extra umbilic of $\Bsurface{k}{2}^\gamma$
  is at the vertex marked $k$,
  by the curvature line foliation Theorem \ref{thm:foliation}(K),
  there is a curvature line in the pentagon
  from this vertex to the opposite side of the pentagon,
  and the surface reflects in this curvature line.
  On the other hand, the vertex and edge integers
  (Definition \ref{def:vertex-and-edge-integers}),
  of the fundamental pentagon are as shown:
  \begin{equation}
    {\fontsize{7}{8}
  \def\svgwidth{0.2\textwidth}%
\begingroup%
  \makeatletter%
  \providecommand\color[2][]{%
    \errmessage{(Inkscape) Color is used for the text in Inkscape, but the package 'color.sty' is not loaded}%
    \renewcommand\color[2][]{}%
  }%
  \providecommand\transparent[1]{%
    \errmessage{(Inkscape) Transparency is used (non-zero) for the text in Inkscape, but the package 'transparent.sty' is not loaded}%
    \renewcommand\transparent[1]{}%
  }%
  \providecommand\rotatebox[2]{#2}%
  \newcommand*\fsize{\dimexpr\f@size pt\relax}%
  \newcommand*\lineheight[1]{\fontsize{\fsize}{#1\fsize}\selectfont}%
  \ifx\svgwidth\undefined%
    \setlength{\unitlength}{89.53123602bp}%
    \ifx\svgscale\undefined%
      \relax%
    \else%
      \setlength{\unitlength}{\unitlength * \real{\svgscale}}%
    \fi%
  \else%
    \setlength{\unitlength}{\svgwidth}%
  \fi%
  \global\let\svgwidth\undefined%
  \global\let\svgscale\undefined%
  \makeatother%
  \begin{picture}(1,0.89924311)%
    \lineheight{1}%
    \setlength\tabcolsep{0pt}%
    \put(0.3141186,0.63346134){\color[rgb]{0,0,0}\makebox(0,0)[lt]{\lineheight{1.25}\smash{\begin{tabular}[t]{l}$2$\end{tabular}}}}%
    \put(0.15934794,0.00348475){\color[rgb]{0,0,0}\makebox(0,0)[lt]{\lineheight{1.25}\smash{\begin{tabular}[t]{l}$2$\end{tabular}}}}%
    \put(0.70858108,0.00297346){\color[rgb]{0,0,0}\makebox(0,0)[lt]{\lineheight{1.25}\smash{\begin{tabular}[t]{l}$2$\end{tabular}}}}%
    \put(0.91489739,0.52888607){\color[rgb]{0,0,0}\makebox(0,0)[lt]{\lineheight{1.25}\smash{\begin{tabular}[t]{l}$2$\end{tabular}}}}%
    \put(0.45366538,0.86033176){\color[rgb]{0,0,0}\makebox(0,0)[lt]{\lineheight{1.25}\smash{\begin{tabular}[t]{l}$2$\end{tabular}}}}%
    \put(0.61129786,0.64466109){\color[rgb]{0,0,0}\makebox(0,0)[lt]{\lineheight{1.25}\smash{\begin{tabular}[t]{l}$k$\end{tabular}}}}%
    \put(-0.00317853,0.53416999){\color[rgb]{0,0,0}\makebox(0,0)[lt]{\lineheight{1.25}\smash{\begin{tabular}[t]{l}$k$\end{tabular}}}}%
    \put(0.19491505,0.27525325){\color[rgb]{0,0,0}\makebox(0,0)[lt]{\lineheight{1.25}\smash{\begin{tabular}[t]{l}$2$\end{tabular}}}}%
    \put(0.43914715,0.10121976){\color[rgb]{0,0,0}\makebox(0,0)[lt]{\lineheight{1.25}\smash{\begin{tabular}[t]{l}$1$\end{tabular}}}}%
    \put(0,0){\includegraphics[width=\unitlength,page=1]{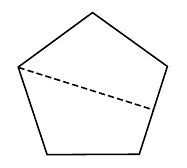}}%
    \put(0.71387299,0.29943021){\color[rgb]{0,0,0}\makebox(0,0)[lt]{\lineheight{1.25}\smash{\begin{tabular}[t]{l}$\ell$\end{tabular}}}}%
  \end{picture}%
\endgroup%

    }
  \end{equation}
  The edge integers
  do not reflect across the dotted curvature line,
  contradicting that
  $\Asurface{2k}{2}$ and $\Bsurface{k}{2}^\gamma$
  are the same.
\end{proof}

By Theorem \ref{thm:dihedral-classification},
a minimal surface of type $\Bsurface{k}{\ell}$
in $\bbS^3$ is not a Lawson surface $\xi_{ab}$.
The Lawson surface $\xi_{2,1}$ (see Figure \ref{fig:lawson}) is actually a minimal surface of type $\Bsurface{2}{1},$
see Figure \ref{fig:bn1}, and the Lawson surface $\xi_{3,1}$ is a minimal surface of type $\Bsurface{2}{2}^\beta,$
see Figure \ref{fig:bn2}.

\typeout{== figure/B41 ============================================}\begin{figure}[b]
  \includegraphics[width=0.375\textwidth]{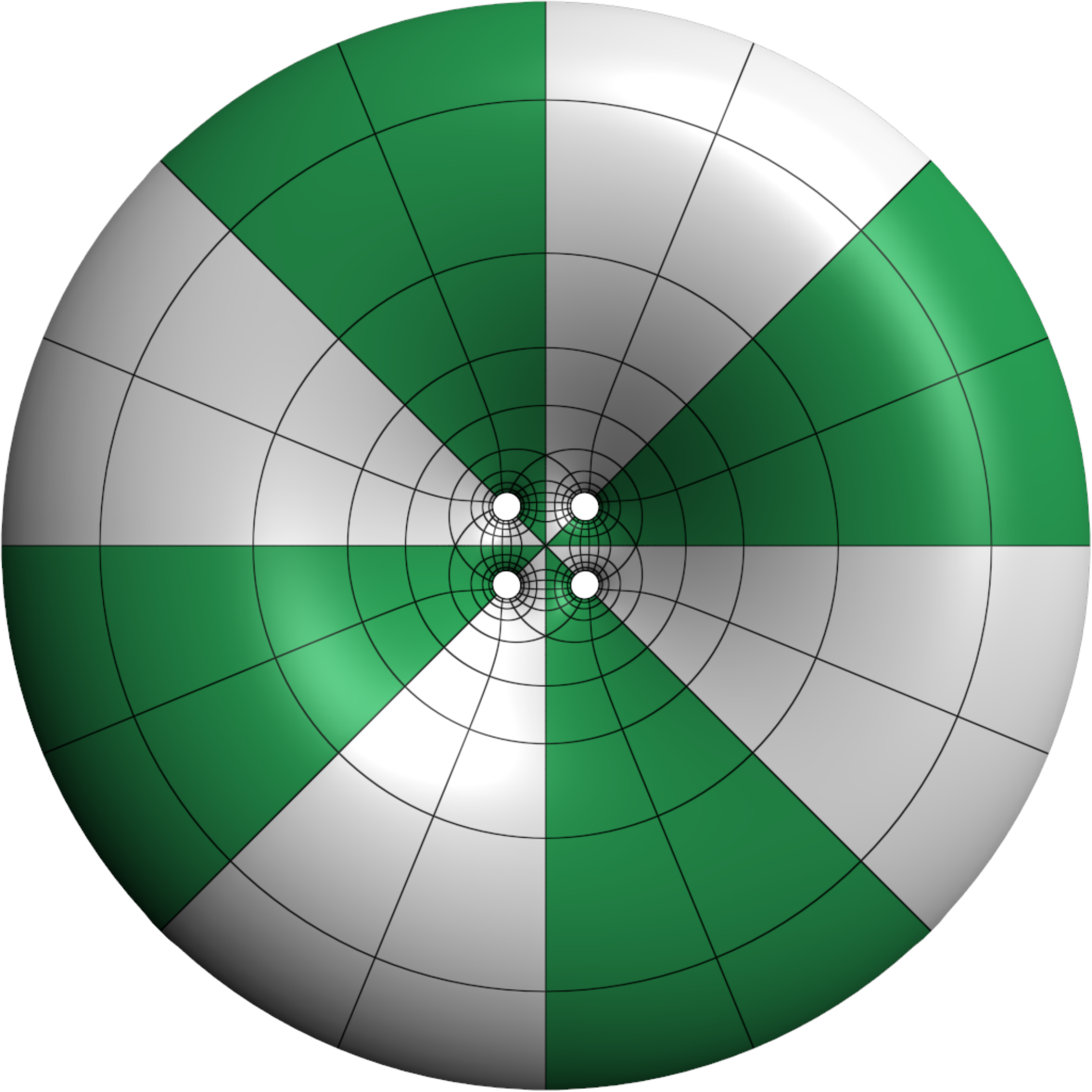}
  \includegraphics[width=0.375\textwidth]{image/B41b.pdf}
  \includegraphics[width=0.375\textwidth]{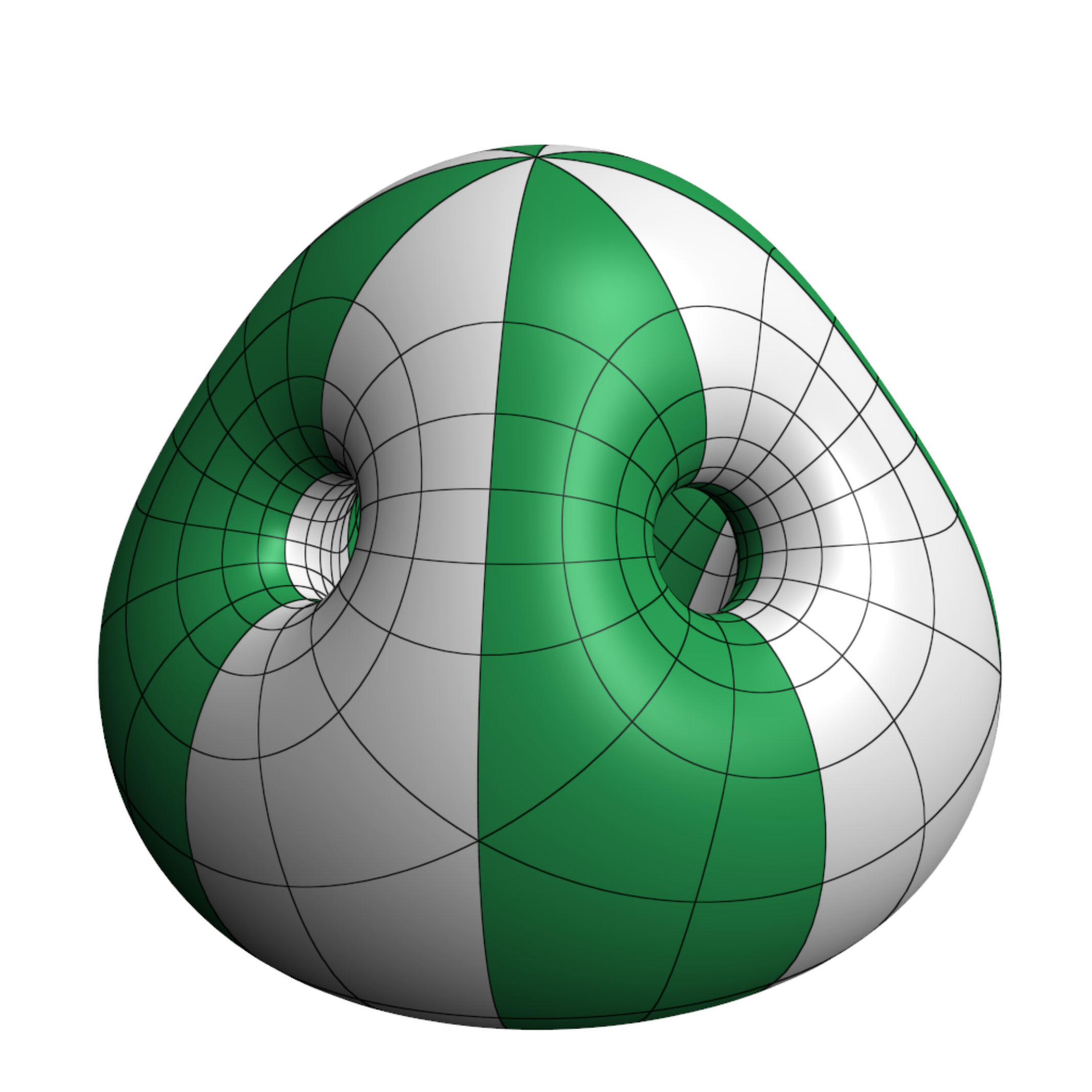}
  \caption{Three views (stereographic projections) of
    surface $\Bsurface{4}{1}$ of genus $4$.}\label{fig:b41}
\end{figure}


\section{Minimal surfaces in $\bbS^3$ via DPW method}\label{sec:dpw}
In this section we recall the basic principles of the DPW approach \cite{Dorfmeister_Pedit_Wu_1998} to minimal surfaces in $\mathbb S^3$, based on their associated families of flat connections \cite{Hitchin}.
For details see  \cite{Bobenko_Heller_Schmitt_2021,Heller_Heller_Traizet_2023} and references therein. 

\subsection{Minimal surfaces in $\mathbb S^3$}
A minimal surface $f\colon \Sigma\to \bbS^3$ is a critical point of the area functional.
As such, it is characterized by vanishing mean curvature $H=0.$ Because $f$ is an immersion (by assumption)
 $f$ induces a Riemannian metric and a conformal structure on $\Sigma.$ We  assume that $\Sigma$ is orientable. This is particularly the case when $\Sigma$ is compact and $f$ is an embedding.
Then $\Sigma$ is equipped with the structure of a Riemann surface
such that $f$ is conformal. It is well-known that a conformal map from a Riemann surface to $\bbS^3$ (or any other Riemannian manifold) is harmonic if and only if it is minimal.
In fact, the tension and the mean curvature of $f$ are related by
\[d^\nabla*df=2HNdA\]
where $\nabla$ is the pull-back of the Levi-Civita connection of $\bbS^3$, $H,$ $N$ and $dA$  are the mean curvature, the normal and the induced area form  of $f$, respectively.

A surface $f$ in $\bbS^3$  is uniquely determined by its first and its second fundamental forms up to spherical
isometry. Conversely, every pair consisting of a Riemannian metric $g$ and a symmetric bilinear form $II$ satisfying the (spherical) Gauss-Codazzi equations
is induced by an immersion which, in general, is only well-defined on some covering. If $f$ is conformal and minimal, then the second fundamental form 
\[II=Q+\bar Q\]
is uniquely determined by a (complex) quadratic differential which turns out to be holomorphic \[Q\in H^0(\Sigma,K_\Sigma^2).\]
In particular, every umbilic of $f$ is a star umbilic. 
Rotating $Q$ by some unimodular complex number $\lambda\in\mathbb S^1$ yields a new solution 
\[(g,\widetilde{II}=\lambda Q+\bar\lambda\bar Q)\]
of the Gauss-Codazzi equations. Consequently, we obtain a $\mathbb S^1$-family of minimal surfaces $f_\lambda$, 
which are in general only well-defined on the universal covering of $\Sigma$.

\subsection{The associated family of flat connections}
The associated $\mathbb S^1$-family of minimal surfaces $f_\lambda$ of a given minimal surface 
$f\colon\Sigma\to\bbS^3$ can be complexified as follows:
Identify $\bbS^3\cong\mathrm{SU}(2)$
such that the round metric of curvature 1 is given by 
$-\tfrac{1}{2}\tr()$ on $\mathfrak{su}(2)=T_e\mathrm{SU}(2).$ Decompose
the Maurer-Cartan form
\[f^{-1}df=2\Phi-2\Phi^*\]
into its complex linear 
\[\Phi\in\Omega^{(1,0)}(\Sigma,\mathfrak{su}(2)\otimes\C)=\Gamma(\Sigma,K_\Sigma\mathfrak{sl}(2,\C))\]
and its complex anti-linear 
\[-\Phi^*\in\Omega^{(0,1)}(\Sigma,\mathfrak{su}(2)\otimes\C)=\Gamma(\Sigma,\bar K_\Sigma \mathfrak{sl}(2,\C))\]
parts, i.e.,
\[\Phi=\tfrac{1}{4}(f^{-1}df-i*f^{-1}df)\quad\text{and}\quad -\Phi^*=\tfrac{1}{4}(f^{-1}df+i*f^{-1}df).\]
Note that we use the convention $*dz=idz$ and $*d\bar z=-id\bar z$ for any local holomorphic coordinate on $\Sigma$.

Define
\[\nabla=d+\tfrac{1}{2}f^{-1}df=d+\Phi-\Phi^*\]
and
\[\nabla^\lambda:=\nabla+\lambda^{-1}\Phi-\lambda\Phi^*.\]
By its very definition,
$\nabla^\lambda$ is unitary for all $\lambda\in\mathbb S^1$, and satisfies
\[\nabla^{\lambda=-1}=d\quad\text{and}\quad \nabla^{\lambda=1}=d+f^{-1}df=\nabla^{\lambda=-1}.f\,\]
i.e., $f$ is given as the gauge between $\nabla^{\lambda=-1}$ and $\nabla^{\lambda=1}.$
Since $f$ is minimal and hence harmonic we have
\[d^\nabla*df=0\quad\Longleftrightarrow  \quad d^\nabla \Phi=0.\]
A direct computation then shows that this is equivalent to flatness of
$\nabla^\lambda$ 
for all $\lambda\in\C^*$. Moreover, conformality of $f$ is equivalent to 
\[-\tfrac{1}{2}\tr(\Phi^2)=0\]
which (using $\tr(\Phi)=0$ which holds by construction) is itself equivalent to
$\Phi$ being nilpotent. Zeros of $\Phi$ are exactly the points where $f$ is branched.

Conversely, given a family of flat $\mathrm{SL}(2,\C)$ connections
\begin{equation}\label{asso}\lambda\in\C^*\mapsto\nabla^\lambda=\nabla+\lambda^{-1}\Phi-\lambda\Phi^*\end{equation}
over the Riemann surface $\Sigma$ satisfying
\begin{itemize}
\item $\nabla^\lambda$ is unitary for all $\lambda\in\mathbb S^1\subset\C^*;$
\item $\nabla^{\lambda=\pm1}$ are trivial;
\item $\Phi\in\Gamma(\Sigma,K_\Sigma\mathfrak{sl}(2,\C))$ is a complex linear nilpotent nowhere vanishing 1-form;
\end{itemize}
then the gauge $f$ satisfying
\[\nabla^1.f=\nabla^{-1}\]
is a conformal minimal immersion which is well-defined on $\Sigma$.
For details see \cite{Hitchin}.

\subsection{The DPW approach}
The classical DPW approach \cite{Dorfmeister_Pedit_Wu_1998} describes minimal surfaces in $\bbS^3$ (and CMC surfaces in $\bbR^3$) in terms 
of a holomorphic $\mathfrak{sl}(2,\C)$-valued 1-form \[\eta=\sum_{k\geq-1}\eta_k\lambda^k\] on $\Sigma$ depending meromorphically on a spectral parameter $\lambda\in\mathbb D^*\subset \C^*.$
Here $\mathbb D$ is a disc centered at $\lambda=0$ which contains the unit circle $\mathbb S^1.$ The 1-form $\eta$ is called a {\em DPW potential}.

The advantage of the DPW approach is that the potential $\xi$ just needs to satisfy the conditions that its
 residue $\res_{\lambda=0}\eta=\eta_{-1}$ is nilpotent and nowhere vanishing.
Then, the following procedure yields a minimal surface in $\mathbb S^3:$ Consider a solution
\[d\Psi+\eta\Psi=0\] 
depending holomorphically (in $\lambda$) on an initial condition  $\Psi(b)$, $b\in\Sigma$ fixed.
The map $\Psi$ is called a {\em holomorphic}  {\em frame}.
Clearly, $\Psi$ is only well-defined on the universal covering $\tilde\Sigma\to\Sigma$ in general.

In a second step, consider the  Iwasawa decomposition
\[\Psi=BF,\]
where
\[B\colon \tilde \Sigma\times \mathbb D\to\mathrm{SL}(2,\C)\]
is holomorphic on a disc $\mathbb D$ of radius $r>1$ centered at $0$,
and 
\[F\colon \tilde \Sigma\times \mathbb D\cap \mathbb D^{-1}\to\mathrm{SL}(2,\C)\]
is unitary (i.e., $\mathrm{SU}(2)$-valued) along $\lambda\in \mathbb S^1\subset \mathbb D\cap \mathbb D^{-1}.$
We call such maps $B$ and $F$ {\em positive} and {\em unitary loops}, respectively.
The (loop group) Iwasawa decomposition always exists, and is unique if one normalizes $B$ to be
upper triangular with positive diagonal entries at $\lambda=0$ (see \cite{SW} or \cite{Dorfmeister_Pedit_Wu_1998}).

Then,
\begin{equation}\label{dpwtrick}
\nabla^\lambda:=(d+\eta).B=d.F^{-1}\end{equation}
is the associated family of flat connections of some minimal surface $f\colon \tilde\Sigma\to\bbS^3,$
where $d=d_{\tilde\Sigma}$ is the partial differential with respect to ${\tilde\Sigma}.$ 
Note that the second equality follows from $d\Psi+\eta\Psi=0$ and $\Psi=BF$. Moreover, \eqref{dpwtrick}  directly implies that
$\nabla^\lambda$ is unitary for all $\lambda\in \bbS^1$, while $\nabla^\lambda=(d+\eta).B$ then implies
it is of the form \eqref{asso} with nilpotent $\lambda^{-1}$-part $\Phi=B(0)^{-1}\eta_{-1}B(0).$

The surface is well-defined on $\Sigma$ provided the following two conditions are satisfied:
\begin{enumerate}
\item $B$ is well-defined on $\Sigma$;
\item $d+\eta$ has trivial monodromy at $\lambda=\pm1$.
\end{enumerate}
In fact, (1) guarantees that $\nabla^\lambda$ is well-defined on $\Sigma$, while $d+\eta$ having trivial monodromy
at $\lambda=\pm1$ 
then implies that the gauge equivalent connections $\nabla^\lambda=(d+\eta).B$ has trivial monodromy at $\lambda=\pm1$ as well.

It should be noted here that there are no holomorphic DPW potentials that fulfill even only (1) on a compact Riemann surface of positive genus. 
In fact, the only holomorphic connection 1-form $\eta\in H^0(\Sigma,K_\Sigma\mathfrak{sl}(2,\C))$
with unitary monodromy is given by $\eta=0$.
On the other hand, one can admit apparent singularities of $\eta$ on $\Sigma$ to fulfill (1) and (2).

In order to deal with condition (1), we first note that $B$ being well-defined implies that
$d+\eta$ must have unitary monodromy (up to conjugation) for all $\lambda\in\bbS^1$. This necessary condition is in fact sufficient as well: if
$\Psi$ is a solution of $d\Psi+\eta\Psi$ with unitary monodromy  for all $\lambda\in\bbS^1$ at $b\in\Sigma$ for appropriate initial condition $\Psi(b),$
then one can deduce from the uniqueness of the Iwasawa decomposition that the positive term $B$ in the factorization of the meromorphic frame $\Psi=BF$ has trivial monodromy.
Note that by an application of the Iwasawa decomposition theorem, the initial condition $\Psi(b)$ can be chosen to be positive, i.e., to be a holomorphic map from $\mathbb D$ to $\mathrm{SL}(2,\C)$. Thus, by conjugating the potential with a positive initial condition $\Psi(b)$, we can always assume that the
monodromy of the holomorphic (or rather meromorphic since we allow singularities on $\Sigma$) frame $\Psi$ at the base-point $b$ with initial condition $\Psi(b)=\id$ is actually unitary, provided the initial 
potential $\eta$ has already unitary monodromy up to conjugation for all $\lambda\in\bbS^1$.

In order to construct compact minimal surfaces in $\mathbb S^3$ we therefore seek for
meromorphic DPW potentials $\eta$ on $\Sigma$ satisfying the following {\em closing conditions}:
\begin{itemize}
\item for  any pole $p$ of $\eta$ on $\Sigma$, there is a positive gauge $B$ such that $(d+\eta).B$ extends smoothly through $p$;
\item  the monodromy  of $d+\eta$  with respect to the base-point $b\in\Sigma$ is unitary  for all $\lambda\in\bbS^1$ ({\em intrinsic closing condition});
\item  the monodromy  of $d+\eta$ is trivial for $\lambda=\pm1$ ({\em extrinsic closing condition}).
\end{itemize}
 
\begin{remark}\label{re:rot}
In practice, it is often useful to rotate the spectral plane by some factor $e^{i\varphi}.$ After doing that, the extrinsic closing condition is that  the monodromy  of $d+\eta$ is trivial for $\lambda_1=e^{i\varphi}$ and $\lambda_2=-e^{i\varphi}$, and the surface is obtained as the gauge between these
two flat connections. We mainly use $\lambda=\pm i$ as this is most natural for dealing with reflections, 
see e.g. Lemma \ref{lem:refpotpro} below or \cite[Theorem 3.1]{Bobenko_Heller_Schmitt_2021}.
\end{remark}
The spectral parameters $\lambda_1$ and $\lambda_2=-\lambda_1\in\bbS^1$ for which $\nabla^\lambda$ has trivial monodromy are called
the
{\em evaluation points} of the surface $f$.
The minimal immersion $f\colon\Sigma \to \bbS^3$ is then given by the explicit formula from \cite{Bobenko_cmc_1991}:
$$
 f=(F^{\lambda_1})^{-1}F^{\lambda_2}.
 $$

For a simply connected surface, the intrinsic  and extrinsic closing conditions are vacuous, while on a compact
Riemann surface of higher genus the conditions are very restrictive. Moreover, in the case of 
of a compact
Riemann surface of higher genus, the monodromy of $\nabla^\lambda$  is irreducible for generic $\lambda\in\C^*.$
The same   also holds true for the potential $d+\eta^\lambda$.
 This 
observation implies that the choice of initial condition $\Psi(b)$ is unique up to multiplication with a unitary loop $F$ from the right. This multiplication does 
not change $\nabla^\lambda$ in \eqref{dpwtrick} and therefore produces the same minimal surface.

\begin{remark}
Dressing is a procedure which produces new minimal surfaces from existing one. In order to not change the topology of the immersion by a dressing transformation, it is necessary to have spectral parameters $\lambda_0\in \C^*\setminus\bbS^1$ at which $\nabla^{\lambda_0}$ has abelian monodromy representation. 
As dressing (on higher genus surfaces) changes the conjugacy classes of the monodromy representations at finitely many spectral parameters, dressing changes the DPW potential $\eta$ in a non-trivial way as well. 
Therefore, dressing transformations do not exist for the minimal surfaces considered in this paper as there are no additional  abelian monodromies besides the trivial one for the evaluation points.
It would be very interesting to have an example of a compact  minimal (or CMC) surface
of genus at least 2 which actually allows for non-trivial dressing transformations. 
For more details on dressing see for example \cite{Burstall-Darboux} or \cite{He3} and the references therein.
\end{remark}
\subsection{Reflection potentials}
\label{sec:potential}
Our aim is to construct and study compact minimal reflection surfaces $f\colon\Sigma\to\bbS^3,$
i.e., reflection surfaces which are minimal. 
Examples of minimal reflection surfaces in $\bbS^3$
built from fundamental quadrilaterals in
fundamental tetrahedra include the following:
\begin{itemize}
\item
  the minimal Lawson surfaces~\cite{Lawson_1970};
\item  and minimal surfaces constructed by Karcher-Pinkall-Sterling~\cite{Karcher_Pinkall_Sterling_1988}.
\end{itemize}

By  definition, the fundamental region $P$ of a minimal reflection surface is a topological disc. We can therefore apply results from \cite{Bobenko_Heller_Schmitt_2021} about the existence of DPW potentials.
If the number of vertices of $P$ is even, it is shown in \cite[Proposition 5.5]{Bobenko_Heller_Schmitt_2021} that there is a meromorphic DPW
potential $\xi$ on $\CPone$ generating $f$ in the following way: the order 2 subgroup $\Gamma<G$ of orientation preserving symmetries
acts on $\Sigma$ by holomorphic automorphisms, and the quotient Riemann surface is given by $\pi\colon\Sigma\to\Sigma/\Gamma=\CPone.$ Then $\eta=\pi^*\xi$ is a meromorphic
DPW potential on $\Sigma$ satisfying the above three closing conditions.
Furthermore, in the case of the Lawson surfaces $\xi_{1,g}$ with $n=4$ vertices, it is shown in \cite{Heller_Heller_2023} that the DPW potential $\xi$ is actually Fuchsian, i.e., of the form
\[\xi=\sum_{k=1}^4A_k \frac{dz}{z-p_k}\]
for some $A_k\colon \bbD^*\to\mathfrak{sl}(2,\C)$ constant in $z.$ 
We expect, without having a general proof, that all minimal reflection surfaces can be constructed by Fuchsian DPW potentials, which motivates the following definition.

\begin{definition}
  A \emph{reflection potential}
  is a Fuchsian DPW potential on $\CPone$
  \begin{equation}
    \label{eq:potential}
    \xi = \sum_{k=1}^p \frac{A_k}{z-z_k}\deriv z
  \end{equation}
  with $p$ simple poles $z_1,\dots,z_p\in\bbS^1$
  (and no pole at $\infty$)
  satisfying the following conditions:
  \begin{itemize}
  \item
    The potential satisfies the reality condition
    \begin{equation}
      \label{eq:potential-reality-condition}
      \ol{\tau^\ast\xi(\ol{\lambda})} = \xi(\lambda)
      \quad
\text{for}\quad      \tau(z) \coloneq 1/\ol{z}
      \spacecomma
    \end{equation}
  \item
    The eigenvalues of each $A_k$
    are $\lambda$-independent and contained in the interval $(-\tfrac{1}{2},\tfrac{1}{2}).$
      \end{itemize}
\end{definition}

The ($\lambda$-independent)  eigenvalues $\pm\nu_k$ of $A_k$ encode the dihedral angle $\theta$ of the adjacent reflection planes at $z_k$ of a minimal reflection surface.
In fact, we have
\begin{equation}\label{evangle}\theta=\begin{cases}2\pi\nu_k &\text{ if } \nu_k\in(0,\tfrac{1}{4}]\\ 
\pi-2\pi\nu_k &\text{ if } \nu_k\in(\tfrac{1}{4},\tfrac{1}{2})\end{cases}.\end{equation}
Relationship \eqref{evangle} follows from
\cite[Theorem 3.3]{Bobenko_Heller_Schmitt_2021} since the
monodromy along  a simple closed curve around $z_k$ is conjugated to $\exp(2\pi i A_k).$
By \cite[Theorem 3.5]{Bobenko_Heller_Schmitt_2021} the eigenvalues of the residues  $A_k$
have to be $\pm\tfrac{1}{2n_k}$ or $\pm\tfrac{n_k-1}{2n_k}$ in order for $f$ being immersed at $z_k$ (or rather at its preimages in $\Sigma$).
For later use, we call the first and the second case of \eqref{evangle} to be of {\em spin} -1 and of {\em spin} 1, respectively. See
\cite{Bobenko_Heller_Schmitt_2021} for more details about the spin of a potential, and its relationship with the spin structure of an immersion.

We choose a base-point $b\in\bbS^1\subset\CPone$.
The monodromy of the meromorphic frame $\Psi$ for a reflection potential $\xi$ is computed
with initial condition $\Psi(b)=\id$.

Let $\bbS^1\subset\CPone$ be divided into $p$ segments $s_1,\dots,s_p$
at $p$ distinct consecutive points $z_{k}$, $k=1,\dots,p\,,$ dividing $s_k$ and $s_{k+1}$ (or $s_p$ and $s_1$).
Let $\xi$ be a reflection potential on $\CPone$ with
singularities at these points $z_{k}$.  With $b$ a base-point on $\bbS^1$, for $i,\,j\in\{1,\dots,n\}$ let $\gamma_{ij}$ be a
simple closed counterclockwise curve based at $b$ which crosses the
segments $s_i$ and $s_j$, and let $M_{ij}$, $i,\,j\in\{1,\dots,n\}$,
$i < j$ be the monodromy of a meromorphic frame $\Psi$ for $\xi$ along $\gamma_{ij}$.  The $n$ \emph{local}
monodromies are those along paths which enclose one singularity; the
remaining monodromies are called \emph{global}.

We then have the following theorem.
\begin{theorem}\cite[Theorem 3.3]{Bobenko_Heller_Schmitt_2021}
If the meromorphic frame $\Psi$ of a reflection potential $\xi$
\begin{itemize}
\item
 has unitary monodromy on $\bbS^1$, and
\item
  its local and global logarithmic monodromy eigenvalues
  are correct, i.e., at the evaluation points $\lambda_l$, $l=1,2$, we have
  \begin{equation}\label{gevangle}
  \half \tr M_{ij}|_{\lambda_l} = \pm  \cos\theta_{ij} \spacecomma\quad
  i,\,j\in\{1,\dots,n\}\spacecomma\quad i < j,
  \end{equation}
\end{itemize}
then the unit disk maps to a minimal polygon 
whose boundaries reflect in $p$ totally geodesic spheres $P_1$, . . . , $P_p$, with internal dihedral angles $\theta_{i,j}$ between $P_i$ and $P_j$.
\end{theorem}
In \eqref{gevangle}, the appropriate sign on the right hand side can be determined by using the 
spin of a potential.  For details, see \cite[Section 3.4]{Bobenko_Heller_Schmitt_2021}.
In the case at hand of a reflection surface, the angles are of the form $\theta_{ij}=\frac{\pi}{n},$ where $n\in\N^{\geq 2}$  is
the vertex integer of the corresponding point, compare with Definition \ref{def:vertex-and-edge-integers}.

\typeout{== figure/B32.tex ============================================}\begin{figure}[b]
  \includegraphics[width=0.375\textwidth]{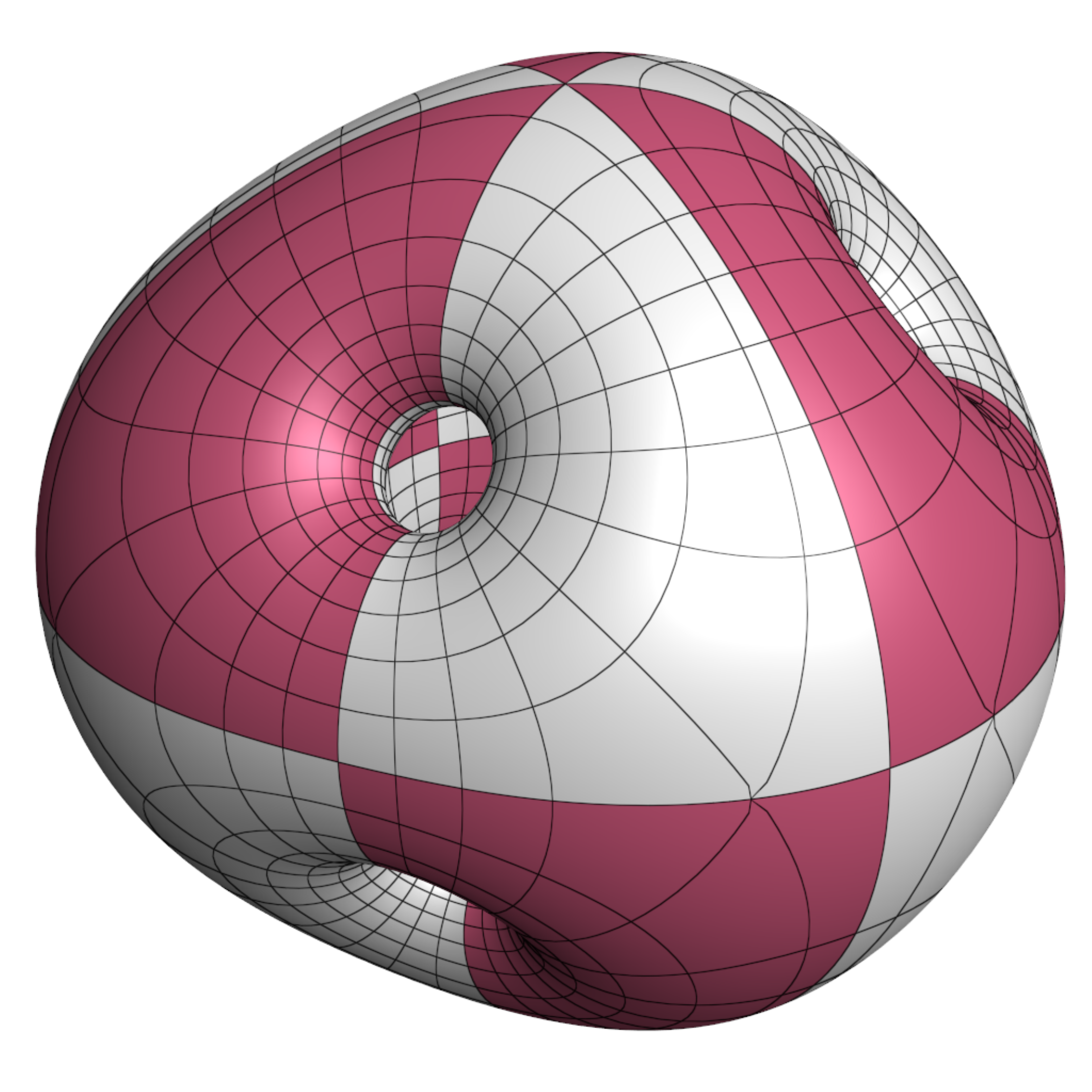}
  \includegraphics[width=0.375\textwidth]{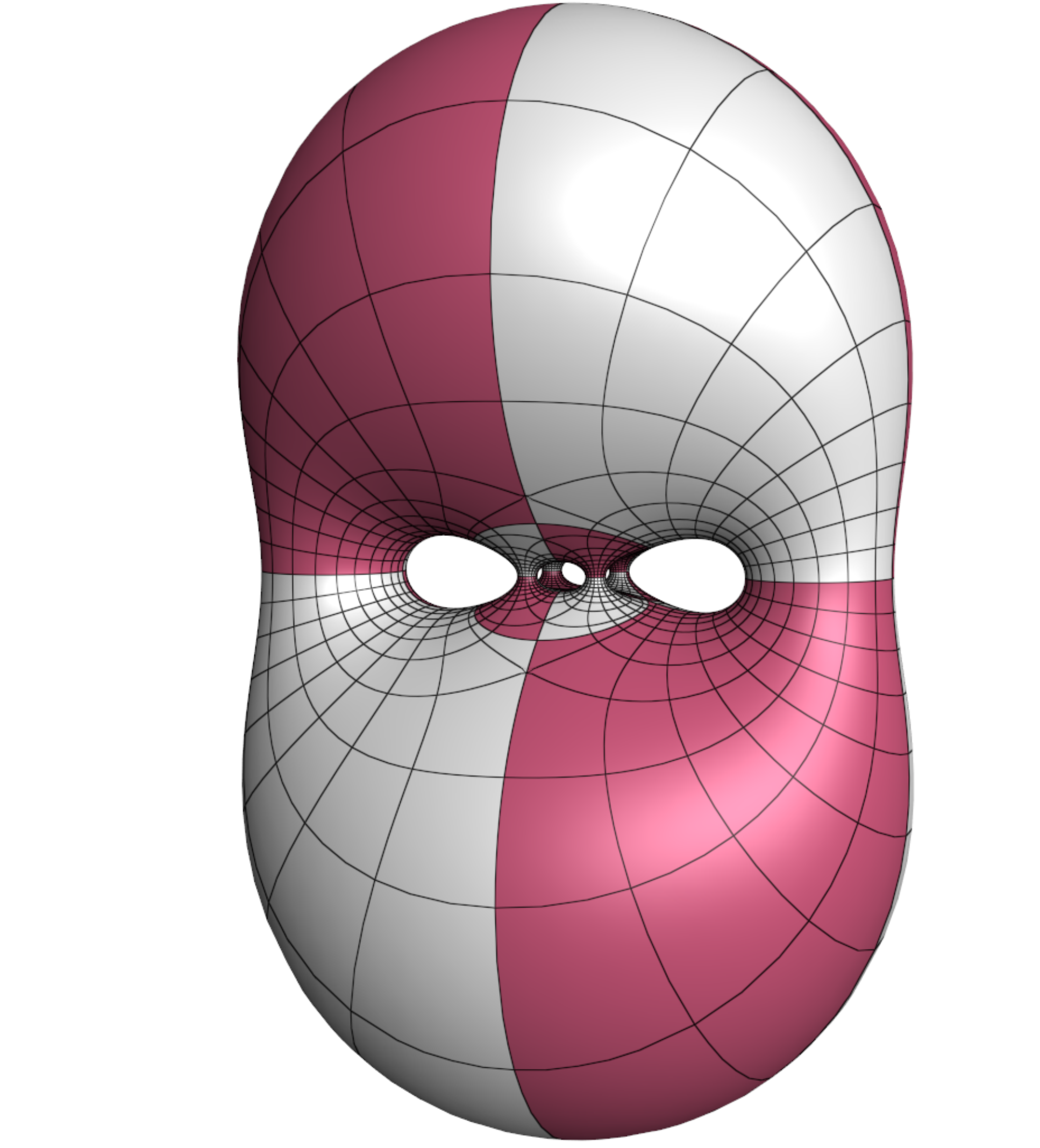}
  \includegraphics[width=0.375\textwidth]{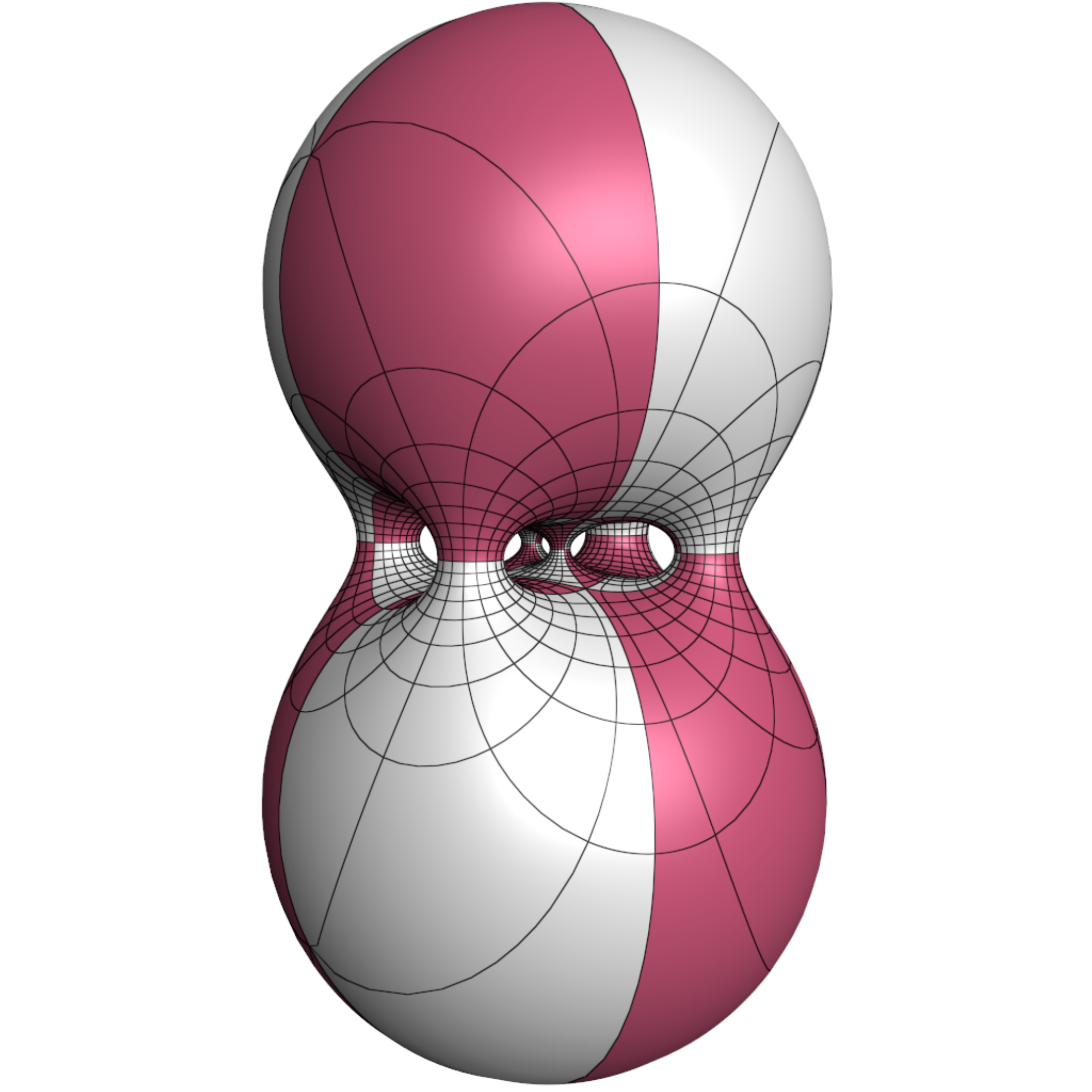}
  \includegraphics[width=0.375\textwidth]{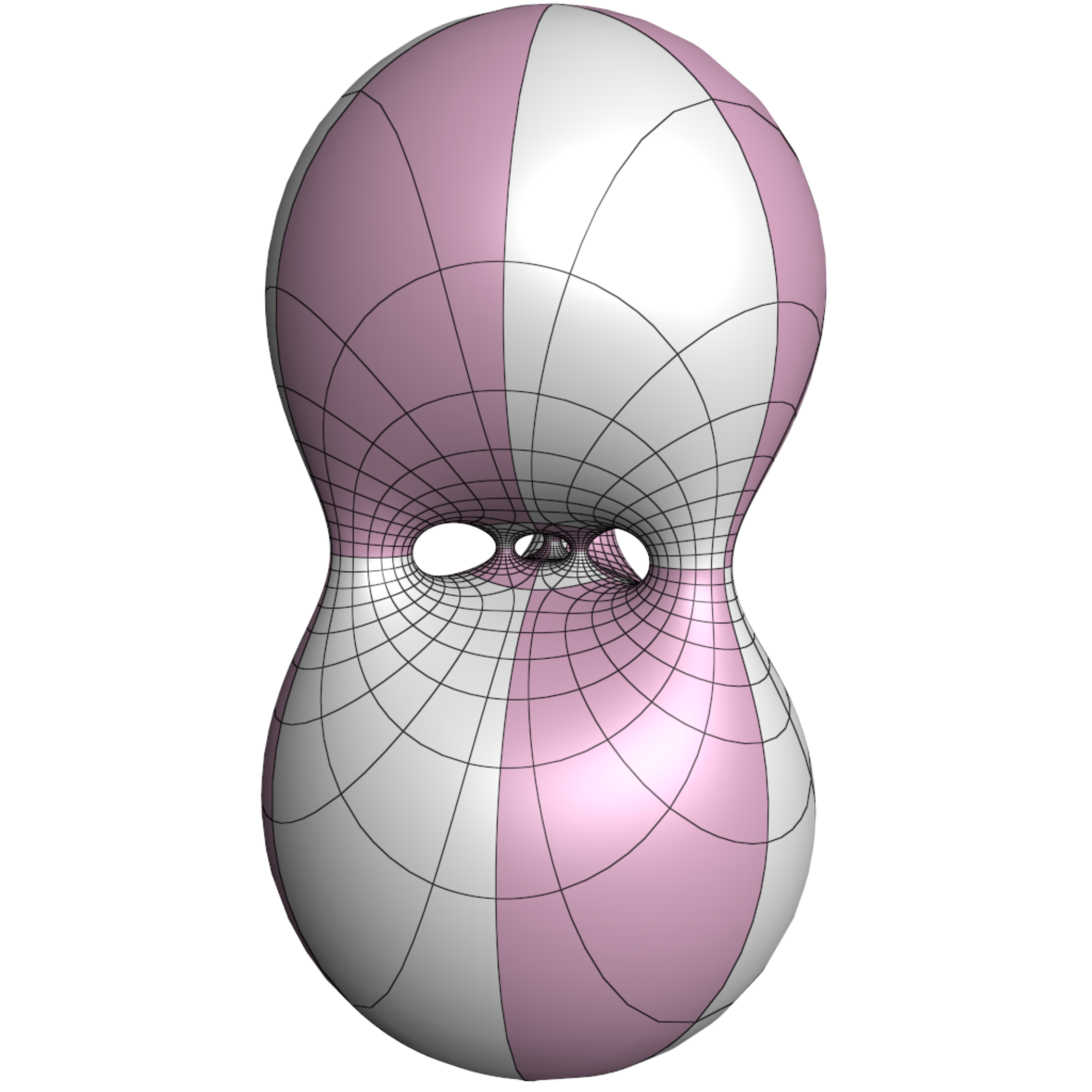}
  \caption{Four views (stereographic projections) of
    the surface $\Bsurface{3}{2}$ of genus $5$.}\label{fig:b32}
\end{figure}

\subsection{The flow}\label{ssec:flow}
Finding reflection potentials satisfying the closing conditions is a difficult task, as this requires 
a detailed understanding of the monodromy representations of Fuchsian systems. At the moment, no reflection
potential of a reflection surface of genus $g\geq2$  is known explicitly.
On the other hand, the existence of the Lawson surfaces $\xi_{g,1}$ has been proven recently by
showing the existence of reflection potentials satisfying the closing conditions via an implicit function theorem argument, at least for genus large enough, see \cite{HHT2} and also \cite{HHS, Heller_Heller_Traizet_2023}. In \cite{HHT2}, the positive eigenvalue  of the residue $A_1$ 
(or equivalently $A_l$ for any $l=1,..,4$) has been
used as a flow parameter.
One can then reformulate the implicit function theorem as a flow.
This flow has been numerically implemented and extended to other surface
classes with reflection potentials with 4 vertices, see \cite{Bobenko_Heller_Schmitt_2021}.
In the following, we report on our numerical experiments with reflection potentials with more than 4 vertices.

Given $F:\bbR\times\bbR^n\to\bbR^n$,
then $x:\bbR\to\bbR^n$ satisfying $F(t,\,x) = 0$
can be computed by solving the differential equation
\begin{equation}\label{withttheflow}
\frac{\deriv F}{\deriv t} +
\frac{\deriv F}{\deriv x} \frac{\deriv x}{\deriv t} = 0
\spaceperiod
\end{equation}

The free variables $x$ for the flow are:
\begin{itemize}
\item
  the poles of the potential, and 
  \item the coefficients of the residues of the potential
  as a finite approximated Laurent series.
  \end{itemize}
  As such, these variables encode the geometric parameters:
  \begin{itemize}
\item
 the conformal type  of the surface  in terms of the poles, 
\item the local and global logarithmic monodromy
  eigenvalues, and
  \item the mean curvature $H$ via the evaluation points $\lambda_1\neq\lambda_2\in\bbS^1$ (with $\lambda_2=-\lambda_1$ for minimal surfaces). 
\end{itemize}
The flow runs over the interval $t\in[0,\,1]$.

The constraints $F$ for the flow are
\begin{itemize}
\item
  sum condition: $\sum_{k=1}^{p} A_k = 0$.
\item
  determinant condition: the $\lambda^{\NEG 1}$ coefficient $\xi^{(\NEG 1)}$ of $\xi$ has determinant zero and is nowhere vanishing.
\item
  eigenvalue condition: the eigenvalues of $A_k$ are constant in $\lambda$ and specified, see \eqref{evangle}.
\item
  intrinsic closing constraints: the monodromy representation
  is contained in $\matSU{2}{}$  (at appropriately many sample points) along $\bbS^1$,
\item
  extrinsic closing constraints:
  the monodromy at the evaluation points is specified, see  \eqref{gevangle}.
\end{itemize}

The geometric parameters are chosen
appropriately according to the target surface, as follows.
For flows through minimal surfaces in $\bbS^3$,
\begin{itemize}
\item
  The conformal type is left free by
  fixing three of the poles and leaving the others
  free on $\bbS^1$.
\item
  the local and global logarithmic monodromy eigenvalues are
  set to be linear functions of $t$ so as
  to have the desired target values at $t=1$.
\item
  The evaluation points are fixed to $\lambda=\pm i$.
\end{itemize}

Other flows are possible with other choices of constraints
on the geometric parameters.
For example, to flow through CMC surfaces
instantiating a fixed reflection surface (fixed local
and global logarithmic monodromy eigenvalues),
the conformal type is made to depend on $t$ and the evaluation points
$\lambda_1,\,{\lambda}_{2}\in\bbS^1$ are left free.

For the case of the Lawson surfaces $\xi_{g,1}$, the above finite dimensional flow approximates the infinite-dimensional 
flow constructed in \cite{HHT2}. Using the theory of parabolic Deligne-Hitchin moduli spaces (see e.g. \cite{HHT3} and the references therein), 
reflection potentials satisfying the closing conditions can be interpreted as holomorphic sections of the
Deligne-Hitchin moduli spaces satisfying a reality constrained.
A native
dimension count then suggest that the experimental flow  \eqref{withttheflow} can be set up to model
deformations of reflection potentials satisfying the closing conditions.
With that, our experiments suggest that these deformations are in fact generically possible and unique as long as
the parameters are carefully chosen, e.g., fixing the monodromy at the evaluation points and the mean curvature $H$ to be contained in a given reflection group should determine the conformal type locally.

\typeout{== figure/B23.tex ============================================}\begin{figure}[b]
  \includegraphics[width=0.375\textwidth]{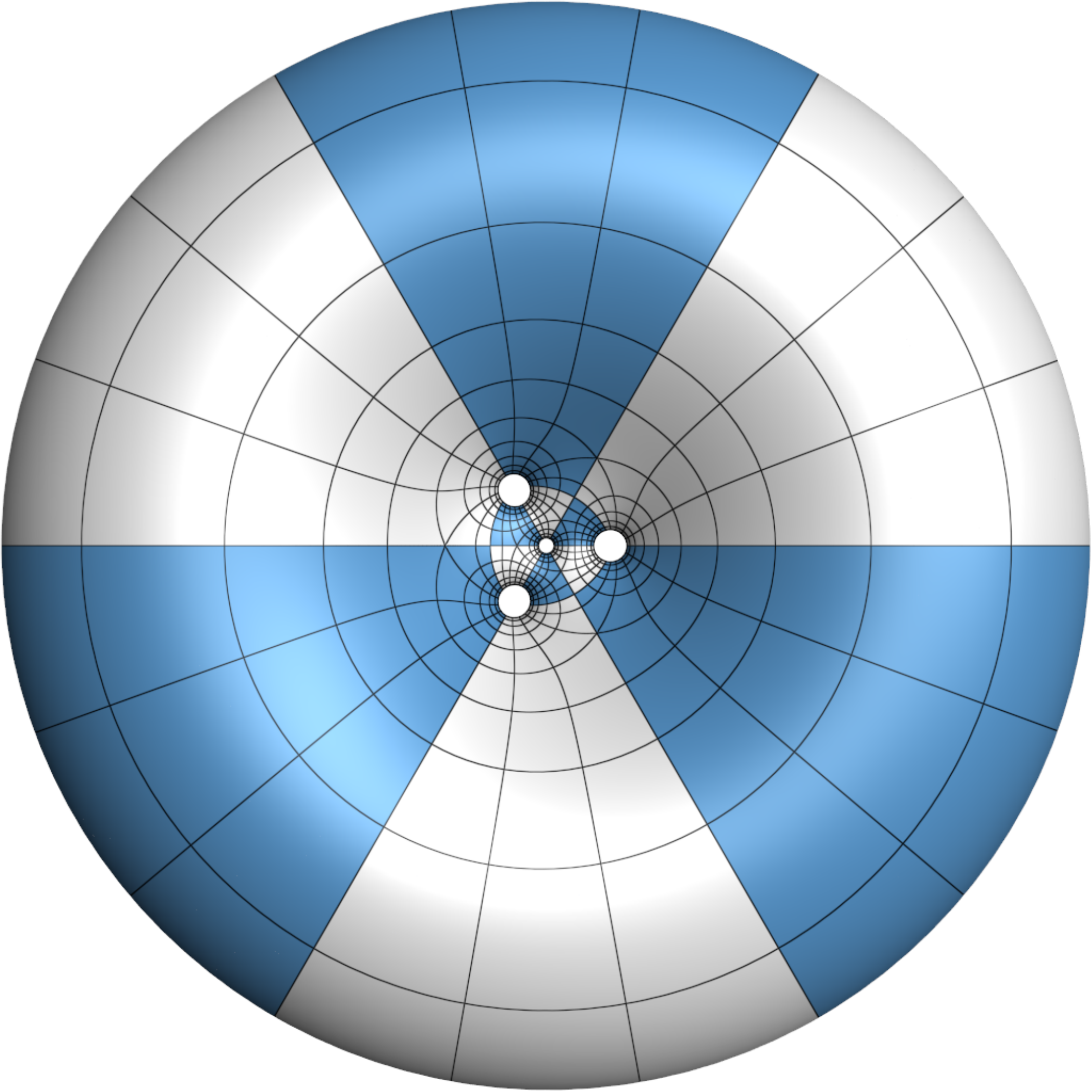}
  \includegraphics[width=0.375\textwidth]{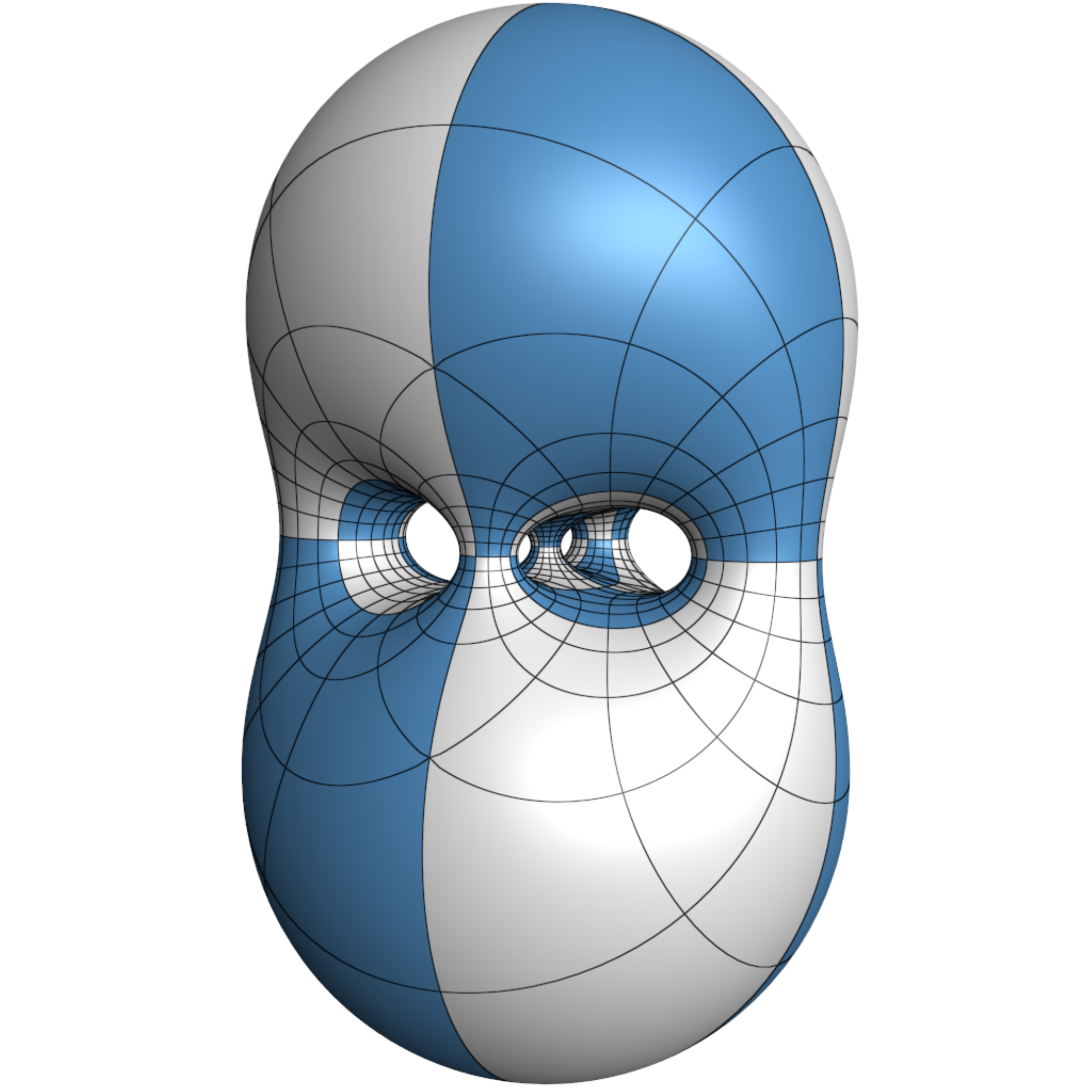}
  \includegraphics[width=0.375\textwidth]{image/B23c.pdf}
  \includegraphics[width=0.375\textwidth]{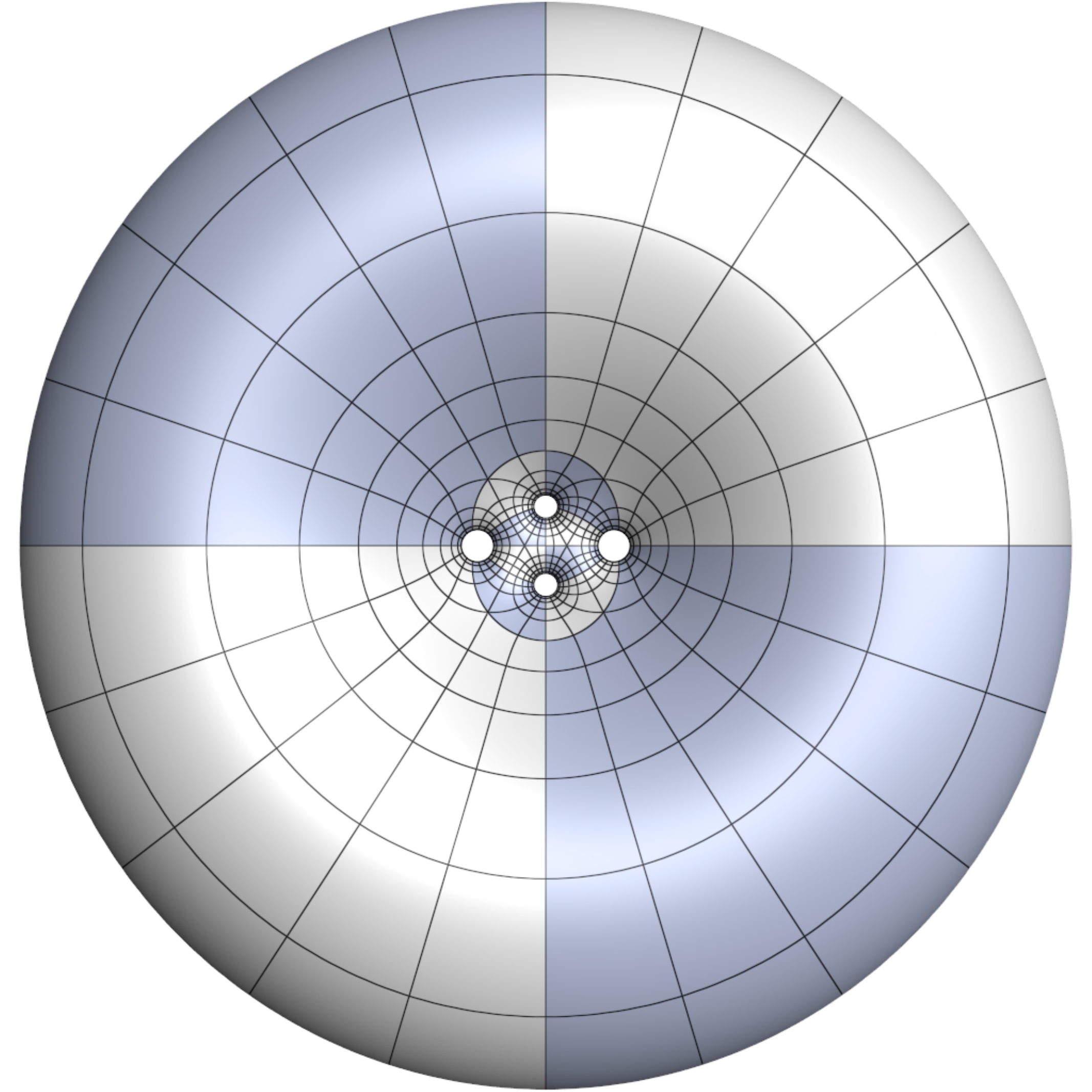}
  \caption{Four views (stereographic projections) of
    surface $\Bsurface{2}{3}$ of genus $4$.}\label{fig:b23}
\end{figure}

\section{Lawson surface potentials}
\label{sec:lawson-potential}

In this section we discuss the existence of reflection potentials for the Lawson surfaces with arbitrarily many  poles $p$. These potentials  serve
as initial conditions for numerical experiments. In particular,
for the case $p=5$,  Theorem \ref{thm:lawson-potential} constructs
\begin{itemize}
\item
  a $5$-pole potential for the Lawson surface $\xi_{1,2}$ with a fundamental polygon bounded by three planes,
   and
\item
  a $5$-pole potential for the Lawson surface $\xi_{1,3}$ with a fundamental polygon bounded by four planes.
\end{itemize}
These are used as the initial potentials for
flows to all the surfaces constructed
from fundamental pentagons.
Both these potentials are seen numerically to extend
holomorphically to the punctured unit $\lambda$ disk.

Fuchsian DPW potentials for Lawson surfaces
with $5$ or more poles
are computed in
Theorem \ref{thm:lawson-potential}
by a combination
of pushdowns, pullbacks, gauges and coordinate changes,
starting from the $4$-pole Fuchsian DPW potential for Lawson surface constructed in \cite{HHT2,Heller_Heller_2023}.
The following technical lemmas
are required to compute these potentials.

A {\em flip gauge} $g$ of a Fuchsian potential $\xi$ is a gauge transformation $g$ 
(possibly only well-defined up to $g\mapsto -g$)
such that $\xi.g=g^{-1}dg+g^{-1}\xi g$ is again Fuchsian and such that
the pole set of $\gauge{\xi}{g}$ is either the same as
or is a subset of the pole set of $\xi$.

In the following, an \emph{eigenline} of a residue $R$
of a Fuchsian potential
is a holomorphic map from a neighborhood of $\lambda=0$ to
$\bbC^2\setminus \transpose{(0,\,0)}$
which is an eigenline of $R$.

\begin{lemma}
  \label{lem:flip-gauge}
  \theoremname{Flip gauge}
  Let $\xi$ be a Fuchsian DPW potential with
  simple poles at $z=p$ and $z=\infty$
  and respective residues $A$ and $B$
  with eigenvalue in $\bbR^\ast$.
  Let $\ell$ and $m$ be the eigenlines of $A$ and $B$
  with respect to their respective positive eigenvalues.
  Let $h = (\ell,\,m)$ be the $2\times 2$ matrix valued
  map with columns
  $\ell$ and $m$.
  If $\det h$ is nonzero at $\lambda=0$,
  then with
  \begin{equation}
    g\coloneq hk
    \spacecomma\quad k\coloneq \diag({(z-p)}^{-1/2},\,{(z-p)}^{1/2})
  \end{equation}
  is a flip gauge for $\xi$.

  Moreover, if the positive eigenvalue of $A$ is $1/2$ and the monodromy around $p$ is $-\id$, then
  $\gauge{\xi}{g}$ does not have a pole at $0$.
\end{lemma}

\begin{proof}
  Compute that $h^{-1}Ah$ is upper triangular and
  $h^{-1}Bh$ is lower triangular.
  Then gauging $\xi$ by $g$ does not add any poles to $\xi$,
  so $g$ is a flip gauge for $\xi$.
  
  The last statement follows since $\xi.g$ has monodromy $\id$ around $p$ and its residue at $p$ has eigenvalues
  $0$. Hence, the residue vanishes as claimed.
\end{proof}
In principle, applying a flip gauge can produce apparent singularities of the potential inside the unit spectral disc.
This happens exactly when the eigenlines $\ell $ and $m$ fall together.

\begin{lemma}
  \label{lem:push-down}
  \theoremname{Push down}
  Consider \[ \tau(z) \coloneq-z
    \spacecomma\quad
    \sigma = \diag(i,\,-i)
    \]
and  let $\xi$ be a Fuchsian DPW potential with symmetry
  \begin{equation}
    \label{eq:reflection-symmetry}
    \tau^\ast \xi \coloneq  \gauge{\xi}{\sigma}    \spaceperiod
  \end{equation}
  Then there exists a Fuchsian DPW potential $\eta$
  such that
  \begin{equation}
    \gauge{\xi}{g} = f^\ast \eta
    \spacecomma\quad
    g\coloneq \diag( z^{1/2},\, z^{-1/2} )
    \spacecomma\quad
    f(z) \coloneq z^2
    \spaceperiod
  \end{equation}
\end{lemma}

\begin{proof}
  The potential $\gauge{\xi}{g}$
  is invariant under $\tau$,
  so it is the pull back of a potential
  $\eta$ under $f$.
\end{proof}

\begin{lemma}
  \label{lem:symmetrizing-gauge}
  \theoremname{Symmetrizing gauge}
  Let
  \begin{equation}
    \xi =
    \frac{A}{z-p}\deriv z +
    \frac{B}{z-\tau(p)}\deriv z +
    \frac{C}{z-q}\deriv z +
    \frac{D}{z-\tau(q)}\deriv z
  \end{equation}
  be a $4$-pole Fuchsian DPW potential on $\CPone$,
  where $\tau$ is an involutive M\"obius transform.
  If
  \begin{itemize}
  \item
    $\det A = \det B$
     \item
    $\det C = \det D$
  \item
    the kernels of
    $A^{(-1)}$ and $B^{(-1)}$
    (the $\lambda^{-1}$ coefficients of $A$ and $B$)
    are independent
  \end{itemize}
  then there exists a $z$-independent gauge $g$ such that
  $\eta\coloneq \gauge{\xi}{g}$ satisfies the symmetry
  \begin{equation}
    \label{eq:symmetryc-gauge-symmetry}
    \tau^\ast \eta = \gauge{\eta}{\sigma}
    \spacecomma\quad \sigma \coloneq \diag(i,\,-i)
    \spaceperiod
  \end{equation}
\end{lemma}
\begin{proof}
  The  last condition implies that
  \begin{equation}
    h \coloneq \frac{ A + B }{ \sqrt{\det(A + B)} }=-\frac{ C + D }{ \sqrt{\det (C + D)} }
  \end{equation}
  is holomorphic at $\lambda=0$.
  From the first and second condition we obtain by using $\tr A=0=\tr B=\tr C=\tr D$ that $Ah = hB$ and $Ch = hD$.

  Since $\tr h = 0$, there exists
  $g$ holomorphic at $\lambda=0$ such that $h = g \sigma g^{-1}$.
  Namely, $g = (m_1,\,m_2)$, where the columns $m_1$ and $m_2$
  are the independent eigenlines of $h$ corresponding
  to the eigenvalues $\pm i$.

  From $B = h^{-1}A h$ and $D = h^{-1}C h$  we observe
  \begin{equation}
    g^{-1}B g = \sigma^{-1}(g^{-1}Ag)\sigma
    \quad
    \text{and}
    \quad
    g^{-1}D g = \sigma^{-1}(g^{-1}Cg)\sigma
    \spaceperiod
  \end{equation}
  Hence $\eta \coloneq \gauge{\xi}{g} = g^{-1}\xi g$
  has the required symmetry~\eqref{eq:symmetryc-gauge-symmetry}.
\end{proof}


\begin{lemma}\label{lem:refpotpro}
  Let $\xi$ be a Fuchsian potential with poles on $\mathbb S^1$
  with unitarizable generically irreducible monodromy. Assume that $\xi$
  induces a minimal surface $f$ in $\mathbb S^3$ which reflects across each arc of
  $\mathbb S^1$ for the evaluation points $\lambda_{1/2}=\pm i$. Then there exists a $z$-independent gauge $g$ such that $\eta=\xi.g$
  satisfies the symmetry~\eqref{eq:reflection-symmetry}.
  Moreover, for every $b\in\mathbb S^1$ distinct from the poles of $\xi$,
  we can chose a $z$-independent gauge $g=g_b$ such that
  $\eta=\xi.g$ has unitary monodromy with respect to the base-point $b.$
\end{lemma}
\begin{proof}
As $\xi$ has unitarizable monodromy, we can assume without loss of generality that $\xi$ already has unitary monodromy with respect to the base-point $b\in\mathbb S^1.$ Therefore, the minimal surface $f$ is obtained by
the following process 
\begin{itemize}
\item solve $d\Psi+\xi\Psi=0$ with $\Psi(b)=\id$; 
\item factor $\Psi=BF$ into positive part $B$ and unitary  part   $F$;
\item then $f=(F^{\lambda=-i})^{-1}F^{\lambda=i}$ (up to possible spherical isometries).
\end{itemize}
We claim that
\[d+\tau^*\overline{\eta(\bar\lambda)}\]
with the same initial condition $\Psi(b)=\id$ yields the reflected surface of $f$. 
In fact $\tau^*\overline{\Psi(\bar\lambda)}$ solves
\[(d+\tau^*\overline{\eta(\bar\lambda)})\tau^*\overline{\Psi(\bar\lambda)}=0\]
with $\tau^*\overline{\Psi(\bar\lambda)}(b)=\id.$ Then we have the Iwasawa factorization
\[\tau^*\overline{\Psi(\bar\lambda)}=\tau^*\overline{B(\bar\lambda)}\tau^*\overline{F(\bar\lambda)}\]
into positive and unitary part, so that we  obtain in the third step the reflected surface
\[\tilde f=\tau^*\overline{F(\bar\lambda=-i)}^{-1}\tau^*\overline{F(\bar\lambda=i)}=\tau^*\bar f^{-1}=\tau^* f^t\]
as claimed.

Consider the positive gauge $g_1$ satisfying
\[(d+\eta(\lambda)).g_1=d+\alpha_\lambda\]
for the associated family of flat connections, with $g_1(b)=\id.$
This gauge is obtained as $g_1=B$ in the factorisation $\Psi=BF$ where $\Psi$ solves $d\Psi+\xi\Psi=0$ with $\Psi(b)=\id$.

By assumption $f$ satisfies $\tau^*f=f^t$, so that there exists a positive gauge $g_2$ with
\[(d+\tau^*\overline{\eta(\bar\lambda)}).g_2=d+\alpha_\lambda.\]
We have $g_2(b)=\id$ as well.

Then, $g=g_1g_2^{-1}$ is a positive gauge satisfying
\[d+\tau^*\overline{\eta(\bar\lambda)}=(d+\eta(\lambda)).g\]
and \[g(b)=\id.\]

We claim that $g$ is constant in $z$. In fact, any gauge between two Fuchsian systems can be singular only at points where the eigenvalues of the residues of the two Fuchsian systems differ by half-integers. But $d+\tau^*\overline{\eta(\bar\lambda)}$ and $d+\eta(\lambda)$ have the same eigenvalues \eqref{evangle}, so singularities of $g$ can only occur at points with eigenvalues of the residues being $\pm\tfrac{1}{4}$. 
To prove that this is not possible either, we just observe that such a singular gauge $g$
would change the $\lambda^{-1}$-behaviour  at the corresponding residue: if the  $\lambda^{-1}$-part of the residue
does not vanish before applying the singular gauge $g$ (the spin -1 case in the framework of  \cite{Bobenko_Heller_Schmitt_2021})
it does  vanish afterwards (the spin 1 case), and vice versa.
This is not possible, as $d+\tau^*\overline{\eta(\bar\lambda)}$ and $d+\eta(\lambda)$ have the same spin at all corresponding singular points. As $\CPone$ is compact and $g$ is holomorphic, $g$ is constant in $z$. As $g(b)=\Id$ this
 proves the statement.
\end{proof}


\begin{theorem}
  \label{thm:lawson-potential}
  \theoremname{Lawson potential}
  \begin{enumerate}
  \item
    For each $n\in\bbN_{\ge 2}$,
    there exists a reflection potential
    for the Lawson surface $\xi_{1,\,n-1}$
    with $n+2$ poles on $\bbS^1$, each with positive residue eigenvalue $1/4$.
    The unit disk is the domain for a minimal $(n+2)$-gon in $\bbS^3$
    whose boundary lies in a polytope bounded by $3$ planes.
  \item
    For each even $n\in\bbN_{\ge 2}$,
    there exists a reflection potential for the Lawson surface $\xi_{1,\,n-1}$ 
    with $n/2+3$ poles on $\bbS^1$, each with positive residue eigenvalue $1/4$.
    The unit disk is the domain for a minimal $(n/2+3)$-gon in $\bbS^3$
    whose boundary lies in a polytope bounded by $4$ planes.
  \end{enumerate}
\end{theorem}

\begin{proof}
   We first show the following:
  For each $n\in\bbN_{\ge 2}$
  there exists a Fuchsian DPW potential $\xi$ for the Lawson surface $\xi_{1,n-1}$
  with $4$ poles at $(-1,\,0,\,1,\,\infty)$
  with respective vertex integers $(2,\,n,\,2,\,n)$ and the symmetry
  \begin{equation}
    \label{eq:lawson-2-n-2-n-symmetry}
    \tau^\ast \xi = \gauge{\xi}{\sigma}
    \spacecomma\quad
    \tau(z) = -1/z
    \spacecomma\quad
    \sigma = \diag(i,\,-i)
    \spaceperiod
  \end{equation}
(Recall that the vertex integers $n$ are linked to the eigenvalues by \eqref{evangle} and $\theta=\tfrac{\pi}{n}$.
By abuse of notation, we therefore call $n$ the vertex integer of the corresponding singularity of the reflection potential.)

  Step 0.
  By \cite{Heller_Heller_2023},
  the Lawson surface $\xi_{1,n-1}$
  has a $4$-pole Fuchsian DPW potential $\xi$
  with poles at $(1,\,i,\,-1,\,-i)$,
  spin $-1$ at each pole, and respective vertex integers $(n,\,n,\,n,\,n)$
  satisfying the symmetry
  \begin{equation}
    \tau^\ast \xi \coloneq  \gauge{\xi}{\sigma}
    \spacecomma\quad \tau(z) \coloneq-z
    \spacecomma\quad
    \sigma = \diag(i,\,-i)
    \spaceperiod
  \end{equation}
  \begin{equation*}
    \arraycolsep=8pt
    \begin{array}{lll}
      \fontsize{6}{7}
  \def\svgwidth{0.15\textwidth}%
\begingroup%
  \makeatletter%
  \providecommand\color[2][]{%
    \errmessage{(Inkscape) Color is used for the text in Inkscape, but the package 'color.sty' is not loaded}%
    \renewcommand\color[2][]{}%
  }%
  \providecommand\transparent[1]{%
    \errmessage{(Inkscape) Transparency is used (non-zero) for the text in Inkscape, but the package 'transparent.sty' is not loaded}%
    \renewcommand\transparent[1]{}%
  }%
  \providecommand\rotatebox[2]{#2}%
  \newcommand*\fsize{\dimexpr\f@size pt\relax}%
  \newcommand*\lineheight[1]{\fontsize{\fsize}{#1\fsize}\selectfont}%
  \ifx\svgwidth\undefined%
    \setlength{\unitlength}{66.6711721bp}%
    \ifx\svgscale\undefined%
      \relax%
    \else%
      \setlength{\unitlength}{\unitlength * \real{\svgscale}}%
    \fi%
  \else%
    \setlength{\unitlength}{\svgwidth}%
  \fi%
  \global\let\svgwidth\undefined%
  \global\let\svgscale\undefined%
  \makeatother%
  \begin{picture}(1,0.9475587)%
    \lineheight{1}%
    \setlength\tabcolsep{0pt}%
    \put(0.76970549,0.42696312){\makebox(0,0)[lt]{\lineheight{1.25}\smash{\begin{tabular}[t]{l}$n$\end{tabular}}}}%
    \put(0.37950195,0.84510063){\makebox(0,0)[lt]{\lineheight{1.25}\smash{\begin{tabular}[t]{l}$n$\end{tabular}}}}%
    \put(-0.00836943,0.42802362){\makebox(0,0)[lt]{\lineheight{1.25}\smash{\begin{tabular}[t]{l}$n$\end{tabular}}}}%
    \put(0.39022508,0.00782947){\makebox(0,0)[lt]{\lineheight{1.25}\smash{\begin{tabular}[t]{l}$n$\end{tabular}}}}%
    \put(0,0){\includegraphics[width=\unitlength,page=1]{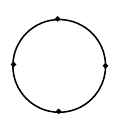}}%
  \end{picture}%
\endgroup%

      &
      \fontsize{6}{7}
  \def\svgwidth{0.25\textwidth}%
\begingroup%
  \makeatletter%
  \providecommand\color[2][]{%
    \errmessage{(Inkscape) Color is used for the text in Inkscape, but the package 'color.sty' is not loaded}%
    \renewcommand\color[2][]{}%
  }%
  \providecommand\transparent[1]{%
    \errmessage{(Inkscape) Transparency is used (non-zero) for the text in Inkscape, but the package 'transparent.sty' is not loaded}%
    \renewcommand\transparent[1]{}%
  }%
  \providecommand\rotatebox[2]{#2}%
  \newcommand*\fsize{\dimexpr\f@size pt\relax}%
  \newcommand*\lineheight[1]{\fontsize{\fsize}{#1\fsize}\selectfont}%
  \ifx\svgwidth\undefined%
    \setlength{\unitlength}{131.04259773bp}%
    \ifx\svgscale\undefined%
      \relax%
    \else%
      \setlength{\unitlength}{\unitlength * \real{\svgscale}}%
    \fi%
  \else%
    \setlength{\unitlength}{\svgwidth}%
  \fi%
  \global\let\svgwidth\undefined%
  \global\let\svgscale\undefined%
  \makeatother%
  \begin{picture}(1,0.19954791)%
    \lineheight{1}%
    \setlength\tabcolsep{0pt}%
    \put(0.75976417,0.14741983){\makebox(0,0)[lt]{\lineheight{1.25}\smash{\begin{tabular}[t]{l}$2$\end{tabular}}}}%
    \put(0.05451945,0.14741983){\makebox(0,0)[lt]{\lineheight{1.25}\smash{\begin{tabular}[t]{l}$n$\end{tabular}}}}%
    \put(0.52714316,0.14741983){\makebox(0,0)[lt]{\lineheight{1.25}\smash{\begin{tabular}[t]{l}$n$\end{tabular}}}}%
    \put(0.3031292,0.14741983){\makebox(0,0)[lt]{\lineheight{1.25}\smash{\begin{tabular}[t]{l}$2$\end{tabular}}}}%
    \put(-0.00425816,0.01976514){\makebox(0,0)[lt]{\lineheight{1.25}\smash{\begin{tabular}[t]{l}$-1$\end{tabular}}}}%
    \put(0.30957192,0.01581829){\makebox(0,0)[lt]{\lineheight{1.25}\smash{\begin{tabular}[t]{l}$0$\end{tabular}}}}%
    \put(0.52903514,0.01402137){\makebox(0,0)[lt]{\lineheight{1.25}\smash{\begin{tabular}[t]{l}$1$\end{tabular}}}}%
    \put(0.75611748,0.01648319){\makebox(0,0)[lt]{\lineheight{1.25}\smash{\begin{tabular}[t]{l}$\infty$\end{tabular}}}}%
    \put(0,0){\includegraphics[width=\unitlength,page=1]{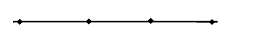}}%
  \end{picture}%
\endgroup%

      &
      \fontsize{6}{7}
  \def\svgwidth{0.25\textwidth}%
\begingroup%
  \makeatletter%
  \providecommand\color[2][]{%
    \errmessage{(Inkscape) Color is used for the text in Inkscape, but the package 'color.sty' is not loaded}%
    \renewcommand\color[2][]{}%
  }%
  \providecommand\transparent[1]{%
    \errmessage{(Inkscape) Transparency is used (non-zero) for the text in Inkscape, but the package 'transparent.sty' is not loaded}%
    \renewcommand\transparent[1]{}%
  }%
  \providecommand\rotatebox[2]{#2}%
  \newcommand*\fsize{\dimexpr\f@size pt\relax}%
  \newcommand*\lineheight[1]{\fontsize{\fsize}{#1\fsize}\selectfont}%
  \ifx\svgwidth\undefined%
    \setlength{\unitlength}{130.01355644bp}%
    \ifx\svgscale\undefined%
      \relax%
    \else%
      \setlength{\unitlength}{\unitlength * \real{\svgscale}}%
    \fi%
  \else%
    \setlength{\unitlength}{\svgwidth}%
  \fi%
  \global\let\svgwidth\undefined%
  \global\let\svgscale\undefined%
  \makeatother%
  \begin{picture}(1,0.22420183)%
    \lineheight{1}%
    \setlength\tabcolsep{0pt}%
    \put(0.75786273,0.17166115){\makebox(0,0)[lt]{\lineheight{1.25}\smash{\begin{tabular}[t]{l}$n$\end{tabular}}}}%
    \put(0.04703608,0.17166115){\makebox(0,0)[lt]{\lineheight{1.25}\smash{\begin{tabular}[t]{l}$2$\end{tabular}}}}%
    \put(0.52340056,0.17166115){\makebox(0,0)[lt]{\lineheight{1.25}\smash{\begin{tabular}[t]{l}$2$\end{tabular}}}}%
    \put(0.29761355,0.17166115){\makebox(0,0)[lt]{\lineheight{1.25}\smash{\begin{tabular}[t]{l}$n$\end{tabular}}}}%
    \put(-0.00429186,0.01992157){\makebox(0,0)[lt]{\lineheight{1.25}\smash{\begin{tabular}[t]{l}$-1$\end{tabular}}}}%
    \put(0.30410726,0.01594349){\makebox(0,0)[lt]{\lineheight{1.25}\smash{\begin{tabular}[t]{l}$0$\end{tabular}}}}%
    \put(0.52530751,0.01413235){\makebox(0,0)[lt]{\lineheight{1.25}\smash{\begin{tabular}[t]{l}$1$\end{tabular}}}}%
    \put(0.75418718,0.01661365){\makebox(0,0)[lt]{\lineheight{1.25}\smash{\begin{tabular}[t]{l}$\infty$\end{tabular}}}}%
    \put(0,0){\includegraphics[width=\unitlength,page=1]{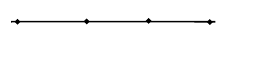}}%
  \end{picture}%
\endgroup%

      \\
      \fontsize{6}{7}
      \text{step 0} &
      \fontsize{6}{7}
      \text{step 1} &
      \fontsize{6}{7}
      \text{step 2}
    \end{array}
  \end{equation*}

Step 1.  Push down by $z\mapsto z^2$ via Lemma \ref{lem:push-down}.
  The resulting potential has
  four simple poles at
  $(-1,\,0,\,1,\,\infty)$
  with respective vertex  integers $(n,\,2,\,n,\,2)$.

 Step 2. Apply the coordinate change taking
  $(-1,\,0,\,1,\,\infty)\to(0,\,1,\,\infty,\,-1)$.
  The resulting potential has
  four simple poles at
  $(-1,\,0,\,1,\,\infty)$
  with respective vertex integers $(2,\,n,\,2,\,n)$.

Step 3. Apply the symmetrizing gauge (Lemma \ref{lem:symmetrizing-gauge})
  so that the resulting potential has the same poles
  and vertex  integers, and satisfies
  the symmetry~\eqref{eq:lawson-2-n-2-n-symmetry}.

  To show the existence of this symmetrizing gauge,
  note that the residues of the potential obtained in step 2 are
  \begin{equation}
    \fontsize{7}{6}
    A_{-1} =\begin{bmatrix}\fourth & -a_0 + i a_1\\ 0 & -\fourth\end{bmatrix}
    \quad
    A_{0} = \begin{bmatrix}b_0 & a_0\\c_0 & -b_0\end{bmatrix}
      \quad
      A_{1} = \begin{bmatrix}-\fourth & 0\\ -c_0 - i c_1 & \fourth\end{bmatrix}
        \quad
        A_{\infty} = \begin{bmatrix}b_1 & -i a_1\\c_1 & -b_1\end{bmatrix}
          .
  \end{equation}
    Since the Hopf differential for a Lawson surface $\xi_{1,n}$, $n>0$, is not $0$ ,
  the real quantities $a_0^{(-1)}$ and $a_1^{(-1)}$ are not both zero,
  and the real quantities $c_0^{(-1)}$ and $c_1^{(-1)}$ are not both zero.
  Hence ${(-a_0 + i a_1)}^{(-1)}\ne 0$ and $({-c_0 - i c_1})^{(-1)}\ne 0$.
  Hence the kernels of $A_{-1}^{(-1)}$ and $A_{1}^{(-1)}$ are
  respectively $\transpose{(1,\,0)}$ and  $\transpose{(0,\,1)}$,
  which are independent.
  Hence the symmetrizing gauge exists by Lemma \ref{lem:symmetrizing-gauge}.

  \textbf{Lawson potentials  with a fundamental polygon bounded by three planes.}
    We claim that for each $n\in\bbN_{\ge 2}$,
  there exists a reflection potential
  for the Lawson surface $\xi_{1,\,n-1}$
  with $n+2$ poles on $\bbS^1$, each with positive residue eigenvalue $1/4$.
  The unit disk is the domain for a minimal $(n+2)$-gon in $\bbS^3$
  whose boundary lies in a polytope bounded by $3$ planes.

  Start with the potential constructed in step 3 above.

  Step 4. Pull back by $z\mapsto z^n$.
  With $\alpha\coloneq e^{i \pi/n}$,
  the resulting potential
  has $2n$ simple poles
  at the $(2n)$th roots of unity $\{\alpha^k\suchthat k\in\{0,\dots,2n-1\}\}$
  each with vertex integer $2$,
  and poles at $0$ and $\infty$ with vertex integer $1$.

  \begin{equation*}
    \arraycolsep=8pt
    \begin{array}{lll}
      \fontsize{6}{7}
  \def\svgwidth{0.25\textwidth}%
\begingroup%
  \makeatletter%
  \providecommand\color[2][]{%
    \errmessage{(Inkscape) Color is used for the text in Inkscape, but the package 'color.sty' is not loaded}%
    \renewcommand\color[2][]{}%
  }%
  \providecommand\transparent[1]{%
    \errmessage{(Inkscape) Transparency is used (non-zero) for the text in Inkscape, but the package 'transparent.sty' is not loaded}%
    \renewcommand\transparent[1]{}%
  }%
  \providecommand\rotatebox[2]{#2}%
  \newcommand*\fsize{\dimexpr\f@size pt\relax}%
  \newcommand*\lineheight[1]{\fontsize{\fsize}{#1\fsize}\selectfont}%
  \ifx\svgwidth\undefined%
    \setlength{\unitlength}{111.58496376bp}%
    \ifx\svgscale\undefined%
      \relax%
    \else%
      \setlength{\unitlength}{\unitlength * \real{\svgscale}}%
    \fi%
  \else%
    \setlength{\unitlength}{\svgwidth}%
  \fi%
  \global\let\svgwidth\undefined%
  \global\let\svgscale\undefined%
  \makeatother%
  \begin{picture}(1,0.51964799)%
    \lineheight{1}%
    \setlength\tabcolsep{0pt}%
    \put(0.48692044,0.22845147){\makebox(0,0)[lt]{\lineheight{1.25}\smash{\begin{tabular}[t]{l}$2$\end{tabular}}}}%
    \put(0.71303151,0.28017294){\makebox(0,0)[lt]{\lineheight{1.25}\smash{\begin{tabular}[t]{l}$1$\end{tabular}}}}%
    \put(0.71359046,0.15859444){\makebox(0,0)[lt]{\lineheight{1.25}\smash{\begin{tabular}[t]{l}$\infty$\end{tabular}}}}%
    \put(0.42687961,0.05796486){\makebox(0,0)[lt]{\lineheight{1.25}\smash{\begin{tabular}[t]{l}$2$\end{tabular}}}}%
    \put(0.14291311,0.00467805){\makebox(0,0)[lt]{\lineheight{1.25}\smash{\begin{tabular}[t]{l}$2$\end{tabular}}}}%
    \put(-0.00500067,0.20438841){\makebox(0,0)[lt]{\lineheight{1.25}\smash{\begin{tabular}[t]{l}$2$\end{tabular}}}}%
    \put(0.41396427,0.42310214){\makebox(0,0)[lt]{\lineheight{1.25}\smash{\begin{tabular}[t]{l}$2$\end{tabular}}}}%
    \put(0.15077955,0.45843007){\makebox(0,0)[lt]{\lineheight{1.25}\smash{\begin{tabular}[t]{l}$2$\end{tabular}}}}%
    \put(0.26205981,0.27636899){\makebox(0,0)[lt]{\lineheight{1.25}\smash{\begin{tabular}[t]{l}$1$\end{tabular}}}}%
    \put(0.26183149,0.16560959){\makebox(0,0)[lt]{\lineheight{1.25}\smash{\begin{tabular}[t]{l}$0$\end{tabular}}}}%
    \put(0,0){\includegraphics[width=\unitlength,page=1]{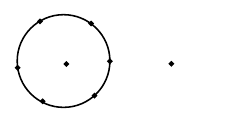}}%
  \end{picture}%
\endgroup%

      &
      \fontsize{6}{7}
  \def\svgwidth{0.25\textwidth}%
\begingroup%
  \makeatletter%
  \providecommand\color[2][]{%
    \errmessage{(Inkscape) Color is used for the text in Inkscape, but the package 'color.sty' is not loaded}%
    \renewcommand\color[2][]{}%
  }%
  \providecommand\transparent[1]{%
    \errmessage{(Inkscape) Transparency is used (non-zero) for the text in Inkscape, but the package 'transparent.sty' is not loaded}%
    \renewcommand\transparent[1]{}%
  }%
  \providecommand\rotatebox[2]{#2}%
  \newcommand*\fsize{\dimexpr\f@size pt\relax}%
  \newcommand*\lineheight[1]{\fontsize{\fsize}{#1\fsize}\selectfont}%
  \ifx\svgwidth\undefined%
    \setlength{\unitlength}{166.87667981bp}%
    \ifx\svgscale\undefined%
      \relax%
    \else%
      \setlength{\unitlength}{\unitlength * \real{\svgscale}}%
    \fi%
  \else%
    \setlength{\unitlength}{\svgwidth}%
  \fi%
  \global\let\svgwidth\undefined%
  \global\let\svgscale\undefined%
  \makeatother%
  \begin{picture}(1,0.39818956)%
    \lineheight{1}%
    \setlength\tabcolsep{0pt}%
    \put(0.80039526,0.20824141){\makebox(0,0)[lt]{\lineheight{1.25}\smash{\begin{tabular}[t]{l}$2$\end{tabular}}}}%
    \put(0.69192741,0.20891339){\makebox(0,0)[lt]{\lineheight{1.25}\smash{\begin{tabular}[t]{l}$2$\end{tabular}}}}%
    \put(0.57424863,0.21188025){\makebox(0,0)[lt]{\lineheight{1.25}\smash{\begin{tabular}[t]{l}$2$\end{tabular}}}}%
    \put(0.28616342,0.21718497){\makebox(0,0)[lt]{\lineheight{1.25}\smash{\begin{tabular}[t]{l}$2$\end{tabular}}}}%
    \put(0.49007898,0.35725514){\makebox(0,0)[lt]{\lineheight{1.25}\smash{\begin{tabular}[t]{l}$1$\end{tabular}}}}%
    \put(0.48881114,0.00312806){\makebox(0,0)[lt]{\lineheight{1.25}\smash{\begin{tabular}[t]{l}$1$\end{tabular}}}}%
    \put(0.17769557,0.21785695){\makebox(0,0)[lt]{\lineheight{1.25}\smash{\begin{tabular}[t]{l}$2$\end{tabular}}}}%
    \put(0.06001682,0.22082382){\makebox(0,0)[lt]{\lineheight{1.25}\smash{\begin{tabular}[t]{l}$2$\end{tabular}}}}%
    \put(0,0){\includegraphics[width=\unitlength,page=1]{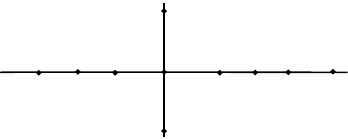}}%
  \end{picture}%
\endgroup%

      &
      \fontsize{6}{7}
  \def\svgwidth{0.25\textwidth}%
\begingroup%
  \makeatletter%
  \providecommand\color[2][]{%
    \errmessage{(Inkscape) Color is used for the text in Inkscape, but the package 'color.sty' is not loaded}%
    \renewcommand\color[2][]{}%
  }%
  \providecommand\transparent[1]{%
    \errmessage{(Inkscape) Transparency is used (non-zero) for the text in Inkscape, but the package 'transparent.sty' is not loaded}%
    \renewcommand\transparent[1]{}%
  }%
  \providecommand\rotatebox[2]{#2}%
  \newcommand*\fsize{\dimexpr\f@size pt\relax}%
  \newcommand*\lineheight[1]{\fontsize{\fsize}{#1\fsize}\selectfont}%
  \ifx\svgwidth\undefined%
    \setlength{\unitlength}{167.65525718bp}%
    \ifx\svgscale\undefined%
      \relax%
    \else%
      \setlength{\unitlength}{\unitlength * \real{\svgscale}}%
    \fi%
  \else%
    \setlength{\unitlength}{\svgwidth}%
  \fi%
  \global\let\svgwidth\undefined%
  \global\let\svgscale\undefined%
  \makeatother%
  \begin{picture}(1,0.13716164)%
    \lineheight{1}%
    \setlength\tabcolsep{0pt}%
    \put(0.63146969,0.08871843){\makebox(0,0)[lt]{\lineheight{1.25}\smash{\begin{tabular}[t]{l}$2$\end{tabular}}}}%
    \put(0.52350556,0.08938728){\makebox(0,0)[lt]{\lineheight{1.25}\smash{\begin{tabular}[t]{l}$2$\end{tabular}}}}%
    \put(0.40637327,0.09234038){\makebox(0,0)[lt]{\lineheight{1.25}\smash{\begin{tabular}[t]{l}$2$\end{tabular}}}}%
    \put(0.19349032,0.09210194){\makebox(0,0)[lt]{\lineheight{1.25}\smash{\begin{tabular}[t]{l}$2$\end{tabular}}}}%
    \put(0.00345522,0.09641732){\makebox(0,0)[lt]{\lineheight{1.25}\smash{\begin{tabular}[t]{l}$1$\end{tabular}}}}%
    \put(0.81709919,0.09041113){\makebox(0,0)[lt]{\lineheight{1.25}\smash{\begin{tabular}[t]{l}$2$\end{tabular}}}}%
    \put(0.1882293,0.01112904){\makebox(0,0)[lt]{\lineheight{1.25}\smash{\begin{tabular}[t]{l}$0$\end{tabular}}}}%
    \put(-0.00332826,0.01044458){\makebox(0,0)[lt]{\lineheight{1.25}\smash{\begin{tabular}[t]{l}$-1$\end{tabular}}}}%
    \put(0.8093767,0.01288358){\makebox(0,0)[lt]{\lineheight{1.25}\smash{\begin{tabular}[t]{l}$\infty$\end{tabular}}}}%
    \put(0,0){\includegraphics[width=\unitlength,page=1]{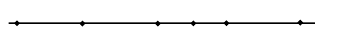}}%
  \end{picture}%
\endgroup%

      \\
      \fontsize{6}{7}
               {\text{step 4}\atop\text{$2n$ poles on $\bbS^1$}} &
               \fontsize{6}{7}
   {\text{step 5}\atop\text{$2n$ poles on $\bbR^1$}} &
    \fontsize{6}{7}
   {\text{step 6}\atop\text{$n+3$ poles on $\bbR^1$}}
    \end{array}
  \end{equation*}

    Step 5.  Apply the coordinate change
  taking $\bbS^1$ to $\bbR$,
  $\alpha^{1/2}\to 0$ and $-\alpha^{1/2}\to\infty$.
  The resulting potential
  has $2n$ simple poles on $\bbR$
  each with vertex integer $2$,
  and poles at $i$ and $-i$ with vertex integer $1$.
  This potential has the symmetry
  \begin{equation}
    \tau^\ast\xi = \gauge{\xi}{\sigma}
    \spacecomma\quad
    \sigma(z) \coloneq -z
    \spacecomma\quad
    \sigma \coloneq \diag(i,\,-i)
    \spaceperiod
  \end{equation}

    Step 6. Push down by $z\mapsto z^2$.
  The resulting potential has $n+2$ simple poles on $\bbR^{>0}$,
  each with vertex integer $2$,
  a simple pole $\infty$ with integer $2$
  and a simple pole at $-1$ with vertex integer $1$.

  \begin{equation*}
    \arraycolsep=8pt
    \begin{array}{lll}
      \fontsize{6}{7}
  \def\svgwidth{0.15\textwidth}%
\begingroup%
  \makeatletter%
  \providecommand\color[2][]{%
    \errmessage{(Inkscape) Color is used for the text in Inkscape, but the package 'color.sty' is not loaded}%
    \renewcommand\color[2][]{}%
  }%
  \providecommand\transparent[1]{%
    \errmessage{(Inkscape) Transparency is used (non-zero) for the text in Inkscape, but the package 'transparent.sty' is not loaded}%
    \renewcommand\transparent[1]{}%
  }%
  \providecommand\rotatebox[2]{#2}%
  \newcommand*\fsize{\dimexpr\f@size pt\relax}%
  \newcommand*\lineheight[1]{\fontsize{\fsize}{#1\fsize}\selectfont}%
  \ifx\svgwidth\undefined%
    \setlength{\unitlength}{135.53962227bp}%
    \ifx\svgscale\undefined%
      \relax%
    \else%
      \setlength{\unitlength}{\unitlength * \real{\svgscale}}%
    \fi%
  \else%
    \setlength{\unitlength}{\svgwidth}%
  \fi%
  \global\let\svgwidth\undefined%
  \global\let\svgscale\undefined%
  \makeatother%
  \begin{picture}(1,0.16461866)%
    \lineheight{1}%
    \setlength\tabcolsep{0pt}%
    \put(0.54414773,0.10973995){\makebox(0,0)[lt]{\lineheight{1.25}\smash{\begin{tabular}[t]{l}$2$\end{tabular}}}}%
    \put(0.41060189,0.11056729){\makebox(0,0)[lt]{\lineheight{1.25}\smash{\begin{tabular}[t]{l}$2$\end{tabular}}}}%
    \put(0.26571553,0.11422011){\makebox(0,0)[lt]{\lineheight{1.25}\smash{\begin{tabular}[t]{l}$2$\end{tabular}}}}%
    \put(0.00239071,0.11392518){\makebox(0,0)[lt]{\lineheight{1.25}\smash{\begin{tabular}[t]{l}$2$\end{tabular}}}}%
    \put(0.77376149,0.11183373){\makebox(0,0)[lt]{\lineheight{1.25}\smash{\begin{tabular}[t]{l}$2$\end{tabular}}}}%
    \put(-0.00411688,0.01376603){\makebox(0,0)[lt]{\lineheight{1.25}\smash{\begin{tabular}[t]{l}$0$\end{tabular}}}}%
    \put(0.76420918,0.0159363){\makebox(0,0)[lt]{\lineheight{1.25}\smash{\begin{tabular}[t]{l}$\infty$\end{tabular}}}}%
    \put(0,0){\includegraphics[width=\unitlength,page=1]{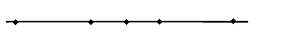}}%
  \end{picture}%
\endgroup%

      \\
      \fontsize{6}{7}
               {\text{step 7}\atop\text{$n+2$ poles on $\bbS^1$}}
    \end{array}
  \end{equation*}

 Step 7. Remove the pole at $-1$ by a flip gauge
  as in Lemma \ref{lem:flip-gauge}.
  The resulting potential has $n+2$ poles on $\bbR^1$.

  To check the existence of the gauge,
  since all residues have spin $-1$,
  let $k_r$ be the kernels of the $-1$ coefficients of the residues.
  Since the Lawson surface does not have zero Hopf differential,
  then the $k_r$ are not all dependent.
  Hence the kernel of the $-1$ coefficient of the residue
  at $-1$ is independent of the kernel of the $-1$ coefficient
  of the residue at another pole.
  Hence the gauge exists by Lemma \ref{lem:flip-gauge}.

    Step 8. Apply a coordinate change mapping $\bbR\to\bbS^1$.
  The resulting potential has
  $n+2$ simple poles on $\bbS^1$, each with vertex integer $2$.

  Step 9. Apply a $z$-independent gauge
  as in Lemma \ref{lem:refpotpro}, so that the potential has
  the real symmetry~\eqref{eq:potential-reality-condition}.

   \textbf{Lawson potentials with a fundamental polygon bounded by four planes.}
  We now prove the following:
  For each even $n\in\bbN_{\ge 2}$,
  there exists a reflection potential for the Lawson surface $\xi_{1,\,n-1}$ 
  with $n/2+3$ poles on $\bbS^1$, each with positive residue eigenvalue $1/4$.
  The unit disk is the domain for a minimal $(n/2+3)$-gon in $\bbS^3$
  whose boundary lies in a polytope bounded by $4$ planes.

  Start with the potential constructed in step 3 above.
  
 Step $4^\prime$. Pull back by $z\mapsto z^{n/2}$.
  The resulting potential has $n+2$ simple poles:
  $n$ on $\bbS^1$ at the $n$th roots of unity
  $\{\beta^k\suchthat k\in\{1,\dots,n-1\}\}$,
  $\beta\coloneq e^{2\pi i/n}$,
  and poles at $0$ and $\infty$.
  Each pole has vertex integer $2$.

  \begin{equation*}
    \arraycolsep=8pt
    \begin{array}{lll}
      \fontsize{6}{7}
  \def\svgwidth{0.25\textwidth}%
\begingroup%
  \makeatletter%
  \providecommand\color[2][]{%
    \errmessage{(Inkscape) Color is used for the text in Inkscape, but the package 'color.sty' is not loaded}%
    \renewcommand\color[2][]{}%
  }%
  \providecommand\transparent[1]{%
    \errmessage{(Inkscape) Transparency is used (non-zero) for the text in Inkscape, but the package 'transparent.sty' is not loaded}%
    \renewcommand\transparent[1]{}%
  }%
  \providecommand\rotatebox[2]{#2}%
  \newcommand*\fsize{\dimexpr\f@size pt\relax}%
  \newcommand*\lineheight[1]{\fontsize{\fsize}{#1\fsize}\selectfont}%
  \ifx\svgwidth\undefined%
    \setlength{\unitlength}{111.58496376bp}%
    \ifx\svgscale\undefined%
      \relax%
    \else%
      \setlength{\unitlength}{\unitlength * \real{\svgscale}}%
    \fi%
  \else%
    \setlength{\unitlength}{\svgwidth}%
  \fi%
  \global\let\svgwidth\undefined%
  \global\let\svgscale\undefined%
  \makeatother%
  \begin{picture}(1,0.51964799)%
    \lineheight{1}%
    \setlength\tabcolsep{0pt}%
    \put(0.48692044,0.22845147){\makebox(0,0)[lt]{\lineheight{1.25}\smash{\begin{tabular}[t]{l}$2$\end{tabular}}}}%
    \put(0.71303151,0.28017294){\makebox(0,0)[lt]{\lineheight{1.25}\smash{\begin{tabular}[t]{l}$2$\end{tabular}}}}%
    \put(0.71359046,0.15859444){\makebox(0,0)[lt]{\lineheight{1.25}\smash{\begin{tabular}[t]{l}$\infty$\end{tabular}}}}%
    \put(0.42687961,0.05796486){\makebox(0,0)[lt]{\lineheight{1.25}\smash{\begin{tabular}[t]{l}$2$\end{tabular}}}}%
    \put(0.14291311,0.00467805){\makebox(0,0)[lt]{\lineheight{1.25}\smash{\begin{tabular}[t]{l}$2$\end{tabular}}}}%
    \put(-0.00500067,0.20438841){\makebox(0,0)[lt]{\lineheight{1.25}\smash{\begin{tabular}[t]{l}$2$\end{tabular}}}}%
    \put(0.41396427,0.42310214){\makebox(0,0)[lt]{\lineheight{1.25}\smash{\begin{tabular}[t]{l}$2$\end{tabular}}}}%
    \put(0.15077955,0.45843007){\makebox(0,0)[lt]{\lineheight{1.25}\smash{\begin{tabular}[t]{l}$2$\end{tabular}}}}%
    \put(0.26205981,0.27636899){\makebox(0,0)[lt]{\lineheight{1.25}\smash{\begin{tabular}[t]{l}$2$\end{tabular}}}}%
    \put(0.26183149,0.16560959){\makebox(0,0)[lt]{\lineheight{1.25}\smash{\begin{tabular}[t]{l}$0$\end{tabular}}}}%
    \put(0,0){\includegraphics[width=\unitlength,page=1]{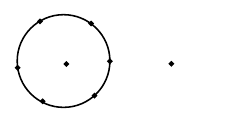}}%
  \end{picture}%
\endgroup%

      &
      \fontsize{6}{7}
  \def\svgwidth{0.25\textwidth}%
\begingroup%
  \makeatletter%
  \providecommand\color[2][]{%
    \errmessage{(Inkscape) Color is used for the text in Inkscape, but the package 'color.sty' is not loaded}%
    \renewcommand\color[2][]{}%
  }%
  \providecommand\transparent[1]{%
    \errmessage{(Inkscape) Transparency is used (non-zero) for the text in Inkscape, but the package 'transparent.sty' is not loaded}%
    \renewcommand\transparent[1]{}%
  }%
  \providecommand\rotatebox[2]{#2}%
  \newcommand*\fsize{\dimexpr\f@size pt\relax}%
  \newcommand*\lineheight[1]{\fontsize{\fsize}{#1\fsize}\selectfont}%
  \ifx\svgwidth\undefined%
    \setlength{\unitlength}{166.87667981bp}%
    \ifx\svgscale\undefined%
      \relax%
    \else%
      \setlength{\unitlength}{\unitlength * \real{\svgscale}}%
    \fi%
  \else%
    \setlength{\unitlength}{\svgwidth}%
  \fi%
  \global\let\svgwidth\undefined%
  \global\let\svgscale\undefined%
  \makeatother%
  \begin{picture}(1,0.39818956)%
    \lineheight{1}%
    \setlength\tabcolsep{0pt}%
    \put(0.80039526,0.20824141){\makebox(0,0)[lt]{\lineheight{1.25}\smash{\begin{tabular}[t]{l}$2$\end{tabular}}}}%
    \put(0.69192741,0.20891339){\makebox(0,0)[lt]{\lineheight{1.25}\smash{\begin{tabular}[t]{l}$2$\end{tabular}}}}%
    \put(0.57424863,0.21188025){\makebox(0,0)[lt]{\lineheight{1.25}\smash{\begin{tabular}[t]{l}$2$\end{tabular}}}}%
    \put(0.28616342,0.21718497){\makebox(0,0)[lt]{\lineheight{1.25}\smash{\begin{tabular}[t]{l}$2$\end{tabular}}}}%
    \put(0.49007898,0.35725514){\makebox(0,0)[lt]{\lineheight{1.25}\smash{\begin{tabular}[t]{l}$2$\end{tabular}}}}%
    \put(0.48881114,0.00312806){\makebox(0,0)[lt]{\lineheight{1.25}\smash{\begin{tabular}[t]{l}$2$\end{tabular}}}}%
    \put(0.17769557,0.21785695){\makebox(0,0)[lt]{\lineheight{1.25}\smash{\begin{tabular}[t]{l}$2$\end{tabular}}}}%
    \put(0.06001682,0.22082382){\makebox(0,0)[lt]{\lineheight{1.25}\smash{\begin{tabular}[t]{l}$2$\end{tabular}}}}%
    \put(0,0){\includegraphics[width=\unitlength,page=1]{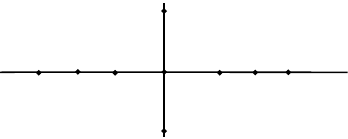}}%
  \end{picture}%
\endgroup%

      &
      \fontsize{6}{7}
  \def\svgwidth{0.25\textwidth}%
\begingroup%
  \makeatletter%
  \providecommand\color[2][]{%
    \errmessage{(Inkscape) Color is used for the text in Inkscape, but the package 'color.sty' is not loaded}%
    \renewcommand\color[2][]{}%
  }%
  \providecommand\transparent[1]{%
    \errmessage{(Inkscape) Transparency is used (non-zero) for the text in Inkscape, but the package 'transparent.sty' is not loaded}%
    \renewcommand\transparent[1]{}%
  }%
  \providecommand\rotatebox[2]{#2}%
  \newcommand*\fsize{\dimexpr\f@size pt\relax}%
  \newcommand*\lineheight[1]{\fontsize{\fsize}{#1\fsize}\selectfont}%
  \ifx\svgwidth\undefined%
    \setlength{\unitlength}{167.65525718bp}%
    \ifx\svgscale\undefined%
      \relax%
    \else%
      \setlength{\unitlength}{\unitlength * \real{\svgscale}}%
    \fi%
  \else%
    \setlength{\unitlength}{\svgwidth}%
  \fi%
  \global\let\svgwidth\undefined%
  \global\let\svgscale\undefined%
  \makeatother%
  \begin{picture}(1,0.13716164)%
    \lineheight{1}%
    \setlength\tabcolsep{0pt}%
    \put(0.63146969,0.08871843){\makebox(0,0)[lt]{\lineheight{1.25}\smash{\begin{tabular}[t]{l}$2$\end{tabular}}}}%
    \put(0.52350556,0.08938728){\makebox(0,0)[lt]{\lineheight{1.25}\smash{\begin{tabular}[t]{l}$2$\end{tabular}}}}%
    \put(0.40637327,0.09234038){\makebox(0,0)[lt]{\lineheight{1.25}\smash{\begin{tabular}[t]{l}$2$\end{tabular}}}}%
    \put(0.19349032,0.09210194){\makebox(0,0)[lt]{\lineheight{1.25}\smash{\begin{tabular}[t]{l}$2$\end{tabular}}}}%
    \put(0.00345522,0.09641732){\makebox(0,0)[lt]{\lineheight{1.25}\smash{\begin{tabular}[t]{l}$2$\end{tabular}}}}%
    \put(0.81709919,0.09041113){\makebox(0,0)[lt]{\lineheight{1.25}\smash{\begin{tabular}[t]{l}$2$\end{tabular}}}}%
    \put(0.1882293,0.01112904){\makebox(0,0)[lt]{\lineheight{1.25}\smash{\begin{tabular}[t]{l}$0$\end{tabular}}}}%
    \put(-0.00332826,0.01044458){\makebox(0,0)[lt]{\lineheight{1.25}\smash{\begin{tabular}[t]{l}$-1$\end{tabular}}}}%
    \put(0.8093767,0.01288358){\makebox(0,0)[lt]{\lineheight{1.25}\smash{\begin{tabular}[t]{l}$\infty$\end{tabular}}}}%
    \put(0,0){\includegraphics[width=\unitlength,page=1]{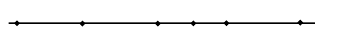}}%
  \end{picture}%
\endgroup%

      \\
      \fontsize{6}{7}
               {\text{step $4^\prime$}\atop\text{$n$ poles on $\bbS^1$}} &
               \fontsize{6}{7}
                        {\text{step $5^\prime$}\atop\text{$n$ poles on $\bbR^1$}} &
                        \fontsize{6}{7}
                                 {\text{step $6^\prime$}\atop\text{$n/2+3$ poles on $\bbR^1$}}
    \end{array}
  \end{equation*}

  Step $5^\prime$. Apply the coordinate change taking $\bbS^1\to\bbR$,
  $\alpha^{1/2}\to 1$ and $-\alpha^{1/2}\to  i$ and $- i$.
  The resulting potential has
  $2n+2$ poles: $n$ poles on $\bbR$, and two poles at $\pm  i$,
  each with vertex integer $2$.
  This potential has the symmetry
  \begin{equation}
    \tau^\ast\xi = \gauge{\xi}{\sigma}
    \spacecomma\quad
    \sigma(z) \coloneq -z
    \spacecomma\quad
    \sigma \coloneq \diag( i,\,- i)
    \spaceperiod
  \end{equation}

  Step $6^\prime$.
  Push down by $z\mapsto z^2$.
  The resulting potential has $n/2+3$ simple poles on $\bbR$,
  each with vertex integer $2$.

  Step $7^\prime$. Apply a coordinate change taking $\bbR\to\bbS^1$.
  The resulting potential has $n/2+3$ simple poles on $\bbS^1$,
  each with vertex integer $2$.

  Step $8^\prime$. Apply a $z$-independent gauge
  as in Lemma \ref{lem:refpotpro} so that the potential has
  the symmetry~\eqref{eq:potential-reality-condition}.
\end{proof}

\typeout{== figure/high-genus.tex ============================================}\begin{figure}[b]
  \includegraphics[width=0.325\textwidth]{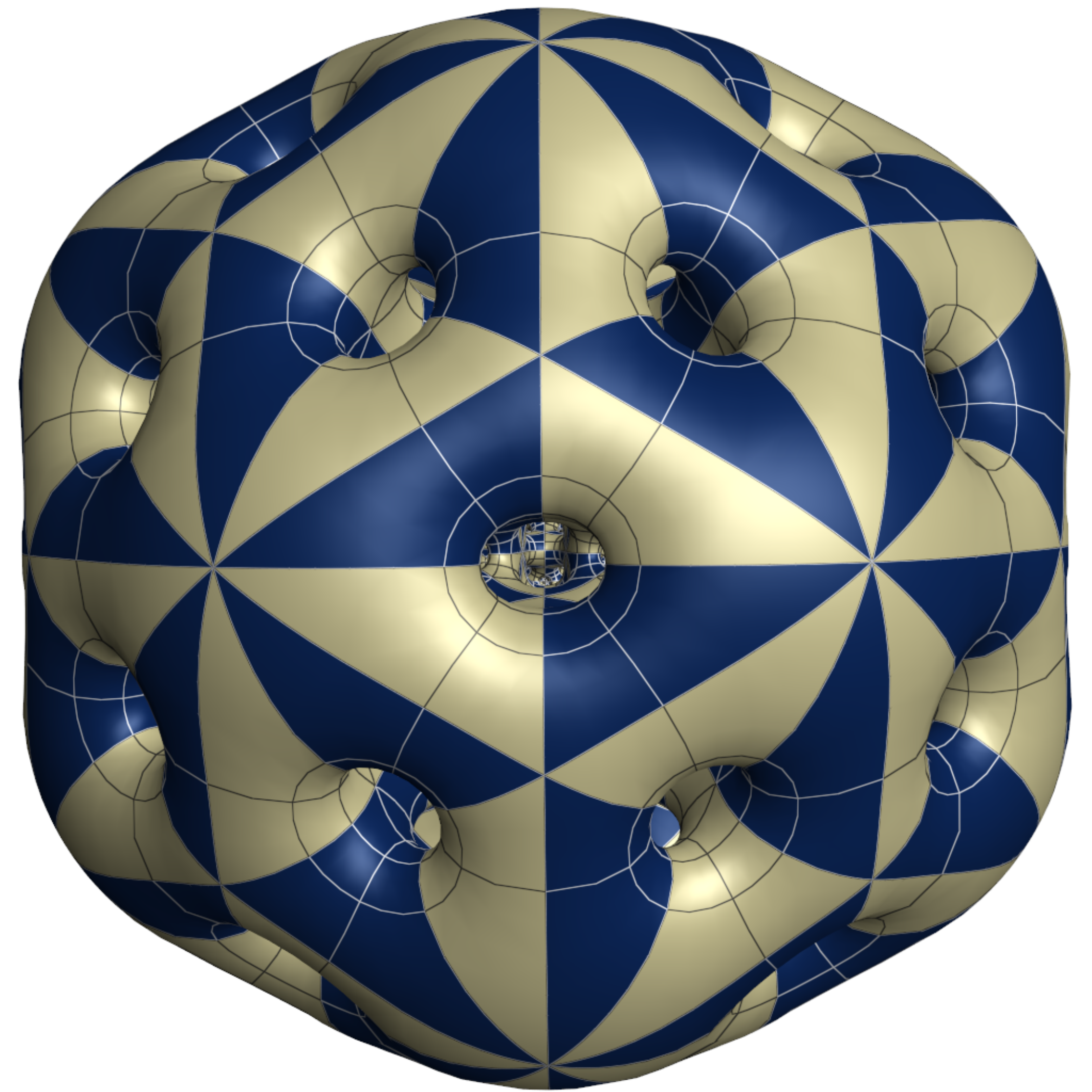}
  \includegraphics[width=0.325\textwidth]{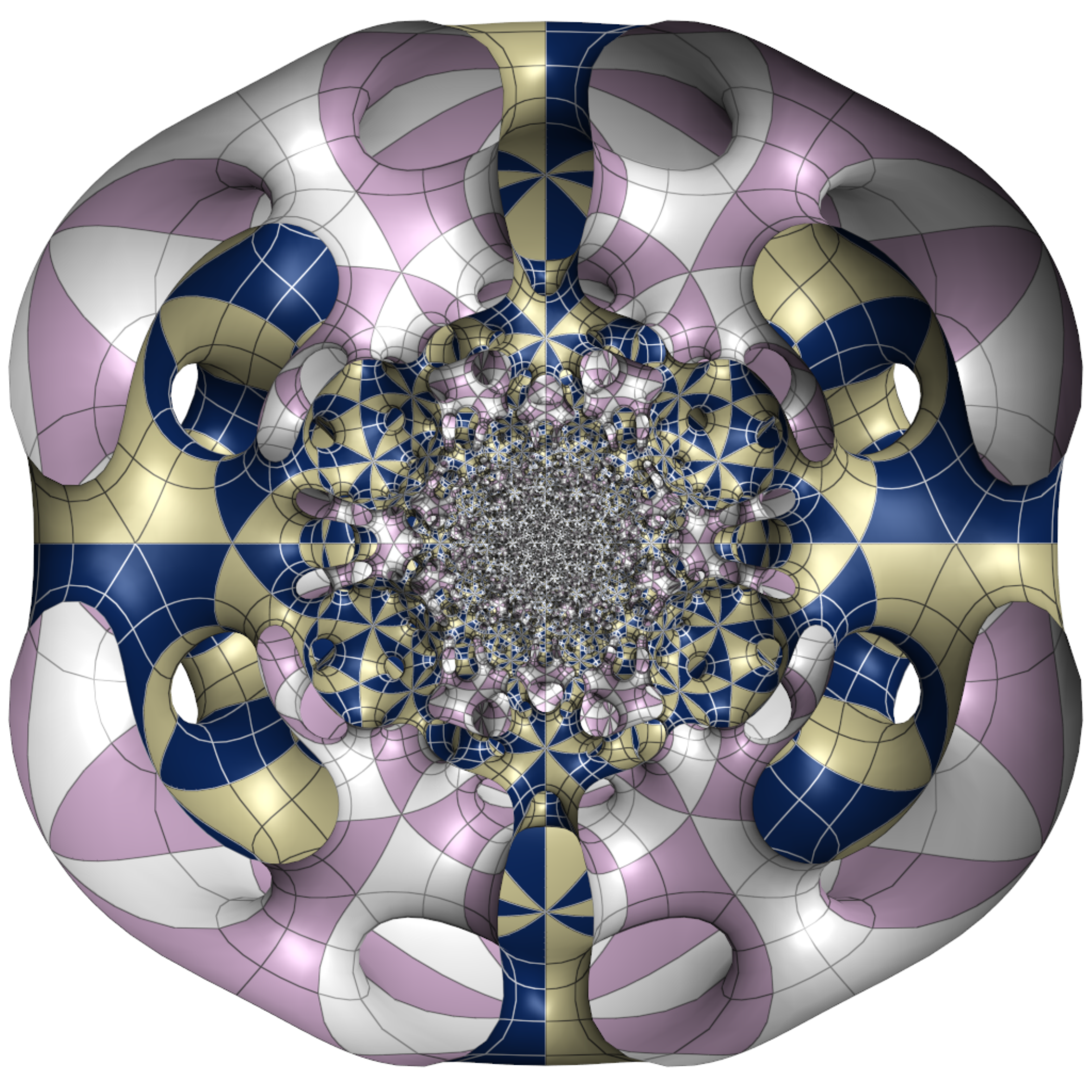}
  \includegraphics[width=0.325\textwidth]{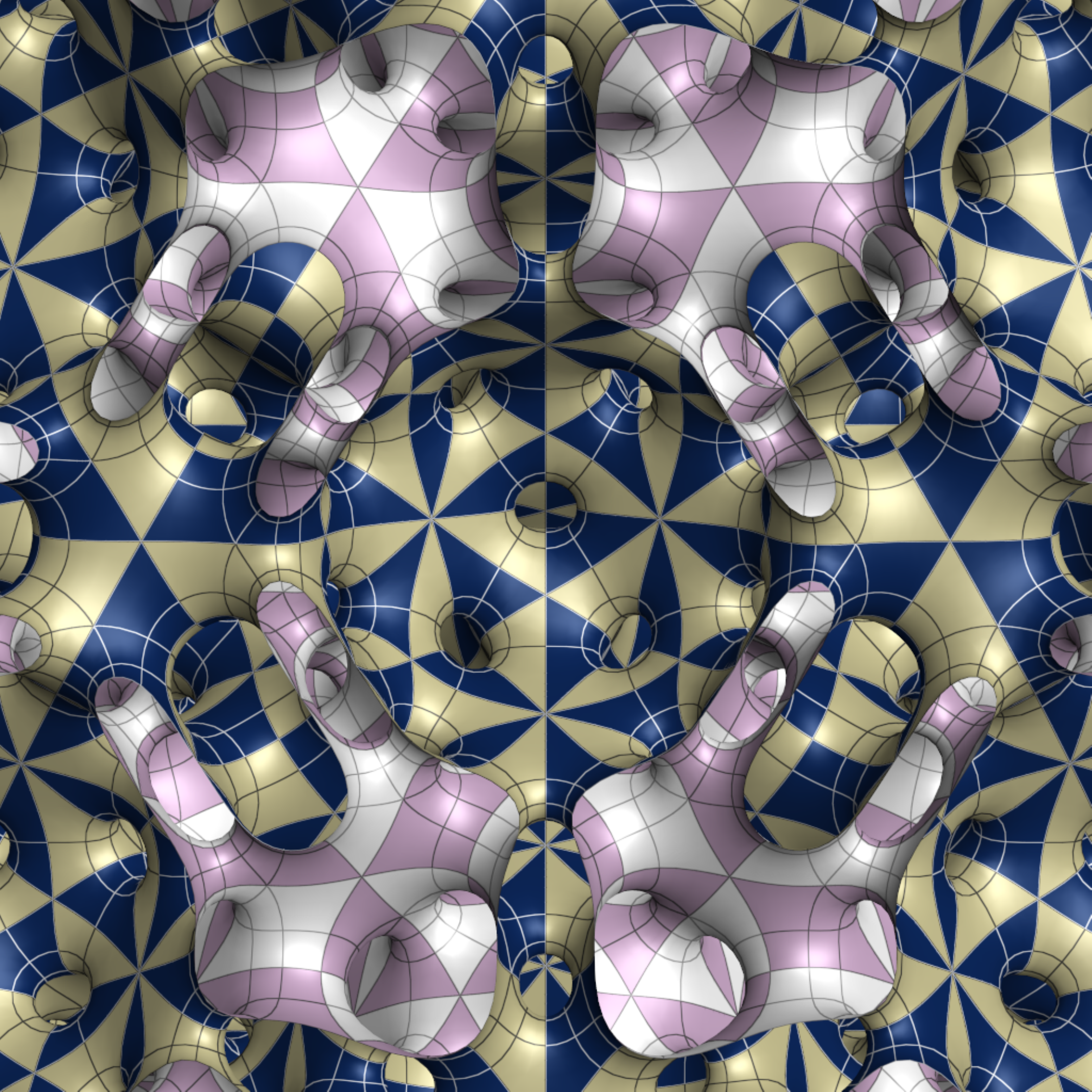}
  \caption{The non-dihedral reflection surface with the highest genus 5161.
    The surface has the symmetry of the 600-cell of order 14400.
    Left to right: the full surface,
    a cutaway showing half the surface,
    and a closeup of the center of the cutaway.}
\end{figure}

\section{Area}
\label{sec:area}

The flow of the DPW potential affords a numerical calculation of the area of the reflection
surfaces. For doing so, we apply Corollary 4.3 of~\cite{Heller_Heller_Traizet_2023}
for the situation at hand.
Before stating and proving the area formula for a minimal reflection surface in terms of the entries of its reflection potential, we first recall some notations and make some simple observations.

Let
\begin{equation}
  \xi=\sum_k R_k\frac{dz}{z-z_k}
\end{equation}
be a reflection  potential  for a minimal reflection surface $f\colon\Sigma\to\bbS^3$. As all poles of $\xi$
are  on the unit circle, the residue condition $\sum_kR_k=0$ must hold.
Let $G$ be the corresponding reflection group.  Note that its order $|G|$ is finite since $\Sigma$ is 
compact by definition of a reflection surface in $\mathbb S^3.$

Let $\Gamma\subset G$ be the subgroup of index 2 consisting of orientation preserving maps. Then $\Gamma$ acts on $\Sigma$ by
holomorphic automorphisms, and we obtain the following lemma.

\begin{lemma}
  There is a  holomorphic map $\pi\colon \Sigma\to\Sigma/\Gamma=\CP^1$ which branches exactly over the preimages of $z_1,\dots,z_p.$
  The preimage of the open unit disc consists of $|G|$ many copies of the fundamental polygon $P.$
\end{lemma}

Let $p\in\Sigma$ be a branch point, i.e., the preimage of a branch value $\pi(p)=z_k$. Then its branch order $b_p$ satisfies
\begin{equation}
  b_p=n_k-1
  \spacecomma
\end{equation}
where
$n_k$ the vertex integer of the point $z_k$.

For $k=1,\dots,p$ define
\begin{equation}
  m_k \coloneq \frac{|\Gamma|}{n_k}=\frac{|G|}{2n_k}
  \spaceperiod
\end{equation}
Obviously, $m_k\in\N^{>0}$ is  the number of preimages of $z_k$ of $\pi\colon\Sigma\to\CP^1.$

Recall that the eigenvalues of the residue $R_k$ of $\xi$ at $z_k$ are
\begin{equation}
  \begin{cases}
    \pm \frac{1}{2n_k} & \quad \text{if }\xi\text{ has spin } -1\text{ at }z_k\\
    \pm \frac{n_k-1}{2n_k} &\quad \text{if }\xi\text{ has spin } 1\text{ at }z_k
  \end{cases}\;,
\end{equation}
independently of $\lambda.$

If $z_k$ has spin $1$, then
\begin{equation}
  \res_{\lambda=0}R_k=0
\end{equation}
by section 2.4 in \cite{Bobenko_Heller_Schmitt_2021}.

If $z_k$ has spin $-1$,
\begin{equation}
  \res_{\lambda=0}R_k\neq 0
\end{equation}
as $f$ is an immersion (see  Section 2.4 in \cite{Bobenko_Heller_Schmitt_2021}).
Since the eigenvalues of $R_k$ are independent of $\lambda$, $R_k$ is nilpotent.
Consequently, there is a matrix $C_k\in\matSL{2}{\bbC}$ such that
\begin{equation}
  \label{eq:resnormform}
  C_k^{-1}R_kC_k=\tfrac{1}{2n_k}\begin{bmatrix}a&b/\lambda\\c \lambda&-a\end{bmatrix}
\end{equation}
for
holomorphic functions $a,b,c\colon \bbD_{1+\epsilon}\to \C$. Here,  $\bbD_{1+\epsilon}$ is the disc of radius $1+\epsilon>1$ centered at $\lambda=0$ in the spectral plane, for which the reflection potential is defined. Then,
\begin{equation}
  \alpha_k:=\frac{a(0)}{2n_k}\in\bbC
\end{equation}
is independent of the choice of $C_k$ (amongst all $C_k$ satisfying \eqref{eq:resnormform} for some holomorphic functions $a,b,c$), and is called the
\emph{area defect} of $\xi$ at the spin -1 point $z_k$.

With these notations we have the following theorem.
\begin{theorem}
  Let $\xi=\sum_k R_k\frac{dz}{z-z_k}$ be a reflection potential for the minimal reflection surface $f\colon\Sigma\to\bbS^3$ with finite
  order reflection group $G.$ Let $n_k$ be vertex integer of $z_k$ and $m_k$ be the number of preimages of $z_k$ in $\Sigma.$
  Assume that $\xi$ has spin $-1$ at $z_1,\dots,z_r$,   and spin 1 at $z_{r+1},\dots,z_p$. Let $\alpha_k$ be the area defect at $z_k$ for $k=1,\dots,r.$
  Then, the area of $f$ is given by
  \begin{equation}
    A(f)=2\pi\sum_{k=1}^r(1-2n_k\alpha_k)m_k
    \spaceperiod
  \end{equation}
\end{theorem}

\begin{proof}
  To prove the area formula, we will make use of Corollary 4.3
  of~\cite{Heller_Heller_Traizet_2023}.
  The surface is obtained globally from the pull-back potential $\pi^*\xi$.
In order to apply~\cite{Heller_Heller_Traizet_2023} we  have to find at any singular point $p$ of $\pi^*\eta$ a local gauge transformation $G^p$ which depends holomorphic on $\lambda$
  such that
  \begin{equation}
    (d-\pi^*\xi).G_p
  \end{equation}
  is smooth at $p.$
  It turns out to be easier to work on a twofold covering $\hat\pi\colon\hat\Sigma\to\Sigma$ which branches exactly over $\pi^{-1}(\{z_1,\dots,z_p\})$.
  Such a covering exists.

  First, let $k\in\{1,\dots,r\},$ i.e., $\eta$ has spin -1 at $z_k$.
  Write
  $C_k^{-1}R_kC_k=\tfrac{1}{2n_k}\begin{smatrix}a&b/\lambda\\c \lambda&-a\end{smatrix}$
    for holomorphic functions $a,b,c$ defined on an open neighborhood of $\lambda=0$
    after a  {\em $\lambda$-independent} conjugation $C_k\in\matSL{2}{\bbC}$.
    Since the eigenvalues of $R_k$ at the spin -1 point $z_k$ are $\pm\tfrac{1}{2n_k}$, it holds
    \begin{equation}
      -a^2-bc=-1
      \spaceperiod
    \end{equation}

    Consider a local holomorphic coordinate $z$ centered at $z_k$. As $\pi$ is branched over $z_k$ with branch order $n_k-1$ and
    $\hat\pi$ is a two fold covering which singly branches over every $\pi^{-1}(z_k)$, there
    is local holomorphic coordinate $y$ on $\hat\Sigma$ centered at  $\hat p:=\hat\pi^{-1}(p)$
    satisfying
    \begin{equation}
      y=z^{2n_k}
      \spaceperiod
    \end{equation}

    Take the pull-back $\mu=(\hat\pi\circ\pi)^*\xi$ of the potential $\xi$ on $\hat\Sigma.$
    It expands at a preimage $y_k=(\pi\circ\hat\pi)^{-1}(z_k)$ as
    \begin{equation}
      \mu=\begin{bmatrix}a&b/\lambda\\c \lambda&-a\end{bmatrix}\frac{dy}{y}+y^2(\dots)
    \end{equation}
    since $n_k>1$ by assumption. Consider the local gauge
    \begin{equation}
      G^p = C_k
      \begin{bmatrix}\kappa&0\\\lambda&\kappa^{-1} \end{bmatrix}
      \begin{bmatrix}y^{-1}&0\\0&y \end{bmatrix}
    \end{equation}
    where  $\kappa$ is   defined by
    \begin{equation}\kappa=\frac{ b }{1-a}.\end{equation}
    Note that $\kappa$ is holomorphic and non-vanishing at $\lambda=0$ if $a(0)\neq1$,
    which we assume first.

    Then,
    \begin{equation}
      (d-\mu).G^p =
      \begin{bmatrix} 0&0\\\frac{b \lambda  \left(a^2+b c-1\right)}{(a-1)^2}&0\end{bmatrix}
        \frac{dy}{y^3}+y^0(\dots)
    \end{equation}
    is holomorphic at $y_k\in\hat\Sigma$ since $a^2+bc=1.$
    Expanding $G^p=G^p_0+G^p_1\lambda+\dots$ in terms of $\lambda$
    gives
    \begin{equation}
      G^p_0 =
      \begin{bmatrix}
        -\frac{b_0}{(a_0-1) z} & 0 \\
        0 & -\frac{(a_0-1) z}{b_0} \\
      \end{bmatrix}
    \end{equation}
    and
    \begin{equation}
      G^p_1 =
      \begin{bmatrix}
        \frac{-a_0 b_1+a_1 b_0+b_1}{(a_0-1)^2 z} & 0 \\
        \frac{1}{z} & -\frac{z (-a_0b_1+a_1 b_0+b_1)}{b_0^2} \\
      \end{bmatrix}
    \end{equation}
    where $a=a_0+a_1\lambda+\dots$, $b=b_0+b_1\lambda+\dots\;$. Note that
    \begin{equation}a(0)=a_0=2n_k\alpha_k.\end{equation}
    We obtain
    \begin{equation}
      \res_{\hat p}\;\text{tr}(\mu_{-1}G^k_1 (G^k_0)^{-1})=(1-a_0)=(1-2n_k\alpha_k)
      \spacecomma
    \end{equation}
    where $\mu=\mu_{-1}\lambda^{-1}+\mu_0+\dots\;$.

    If  $a(0)=1$, there is a positive (and constant in $y$) conjugator $H^1$ (there are several cases to consider depending on $a$ and $b$, but the result is always the same)
    such that
    \begin{equation}
      (H^1)^{-1}
      \begin{bmatrix}a&b/\lambda\\c \lambda&-a\end{bmatrix} H^1 =
        \begin{bmatrix}1&*\\0&-1\end{bmatrix}
          \spaceperiod
    \end{equation}
    In particular, the pole can be removed by gauging with $G^k:=H^1H^2$ where
    $H^2=\begin{smatrix}y^{-1}&0\\0&y\end{smatrix},$ and a direct
    computation shows that  we obtain
    \begin{equation}
      \res_{\hat p}\; \text{tr}(\mu_{-1}G^k_1 (G^k_0)^{-1})=0=(1-a_0)=(1-2n_k\alpha_k).
    \end{equation}

    At points $z_k$ where $\eta$ has spin 1, i.e., $k>r$,
    we know from Section 2.4 in \cite{Bobenko_Heller_Schmitt_2021} that $R_k$ has vanishing $\lambda^{-1}$-term.
    In particular, we find a positive conjugator $H^1$ such that
    \begin{equation}
      (H^1)^{-1}R_kH^1=\begin{bmatrix} n_k-1&0\\0&1-n_k\end{bmatrix}
      \spacecomma
    \end{equation}
    so that $G^k=H^1H^2$ with
    \begin{equation}
      H^2=\begin{bmatrix} y^{1-n_k}&0\\0&y^{n_k-1}\end{bmatrix}
    \end{equation}
    is a gauge which desingularizes $\mu$ at $\pi^{-1}(p).$
    Then, a short computation shows
    \begin{equation}
      \res_{\hat p}\;\text{tr}(\mu_{-1}G^k_1 (G^k_0)^{-1})=0.
    \end{equation}

    Note that there are exactly $|G| / (2n_k)=m_k$  many preimages  of $z_k$ on $\hat\Sigma.$
    The immersion $f\circ\hat\pi\colon\hat\Sigma\to\bbS^3$ has twice the area as $f$, and we therefore obtain from
    Corollary 4.3 of~\cite{Heller_Heller_Traizet_2023}
    \begin{equation}
      A(f)=2\pi\sum_{k=1}^r(1-2 n_k \alpha_k)m_k
    \end{equation}
    as claimed.
\end{proof}

Applying the area formula, we obtain
numerical values for the areas of experimentally constructed reflection surfaces. We list these values below for some of the surfaces. We like to emphasis that our numerical values
for the Lawson surfaces
 agree well with the values of other experiments
by very different methods, e.g., using Brakke's surface evolver \cite{Rob}.
Even more convincing, however, is the high agreement between our experimentally determined area of Lawson surfaces and the area values of  \cite{CHHT} with proven  error estimates: for $g\geq 5$ the error estimate in
\cite{CHHT} is $\leq10^{-5}.$ On the other hand, the values determined  in \cite{CHHT} and the numerical values from Table \eqref{area-tableAkl}
also differ at most by $\leq10^{-5}.$

Finally, it is shown  in \cite{HHT4} that the area of the Lawson surfaces $\xi_{k,k}$ have limiting behaviour
\[A(\xi_{k,k})\sim 8\pi(1-\tfrac{1}{\sqrt{2}})k+O(1),\] which was actually first predicted from looking at
Table \eqref{area-tableAkl}.
\\

\begin{statictable}
  \def\svgwidth{0.6\textwidth}%
\begingroup%
  \makeatletter%
  \providecommand\color[2][]{%
    \errmessage{(Inkscape) Color is used for the text in Inkscape, but the package 'color.sty' is not loaded}%
    \renewcommand\color[2][]{}%
  }%
  \providecommand\transparent[1]{%
    \errmessage{(Inkscape) Transparency is used (non-zero) for the text in Inkscape, but the package 'transparent.sty' is not loaded}%
    \renewcommand\transparent[1]{}%
  }%
  \providecommand\rotatebox[2]{#2}%
  \newcommand*\fsize{\dimexpr\f@size pt\relax}%
  \newcommand*\lineheight[1]{\fontsize{\fsize}{#1\fsize}\selectfont}%
  \ifx\svgwidth\undefined%
    \setlength{\unitlength}{460.8bp}%
    \ifx\svgscale\undefined%
      \relax%
    \else%
      \setlength{\unitlength}{\unitlength * \real{\svgscale}}%
    \fi%
  \else%
    \setlength{\unitlength}{\svgwidth}%
  \fi%
  \global\let\svgwidth\undefined%
  \global\let\svgscale\undefined%
  \makeatother%
  \begin{picture}(1,0.76116437)%
    \lineheight{1}%
    \setlength\tabcolsep{0pt}%
    \put(0,0){\includegraphics[width=\unitlength,page=1]{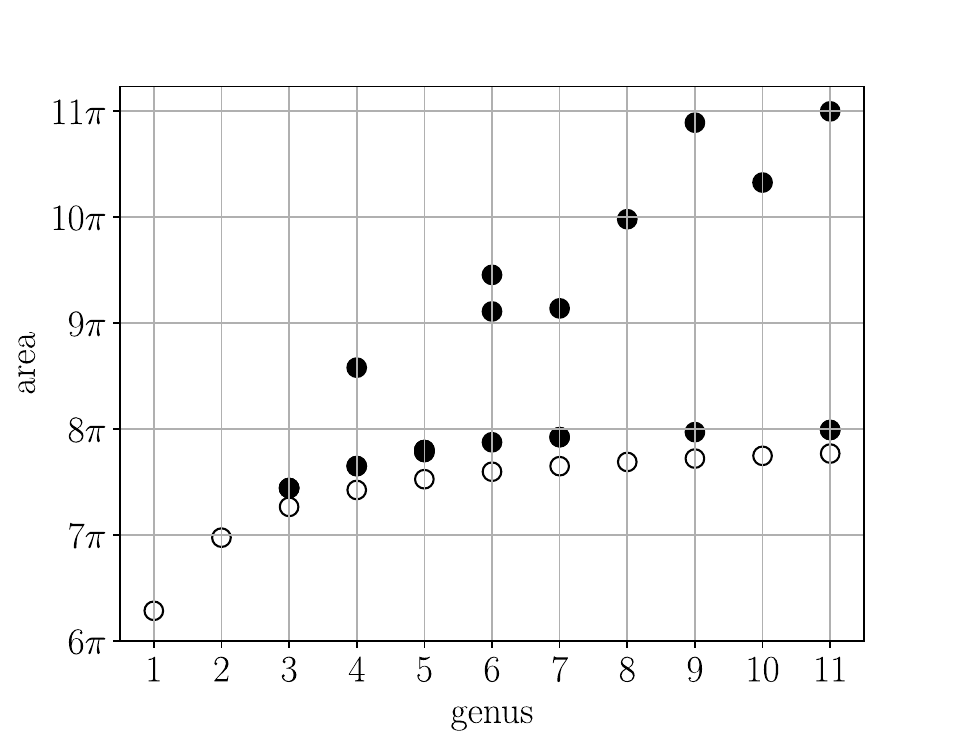}}%
  \end{picture}%
\endgroup%

  \captionof{table}{This graph shows the area  of embedded minimum surfaces of low genus examples of types $\Asurface{k}{\ell}$ and $\Bsurface{k}{\ell}$. The conjectured minimal areas are indicated by circles. }
\end{statictable}

Area of $\Asurface{k}{\ell}$:
\begin{statictable}
  $
  {
    \fontsize{6}{7}
    \begin{array}{l|llllllllllll}
      & 2 & 3 & 4 & 5 & 6 & 7 & 8 & 9 & 10 & 11 & 12 & 13\\
      \hline
      2 & 19.7392 &  21.9085 &  22.8203 &  23.3219 &  23.6413 &  23.8635 &  24.0273 &  24.1532 &  24.2532 &  24.3345 &  24.4020 &  24.4588 \\
      3 & &  26.9496 &  29.7001 &  31.3506 &  32.4381 &  33.2068 &  33.7788 &  34.2210 &  34.5731 &  34.8603 &  35.0990 &  35.3005\\
      4 & & &  34.2138 &  37.2281 &  39.3296 &  40.8618 &  42.0230 &  42.9314 &  43.6607 &  44.2586 &  44.7575 &  45.1802\\
      5 & & & &  41.5084 &  44.6718 &  47.0665 &  48.9264 &  50.4055 &  51.6066 &  52.5996 &  53.4335 &  54.1431\\
      6 & & & & &  48.8213 &  52.0800 &  54.6783 &  56.7842 &  58.5182 &  59.9666 &  61.1924 &  62.2420\\
      7 & & & & & &  56.1460 &  59.4707 &  62.2178 &  64.5135 &  66.4534 &  68.1098 &  69.5381\\
      8 & & & & & & &  63.4786 &  66.8516 &  69.7121 &  72.1579 &  74.2663 &  76.0984\\
      9 & & & & & & & &  70.8168 &  74.2267 &  77.1762 &  79.7434 &  81.9919\\
      10 & & & & & & & & &  78.1591 &  81.5980 &  84.6193 &  87.2865\\
      11 & & & & & & & & & &  85.5045 &  88.9669 &  92.0471\\
      12 & & & & & & & & & & &  92.8523 &  96.3341\\
      13 & & & & & & & & & & & &  100.202
    \end{array}
  }
  $
  \captionof{table}{Area of the $\Asurface{k}{\ell}$ surfaces.}\label{area-tableAkl}
\end{statictable}

Area of $\Bsurface{k}{\ell}$:
\begin{statictable}
  $
  {
    \fontsize{6}{7}
    \begin{array}{l|lllll}
      & 1 & 2 & 3 & 4 & 5\\
      \hline
      2 & 21.9085 & 22.8203 & 24.0331 & 24.5041 & 24.7364 \\
      3 & 23.3787 & 24.4459 & 28.7061 & & \\
      4 & 24.0262 & 24.8879 & & & \\
      5 & 24.4518 & 25.0380 & & &
    \end{array}
  }
  $
  \captionof{table}{Area of the $\Bsurface{k}{\ell}$ surfaces.}
\end{statictable}

\bibliographystyle{amsplain}
\bibliography{references}

\end{document}